\documentclass[12pt]{article}
\usepackage[utf8]{inputenc}
\usepackage{amsthm}
\usepackage{amsmath}
\usepackage{bbm}
\usepackage{amsfonts}
\usepackage{amssymb}
\usepackage[left=2cm,right=2cm,top=2cm,bottom=2cm]{geometry}

\usepackage{tabto}
\TabPositions{2 cm, 4 cm, 6 cm, 8 cm}

\usepackage{tkz-berge}
\usepackage[justification=centering]{caption}

\usepackage{tikz-cd}

\usepackage{tocloft}
\usepackage{appendix}

\usepackage{enumerate}

\usepackage{hyperref}
 
\newtheorem{theorem}{Theorem}[section]
\newtheorem{corollary}[theorem]{Corollary}
\newtheorem{lemma}[theorem]{Lemma}
\newtheorem{proposition}[theorem]{Proposition}

\theoremstyle{definition}
\newtheorem{definition}[theorem]{Definition}
\newtheorem{example}[theorem]{Example}
\newtheorem*{definition*}{Definition}

\newtheorem*{lemma*}{Lemma}
\newtheorem*{proposition*}{Proposition}
\newtheorem*{theorem*}{Theorem}
\newtheorem*{corollary*}{Corollary}

\theoremstyle{remark}
\newtheorem*{remark}{Remark}

\theoremstyle{remark}

\theoremstyle{definition}

\newtheorem{question}[theorem]{Question}

\begin{document}

\begin{titlepage}
\begin{center}

\vspace{1cm}
\Large{Francesco Fournier Facio}
\vfill
\huge{Infinite sums of Brooks quasimorphisms \\ and cup products in bounded cohomology} \\
\vspace{1cm}
\large{Thesis presented for the degree \\ Master of Science ETH in Mathematics} \\
\vspace{0.5cm}
\large{Supervised by Prof. Alessandra Iozzi \\ Department of Mathematics ETH Zurich}
\vfill
\today

\end{center}
\end{titlepage}

\pagebreak

\vspace*{2cm}

\begin{abstract}
We study a class of quasimorphisms of the free group that can be expressed as infinite sums of Brooks quasimorphisms with some nice properties. We then review Heuer's framework of decompositions developed in \cite{Heuer}, and put these quasimorphisms in that context. This allows us to prove triviality of a number of cup products in the bounded cohomology of the free group with trivial real coefficients. Finally, we introduce free products of decompositions and relate them to Rolli's quasimorphisms of free products. A technical result concerning these, which remains an open question, implies a strong triviality result for the cup product.
\end{abstract}

\vspace{2cm}

\renewcommand{\abstractname}{Acknowledgements}

\begin{abstract}
I will start by thanking my supervisor, Alessandra Iozzi, for having pushed me to learn about bounded cohomology early on, and for being so kind, generous and encouraging throughout the development of this project. Thank you to her and to Marc Burger for their hospitality at MSRI, where we could work on this and other problems.

I was very fortunate to have the opportunity to talk to many people during this semester: it is thanks to these conversations that many ideas in the paper came up. First with Roberto Frigerio, that suggested I work on this problem, and whose book \cite{Frigerio} was my introduction to bounded cohomology. Then with Nicolaus Heuer, whose article \cite{Heuer} is the basis for everything in this thesis (and the source of the pictures). Discussions with other people will hopefully lead to future work on related topics: in this context I would like to thank Alessandro Sisto, Merlin Incerti-Medici and Antonius Hase.

A special thank you goes to Tobias Hartnick for inviting me to Karlsruhe to talk about this problem and allowing me to give my first ever seminar talk: I am very grateful for his interest and his enthusiasm.

Last but not least, I need to thank Konstantin Golubev, for being a true mentor throughout my master's degree: supervising my semester paper a year ago, suggesting I go to my first workshop, learning about bounded cohomology with me, and especially listening to the many, many ideas that did not work this semester.
\end{abstract}

\vfill

\pagebreak

\tableofcontents
\restoregeometry

\pagebreak

\section{Introduction}

All relevant definitions will be given rigorously in the following sections.

\subsection{Motivation and context}

The theory of bounded cohomology was initiated by Gromov in his seminal paper \cite{grom}, in which he used it as a tool to control the minimal volume of a manifold. According to him, this theory had already been introduced by Trauber in an unpublished paper, in which he had shown that the bounded cohomology of amenable groups vanishes. It is now a very rich field of study, with applications not only to Riemannian geometry but also to rigidity theory, geometric group theory and dynamics. We refer the reader to \cite{Frigerio} for a book on the theory of bounded cohomology of discrete groups and its applications. A more general theory of continuous bounded cohomology for locally compact groups was then introduced by Burger and Monod in their paper \cite{BM}, which allowed them to settle Zimmer's program in the case of actions on the circle. We will however not work in such a general context in this thesis; the interested reader is referred to the paper and the book \cite{Monbook}. Since it is the focus of this thesis, when mentioning the bounded cohomology $H^\bullet_b(G)$ without further comments, we will mean bounded cohomology of the discrete group $G$ with trivial real coefficients. This is a real vector space equipped with a canonical quotient seminorm, called the \textit{Gromov seminorm}. \\

For any group $G$ one has $H^0_b(G) \cong \mathbb{R}$ and $H^1_b(G) = 0$. As we already mentioned, it has been known since the beginning of this theory that if $G$ is a discrete amenable group, then its bounded cohomology vanishes in all other degrees. We say that amenable groups have \textit{bounded cohomological dimension} equal to 0. In the paper \cite{Loeh}, the author shows that this holds also for the class of mitotic groups, which contains all algebraically closed groups, and the group $Homeo_K(\mathbb{R}^n)$ of self-homeomorphisms of $\mathbb{R}^n$ with compact support. Not all amenable groups are mitotic, so this is not a generalization of the previous case. Apart from these examples, there is \textit{no} group for which $H^\bullet_b(G)$ is known in all degrees. \\

It is natural to look at a non-abelian free group of finite rank $F$ as a first non-trivial example, since it is the prototipical example of a discrete non-amenable group. This group has moreover the advantage that the ordinary group cohomology vanishes in all degrees at least two, since it admits the one-dimensional cell complex $\vee_{i = 1}^{rank(F)} S^1$ as a classifying space. This allows to rephrase all bounded cohomology in terms of \textit{quasicocycles}. In particular, in degree two, the bounded cohomology is determined by the space of \textit{quasimorphisms}: maps $\varphi : F \to \mathbb{R}$ such that the quantity $|\varphi(g) + \varphi(h) - \varphi(gh)|$ is uniformly bounded. More explicitly, the map associating to a quasimorphism $\varphi$ the bounded cohomology class $[\delta^1 \varphi]$ where $\delta^1 \varphi(g, h) = \varphi(g) + \varphi(h) - \varphi(gh)$, is surjective. The kernel consists precisely of the space of \textit{trivial} quasimorphisms, i.e., those which can be expressed as the sum of a homomorphism and a bounded function. Two quasimorphisms that differ by a trivial quasimorphisms are called \textit{equivalent}. A special role is played by \textit{homogeneous} quasimorphisms: those which further satisfy $\varphi(g^n) = n \varphi(g)$ for all $g \in F$ and all $n \in \mathbb{Z}$. It turns out that any quasimorphism $\varphi$ is at bounded distance (so equivalent) to a unique homogeneous quasimorphisms, called its \textit{homogenization} and denoted by $\overline{\varphi}$. \\

In his paper \cite{Brooks}, Brooks introduced the first examples of quasimorphisms of the free group, which in modern terminology are called \textit{big Brooks quasimorphisms} (or \textit{big counting quasimorphisms}), which allowed to show that $H^2_b(F)$ is infinite-dimensional. Given a word $w$, the big Brooks quasimorphism, denoted by $H_w$, counts for a given $g \in F$ positively the occurrences of $w$ as a subword of $g$, and negatively those of $w^{-1}$. Changing slightly the definition one obtains the \textit{small Brooks quasimorphism} $h_w$, which counts disjoint occurrences. The advantage of $H_w$ is that its combinatorics is simpler, that of $h_w$ is that it behaves better algebraically. This motivated Faiziev to look at \textit{non-self-overlapping words} in \cite{Faziev}: a word $w$ is non-self-overlapping if any two occurrences of $w$ as a subword of any other word $g$ are disjoint. For such words $H_w = h_w$, and so one gets the best of both worlds. One can then consider the homogenization $\overline{h_w}$ of $h_w$. Notice that by definition, for any $g \in F$, there are only finitely many $w$ for which $H_w, h_w$ do not evaluate to 0 at $g$, since their length cannot exceed that of $g$. It follows easily from the construction of the homogenization that the same holds for $\overline{h_w}$. Therefore it makes sense to consider infinite sums of Brooks quasimorphisms, which are well-defined functions on the free group.

Building on the work of Faiziev, Grigorchuk proved the following result \cite{Grigor}:

\begin{theorem}[Grigorchuk]
\label{intro_gdt}

There exists a subset $\mathcal{F}$ of $F$, consisting of non-self-overlapping words, such that any quasimorphism is equivalent to a unique infinite sum of $\{ \overline{h_w} \}_{w \in \mathcal{F}}$.
\end{theorem}

The way this theorem is usually stated is that the span of the $\overline{h_w}$ is dense in $H^2_b(F)$ with respect to the \textit{pointwise topology}: a suitable topology of pointwise convergence. This is less good than it sounds, since this topology is too coarse for most practical applications; but it is still useful in that it allows to express any alternating quasimorphism as an infinite sum of Brooks quasimorphisms (a quasimorphism $\varphi \in Q(G)$ is \textit{alternating} if $\varphi(g^{-1}) = - \varphi(g)$ for any $g \in G$). Indeed, the proof of Theorem \ref{intro_gdt} can be adapted to show:

\begin{theorem}[Theorem \ref{gdt3}]
Let $F^+$ be a subset of $F$ which contains precisely one of $g, g^{-1}$ for any $1 \neq g \in F$. Then any alternating quasimorphism may be uniquely written as an infinite sum of $\{ H_w \}_{w \in F^+}$, or as an infinite sum of $\{ h_w \}_{w \in F^+}$.
\end{theorem}

Crucially, not \textit{all} infinite sums of Brooks quasimorphisms are quasimorphisms: some of these are real-valued functions $f$ such that the quantity $|f(g) + f(h) - f(gh)|$ is unbounded. In the paper \cite{Grigor}, after proving Theorem \ref{intro_gdt}, Grigorchuk discusses the question of which of these sums are quasimorphisms, showing that $\ell^1$ sums are, but that it is not a necessary condition. Indeed, for the family $I = \{a^n(ab)^nb : n \geq 1\}$, any infinite sum of $\{ \overline{h_w} \}_{w \in I}$ with bounded coefficients is a quasimorphism.

A much more general criterion than just $\ell^1$ sums is given by Calegari in the book \cite{Calegari}. A family of words $C \subset F$ is \textit{compatible} if there exist words $u, v$ such that the product $uv$ is reduced and any $w \in C$ can be written as $u'v'$, where $u' \neq 1$ is a suffix of $u$ and $v' \neq 1$ is a prefix of $v$. Then:

\begin{proposition}[Calegari]
\label{intro_Cal}

Let $(\alpha_w)_{w \in F}$ be coefficients such that there exists a finite $T \geq 0$ such that for any compatible family $C$:
$$\sum\limits_{w \in C} |\alpha_w| \leq T.$$
Then
$$\sum\limits_{w \in F} \alpha_w h_w$$
is a quasimorphism.
\end{proposition}

To our knowledge, so far there has been no example of a quasimorphism of the free group that is not equivalent to one satisfying the condition above. \\

Another example of a class of quasimorphisms of the free group was given by Rolli in \cite{Rolli}, and the construction generalizes to quasimorphisms of free products. These will be relevant in the last section of this thesis. \\

The \textit{second} bounded cohomology of the free group is especially important. Indeed, by a theorem of Huber \cite[Theorem 2.14]{huber}, given two groups $G, H$ and a surjective homomorphism $f : G \to H$, the induced map in bounded cohomology $H^2_b(f) : H^2_b(H) \to H^2_b(G)$ is an isometric embedding. It follows that the second bounded cohomology of a finitely generated group is a subspace of $H^2_b(F)$, for a free group of large enough rank.

Moving to higher degrees, also here the free group plays a special role: by the work of Frigerio, Pozzetti and Sisto \cite{FPS}, which builds on that of Hull and Osin \cite{HO}, it is possible to promote a non-trivial class in $H^n_b(F)$ to a non-trivial class in the bounded cohomology of an \textit{acylindrically hyperbolic group}, a very rich class of groups introduced by Osin \cite{ahg}, which contains non-elementary hyperbolic and relatively hyperbolic groups, the mapping class group of the $p$-punctured closed orientable surface of genus $g$ (provided that $3g + p \geq 6$), $Out(F)$, groups acting geometrically on a proper $CAT(0)$ space with a rank one isometry, and fundamental groups of several graphs of groups. It is now known that $H^3_b(F)$ is uncountably dimensional \cite{h3f}. Surprisingly, from here on nothing is known. In particular, it is not known whether $H^4_b(F)$ is zero or not. \\

Just as ordinary group cohomology, bounded cohomology is equipped with the \textit{cup product}: a graded-commutative map $\smile : H^n_b(G) \times H^m_b(G) \to H^{n+m}_b(G)$. In particular, we can look at the cup product $\smile : H^2_b(F) \times H^2_b(F) \to H^4_b(F)$. Since we have many examples of non-trivial classes in $H^2_b(F)$, one could hope to find the first examples of non-trivial classes in $H^4_b(F)$ this way. In the past two years, two different articles \cite{Monod, Heuer} showed that this approach fails, at least for Brooks and Rolli quasimorphisms:

\begin{theorem}[Bucher, Monod]
\label{intro_BM}
Let $\varphi, \psi$ be big Brooks quasimorphisms of the non-abelian free group of finite rank $F$. Then $[\delta^1 \varphi] \smile [\delta^1 \psi] = 0 \in H^4_b(F)$.
\end{theorem}

\begin{theorem}[Heuer]
\label{intro_Heuer}

Let $\varphi, \psi$ be quasimorphisms of the non-abelian free group of finite rank $F$, that are either Brooks quasimorphisms on a non-self-overlapping word or Rolli quasimorphisms. Then $[\delta^1 \varphi] \smile [\delta^1 \psi] = 0 \in H^4_b(F)$.
\end{theorem}

The two approaches are very different. The first theorem follows from the triviality of a cup product in the continuous bounded cohomology $H^\bullet_{cb}(Aut(T), E)$ of the automorphism group of a locally finite tree with some coefficient module $E$, and employs the techniques of continuous bounded cohomology of locally compact groups with Banach coefficients. 

The proof of the second theorem is more combinatorial. In the paper \cite{Heuer}, the author introduces the notion of \textit{decomposition} of the free group. These are, roughly speaking, ways of writing down elements of the free group in a non-standard alphabet, in such a way that the concatenation of the decompositions of two elements $g, h$ does not differ too much from the decomposition of the element $gh$ around the center of the tripod with endpoints $1, g, gh$. A decomposition $\Delta$ induces quasimorphisms that are called \textit{$\Delta$-decomposable}, and generalize the examples of Brooks quasimorphisms on non-self-overlapping words and Rolli quasimorphisms. There is another class of quasimorphisms that behave well with respect to $\Delta$, which are called \textit{$\Delta$-continuous}. The theorem then follows from a more general result proving vanishing of the cup product of a $\Delta$-decomposable quasimorphism with a $\Delta$-continuous one. \\

Given such a result, and Grigorchuk's density theorem, one would hope that the vanishing of the whole cup product is a direct corollary. However, the pointwise topology for which Grigorchuk's density theorem holds is a priori too coarse for the cup product to be continuous (although of course it is impossible to prove this without proving that $H^4_b(F) \neq 0$). The right topology to obtain continuity of the cup product is the topology induced by the Gromov seminorm on $H^n_b(F)$, $n = 2, 4$, which is equivalent to the \textit{defect topology} for $n = 2$. But even then, for this topology it is unknown if $H^4_b(F)$ is Hausdorff or not, so a simple application of density can only give cup products that lie in the closure of $0 \in H^4_b(F)$.

\subsection{Present work}

The ultimate goal would be to extend triviality of the cup product to the whole space $H^2_b(F)$. A promising approach is to try and prove triviality of the cup product for some infinite sums of Brooks quasimorphisms, since we know that these represent all bounded cohomology classes. However, since no criterion describing \textit{all} sums that give rise to quasimorphisms is available, it seems reasonable to restrict to studying a subspace for which we do have a criterion. So given coefficients $(\alpha_w)_{w \in F^+}$ that are non-zero \textit{only} on non-self-overlapping words, let us denote $\varphi_\alpha := \sum \alpha_w h_w$. We denote by $Cal$ the space of \textit{Calegari quasimorphisms}, i.e., of those $\varphi_\alpha$ that satisfy Calegari's condition in Proposition \ref{intro_Cal}. Notice that the condition implies that $\alpha$ is bounded. The restriction to non-self-overlapping words is motivated by the fact that Heuer's techniques (and Heuer's Theorem \ref{intro_Heuer}) do not apply to self-overlapping ones.

We will look at many different subspaces of $Cal$, among which are the finite sums of Brooks quasimorphisms $\Sigma_{Br}$ and the $\ell^1$ sums of Brooks quasimorphisms $\ell^1_{Br}$. The main novelty of this work is to put into Heuer's framework sums which were previously considered anomalies: for instance those supported on the family $\{a^n (ab)^n b : n \geq 1\}$ from Grigorchuk's paper. We say that a family $I$ is \textit{independent} if no two words in it overlap, and \textit{symmetric} if $I = I^{-1}$. Grigorchuk's example generalizes to sums supported on symmetric independent families; the space of finite sums of such quasimorphisms is denoted by $\Sigma_{Ind}$, and that of $\ell^1$ sums by $\ell^1_{Ind}$. There are other relevant subspaces whose definitions are more technical and will be spared in this introduction. The lattice of these subspaces is summarized in the following diagram,  where arrows indicate inclusion:

\begin{center}
	\begin{tikzcd}
		& & Cal \\
		& \ell^1_{Ind} \arrow[ru] & & \kappa(c_0) \arrow[lu] \\
		\Sigma_{Ind} \arrow[ru] & & \ell^1_{Br} \arrow[lu] \arrow[ru] & & \kappa(\ell^1) \arrow[lu] \\
		& & & w\ell^1_{Br} \arrow[lu] \arrow[ru] \\
		& & \Sigma_{Br} \arrow[ru] \arrow[lluu]
	\end{tikzcd}
\end{center}

We can now state part of our main theorem:

\begin{theorem}[Theorem \ref{th_main}]
Let $\varphi_\alpha, \varphi_\beta \in Cal$, and let $\phi$ be a Rolli quasimorphism.

\begin{enumerate}
\item If $\alpha$ and $\beta$ are supported on the same symmetric independent family, then $[\delta^1 \varphi_\alpha] \smile [\delta^1 \varphi_\beta] = 0$.

\item If $\varphi_\alpha \in \ell^1_{Ind}$ and $\varphi_\beta \in \kappa(\ell^1)$, then $[\delta^1 \varphi_\alpha] \smile [\delta^1 \varphi_\beta] = 0$.

\item If $\varphi_\beta \in \kappa(\ell^1)$, then $[\delta^1 \phi] \smile [\delta^1 \varphi_\beta] = 0$.
\end{enumerate}
\end{theorem}

As we have mentioned before, the natural topology on $H^4_b(F)$ (the one induced by the Gromov seminorm) is a priori not Hausdorff, so any density argument can only give cup products which have vanishing Gromov seminorm. Results of this kind are still of interest (analogously to Gromov's celebrated vanishing theorem \cite{grom}), and so we include them in our main theorem:

\begin{theorem}[Theorem \ref{th_main}]
\label{intro_main}
With the same notation,
\begin{enumerate}
\setcounter{enumi}{3}

\item If $\varphi_\alpha \in \ell^1_{Ind}$ and $\varphi_\beta \in \kappa(c_0)$, then $[\delta^1 \varphi_\alpha] \smile [\delta^1 \varphi_\beta]$ has vanishing Gromov seminorm.

\item If $\varphi_\beta \in \kappa(c_0)$, then $[\delta^1 \phi] \smile [\delta^1 \varphi_\beta]$ has vanishing Gromov seminorm.
\end{enumerate}
\end{theorem}

The main theorem will be proven in Subsection \ref{proof_main}, along with an estimate for the norm of the primitive of the vanishing cup products. The same techniques allow us to deal with big Brooks quasimorphisms on self-overlapping words as well:

\begin{theorem}[Corollary \ref{cupbig2}]
\label{intro_cupbig}
Let $\varphi_\alpha \in \ell^1_{Ind}$, let $\phi$ be a Rolli quasimorphism and let $w$ be any word. Then $[\delta^1 \varphi_\alpha] \smile [\delta^1 H_w] = [\delta^1 \phi] \smile [\delta^1 H_w] = 0$.
\end{theorem}

This theorem follows from a more general vanishing result treating some $\ell^1$ sums of big Brooks quasimorphisms (Theorem \ref{cupbig}). Also here we will give estimates for the norm of the primitive. \\

These vanishing results motivate the question: how big is the space $Cal$ with respect to the whole of $H^2_b(F)$? To this end we look at another space, denoted $\mathcal{N}_{Br}$, consisting of all $\varphi_\alpha$, where $\alpha$ is supported on non-self-overlapping words, that are quasimorphisms. We will present a different characterization than Calegari's for $\varphi_\alpha \in Cal$, as well as one for $\varphi_\alpha \in \mathcal{N}_{Br}$, which allow to compare these two spaces. The most interesting and relevant question is the following:

\begin{question}
Does $\mathcal{N}_{Br}$ represent the whole $H^2_b(F)$? That is: is every quasimorphism equivalent to some $\varphi_\alpha \in \mathcal{N}_{Br}$?
\end{question}

We are not able to answer this question, but make some progress towards it. Indeed, denote $\tilde{\varphi}_\alpha := \sum \alpha_w \overline{h_w}$. We know that every bounded cohomology class is represented by some quasimorphism $\tilde{\varphi}_\alpha$, by Grigorchuk's density theorem. Homogenization is a linear operation, but it is not clear that it works for infinite sums as well, so starting from the quasimorphism $\tilde{\varphi}_\alpha$ we cannot deduce directly that it is equivalent to $\varphi_\alpha$, so in particular that $\varphi_\alpha$ is also a quasimorphism. Still, we are able to prove the converse of this statement:

\begin{proposition}[Proposition \ref{twhom}]
\label{intro_twhom}

Let $\varphi_\alpha$, $\tilde{\varphi}_\alpha$ be as above, where $\alpha$ is supported on non-self-overlapping words. Suppose that $\varphi_\alpha \in \mathcal{N}_{Br}$, i.e., that it is a quasimorphism. Then $\tilde{\varphi}_\alpha$ is also a quasimorphism, and $\overline{\varphi}_\alpha = \tilde{\varphi}_\alpha$.
\end{proposition}

\begin{question}
Is the converse of Proposition \ref{intro_twhom} true?
\end{question}

In the last section, we will put Rolli's quasimorphisms of free products \cite{Rolli} in Heuer's framework, and see how the relevant constructions are compatible. In an attempt to use these results to prove the vanishing of more cup products, we look at the following map
$$\iota : F \to F*F : g \mapsto \iota_1(g) \iota_2(g),$$
where $\iota_i$ is the inclusion in the $i$-th copy of $F$. This is not a homomorphism, however it is a \textit{quasihomomorphism} in the sense of Hartnick and Schweitzer \cite{qout}, meaning that for any $\varphi \in Q(F*F)$ we have $\varphi \circ \iota \in Q(F)$. This gives the first example of a quasihomomorphism between free groups that is not quasi-Ulam, according to the definitions in \cite{qout}. We ask the following question about the behaviour of the map $\iota$ with respect to cup products:

\begin{question}
Suppose that $\varphi, \psi \in Q(F*F)$ satisfy $[\delta^1 \varphi] \smile [\delta^1 \psi] = 0 \in H^4_b(F*F)$. Is it true that $[\delta^1 (\varphi \circ \iota)] \smile [\delta^1 (\psi \circ \iota)] = 0 \in H^4_b(F)$?
\end{question}

We are unable to prove this seemingly intuitive statement, which is true for Rolli and big Brooks quasimorphisms. But its implications are very consequential:

\begin{theorem}[Theorem \ref{cupall}]
\label{intro_cupall}

Suppose that the previous question is answered affirmatively. Then if $\varphi$ is in $\Sigma_{Ind}$ or is a Rolli quasimorphism, and $\psi \in Q(F)$ is \underline{any} other quasimorphism, we have $[\delta^1 \varphi] \smile [\delta^1 \psi] = 0 \in H^4_b(F)$.
\end{theorem}

This result would be the first one giving triviality of the cup product for some class of non-trivial quasimorphisms with everything else, something much stronger than anything that has been proven so far. However, the general question asked by Heuer remains open:

\begin{question}
Is the cup product $\smile : H^2_b(F) \times H^2_b(F) \to H^4_b(F)$ trivial?
\end{question}

\subsection{Organisation}

In Section \ref{secbc} we will review the basics of bounded cohomology of discrete groups with trivial real coefficients, and define the cup product. We will spend most of our time on quasimorphisms, proving some basic results. We end the section by introducing the defect norm and proving that it is equivalent to the Gromov seminorm.

Section \ref{secf} is dedicated to the free group and its quasimorphisms. We begin by studying Rolli quasimorphisms, which allow us to prove that $H^2_b(F)$ is uncountably dimensional. After spending some time on combinatorics of words, we study Brooks quasimorphisms. The rest of the section is dedicated to infinite sums of Brooks quasimorphisms: the pointwise topology and Grigorchuk's density theorem, the space of Calegari quasimorphisms and its subspaces. The last subsection deals with the question of "largeness" of the spaces in question, and proves Proposition \ref{intro_twhom}.

In Section \ref{secdec} we review Heuer's framework of decompositions of the free group, as well as Heuer's theorem that the cup product of a decomposable and a continuous quasimorphism is trivial. We skip the proof of the theorem but give an explicit bound on the primitive of the relevant cup product, which was not present in Heuer's paper. We then introduce more examples of decompositions and of continuous quasimorphisms, which allow us to prove the main theorem (Theorem \ref{intro_main}), as well as the weaker vanishing result for big Brooks quasimorphisms (Theorem \ref{intro_cupbig}).

Section \ref{secprod} introduces the notion of free product of decompositions, and relates it to Rolli's constructions of quasimorphisms of free products. We introduce the map $\iota$, prove some of its properties, and end by proving Theorem \ref{intro_cupall}.

In Section \ref{secop} we review some open questions that were raised throughout the thesis.

Appendix \ref{og} deals with overlap graphs, and introduces a combinatorial approach to study families of infinite sums of Brooks quasimorphisms. As an application, we are able to generate more explicit examples of elements of $\Sigma_{Ind}$, and thus of more trivial cup products.

\subsection{Notation}
\label{not}

$G, H$ will always denote discrete groups, $F$ a non-abelian free group of finite rank with a fixed basis $S$, which by convention is not symmetric. Given any symmetric subset $W \subset F$, we denote by $W^+$ any subset of $W$ containing precisely one of $g, g^{-1}$ for any $1 \neq g \in F$. Whenever such sets appear in this thesis, the definitions using them will be independent of such a choice. \\

General functions will be denoted by $f$, quasimorphisms by $\varphi, \phi, \psi$. When writing down infinite sums of Brooks quasimorphisms, coefficients will be denoted by $(\alpha_w)_{w \in F}$ and $\alpha$ will be called a \textit{coefficient map}. \\

The following notation for sequences will be used in the section about decompositions, and is the same as the one in Heuer's paper \cite{Heuer}. Given a set $X$, we denote by $X^*$ the set of finite sequences of elements of $X$. Elements of $X^*$ will be denoted by $x = (x_1, \ldots, x_n)$. The length of such a sequence is $|x| = n$. Given a second sequence $y = (y_1, \ldots, y_m)$, we denote by $x \cdot y$ the concatenation $(x_1, \ldots, x_n, y_1, \ldots, y_m)$. If $X$ is a subset of some group (in our case, of $F$), we denote by $\underline{x}$ the product $x_1 \cdots x_n$, and by $x^{-1}$ the sequence $(x_n^{-1}, \ldots, x_1^{-1})$, so that $(\underline{x})^{-1} = \underline{x^{-1}}$. \\

$\mathbbm{1}$ will denote the indicator function of an event. For instance, given a subset $W \subset F$ and an element $g \in F$, $\mathbbm{1} \{ g \in W \} = 1$ if $g \in W$ and 0 otherwise. \\

We will use standard Landau asymptotic notation for estimates on the norms of the primitives of certain cup products. Given functions $f, g : \mathbb{R}_{\geq 0} \to \mathbb{R}_{\geq 0}$, we say that $f = \mathcal{O}(g)$ or $f << g$ if there exists a constant $C > 0$ such that $f(x) \leq Cg(x)$ for all large enough $x$. \\

When working with bounded cohomology, we will introduce the general concepts with the homogeneous resolution $C^\bullet_b(G)$, and then introduce the bar resolution $\overline{C}^\bullet_b(G)$. This distinction is important when first introducing the basic concepts, but since the main topic of this thesis is quasimorphisms, which are more natural to deal with in the inhomogeneous context, we will only work with the bar resolution. Therefore the notation is rigorous and distinct only for Subsection \ref{subsecdef}; from Subsection \ref{subsecqm} on, we will only work with the bar resolution and denote it $C^\bullet_b(G)$ for simplicity, without the bar on top. The reader will be reminded of this convention at the end of Subsection \ref{subsecdef}.

\pagebreak

\section{Bounded cohomology of discrete groups}
\label{secbc}

We start by introducing the framework of bounded cohomology of discete groups with trivial real coefficients. After giving the basic definitions, we will spend most of the time proving basic facts about quasimorphisms that will be useful in the sequel. Since the only focus of this thesis is bounded cohomology with trivial real coefficients, we will give the relevant definitions in this case only. For the more general context of discrete groups with coefficients, see \cite{Frigerio}. For the even more general setting of locally compact groups, see \cite{Monbook}.

\subsection{First definitions}
\label{subsecdef}

Let $G$ be a discrete group. Elements of $G^{n+1}$ will be denoted by $\overline{g} = (g_0, \ldots, g_n)$, and the diagonal action of $G$ on $G^{n+1}$ will be denoted by $g \overline{g} = (gg_0, \ldots, gg_n)$. For all $n \geq 0$, we define
$$C^n_b(G) := \{ f : G^{n+1} \to \mathbb{R} : ||f||_\infty < \infty \},$$
where $|| \cdot ||_\infty$ is the usual supremum norm. $C^n_b(G)$ is thus a Banach space with the norm $|| \cdot ||_\infty$. The diagonal action of $G$ on $G^{n+1}$ induces an action of $G$ on $C^n_b(G)$ by linear isometries, defined by $(g \cdot f)(\overline{g}) = f(g^{-1} \overline{g})$. We denote by $C^n_b(G)^G$ the subspace of $G$-invariant elements. Note that this is a closed subspace of $C^n_b(G)$, so it is a Banach space with the restriction of the norm $|| \cdot ||_\infty$. \\

We define the coboundary map $\delta^n : C^n_b(G) \to C^{n+1}_b(G)$ by
$$\delta^n(f)(g_0, \ldots, g_{n+1}) := \sum\limits_{i = 0}^{n+1} (-1)^i f(g_0, \ldots, \hat{g_i}, \ldots, g_{n+1}).$$
This is a bounded linear map of norm at most $(n+2)$. It is $G$-equivariant, meaning that $\delta(g \cdot f) = g \cdot \delta(f)$, so it restricts to a map $C^n_b(G)^G \to C^{n+1}_b(G)^G$, which we still denote by $\delta^n$. Moreover, $\delta^{n+1} \circ \delta^n = 0$. Thus, $(C^\bullet_b(G), \delta^\bullet)$ and $(C^\bullet_b(G)^G, \delta^\bullet)$ are cochain complexes of Banach spaces. This allows to consider the space of \textit{cocycles} $Z^n_b(G) := C^n_b(G)^G \cap \ker(\delta^n)$ and the space of \textit{coboundaries} $B^n_b(G) := \delta^{n-1}(C^{n-1}_b(G)^G) \subset Z^n_b(G)$. Finally, the quotient space $H^n_b(G) := Z^n_b(G) / B^n_b(G)$ is called the \textit{bounded cohomology of $G$} (with trivial real coefficients). We will denote by $[f]$ the class of a cocycle $f \in Z^n_b(G)$.

Given a coboundary $f \in B^n_b(G)$, any map $\pi \in C^{n-1}_b(G)$ such that $\delta^{n-1} \pi = f$ is called a \textit{(bounded) primitive of $f$}. \\

Since $\delta^n$ is a bounded linear map, the space $Z^n_b(G)$ is closed in $C^n_b(G)^G$, so it is a Banach space with the restriction of the norm $|| \cdot ||_\infty$. This induces a quotient seminorm on $H^n_b(G)$, which we call the \textit{Gromov seminorm}. This is a norm if and only if $B^n_b(G)$ is closed in $Z^n_b(G)$ (or equivalently in $C^n_b(G)^G$ or in $C^n_b(G)$).

\subsubsection{The bar resolution}

Since we are considering elements of $C^n_b(G)^G$, we have $n$ degrees of freedom instead of $(n+1)$ as in $C^n_b(G)$. So it suffices to look at the value of a given $f \in C^n_b(G)^G$ on a well-chosen subset of $G^{n+1}$ to obtain all the relevant information.

Specifically, we notice that every $f \in C^n_b(G)^G$ is completely determined by its value at elements of $G^{n+1}$ whose first coordinate is 1, and that any such element may be uniquely written as $(1, g_1, g_1 g_2, \ldots, g_1 \cdots g_n)$. In other words, if we let $\overline{C}^0_b(G) := \mathbb{R}$ and $\overline{C}^n_b(G) := C^{n-1}_b(G)$ for $n > 0$, still equipped with the norm $|| \cdot ||_\infty$, we have an isometric isomorphism defined by $B^0(f) = f(g)$ (elements of $C^0_b(G)^G$ are constant functions) and for $n > 0$:
$$B^n : C^n_b(G)^G \to \overline{C}^n_b(G) : f \mapsto ((g_1, \ldots, g_n) \mapsto f(1, g_1, g_1 g_2, \ldots, g_1 \cdots g_n)).$$
Under this isomorphism, the map $\delta^n : C^n_b(G)^G \to C^{n+1}_b(G)^G$ becomes $\overline{\delta}^n : \overline{C}^n_b(G) \to \overline{C}^{n+1}_b(G)$ (meaning that $B^{n+1} \circ \delta^n = \overline{\delta}^n \circ B^n$) defined by $\overline{\delta}^0 = 0$ and for $n > 0$:
$$\overline{\delta}^n(f)(g_1, \ldots, g_{n+1}) := f(g_2, \ldots, g_{n+1}) + \sum\limits_{i = 1}^n (-1)^i f(g_1, \ldots, g_i g_{i+1}, \ldots, g_{n+1}) + (-1)^{n+1} f(g_1, \ldots, g_n).$$
It follows that $\overline{\delta}^{n+1} \circ \overline{\delta}^n = 0$, so $(\overline{C}^\bullet_b(G), \overline{\delta}^\bullet)$ is a cochain complex of Banach spaces. By definition of $\overline{\delta}^\bullet$, the map $B^\bullet$ is a chain map $(C^\bullet_b(G)^G, \delta^\bullet) \to (\overline{C}^\bullet_b(G), \overline{\delta}^\bullet)$. Again we may consider the space of \textit{cocycles} $\overline{Z}^n_b(G) := \ker(\overline{\delta}^n)$ and the space of \textit{coboundaries} $\overline{B}^n_b(G) := \overline{\delta}^{n-1}(\overline{C}^{n-1}_b(G))$. Then the quotient space $\overline{Z}^n_b(G) / \overline{B}^n_b(G)$, equipped with the quotient seminorm, is isometrically isomorphic to $H^n_b(G)$ under the map induced by $B^n$ in cohomology.

\begin{example}[Degree 0]
By definition $\overline{C}^0_b(G) = \mathbb{R}$ and $\overline{B}^0_b(G) = 0$. Also by definition $\overline{\delta}^0 = 0$, so $\overline{Z}^0_b(G) = \overline{C}^0_b(G)$. Thus $H^0_b(G) = \mathbb{R}$ for any group $G$.
\end{example}

\begin{example}[Degree 1]
\label{deg1}
$\overline{Z}^1_b(G) = \{ f : G \to \mathbb{R} : ||f||_\infty < \infty \text{ and } f(g_2) - f(g_1 g_2) + f(g_1) = 0 \text{ for any } g_1, g_2 \in G \}$. In other words, $\overline{Z}^1_b(G)$ is the space of bounded homomorphisms from $G$ to $\mathbb{R}$. But if $f$ is a non-zero homomorphism from $G$ to $\mathbb{R}$, then given $g \in G$ such that $f(g) > 0$, for $n \geq 1$: $f(g^n) = nf(g) \xrightarrow{n \to \infty} \infty$. Thus there are no non-zero bounded homomorphisms, and so $\overline{Z}^1_b(G) = 0$. Therefore $H^1_b(G) = 0$ for any group $G$.
\end{example}

\subsubsection{Ordinary group cohomology and the comparison map}

If in all of the above we remove the requirements of boundedness, we obtain the cochain complexes $(C^\bullet(G), \delta^\bullet)$, $(C^\bullet(G)^G, \delta^\bullet)$, $(\overline{C}^\bullet(G), \overline{\delta}^\bullet)$, which are cochain complexes of (not normed) real vector spaces. We may as well consider the spaces of cocycles $Z^n(G), \overline{Z}^n(G)$, of coboundaries $B^n(G), \overline{B}^n(G)$, and the quotient space $H^n(G) := Z^n(G)/B^n(G) \cong \overline{Z}^n(G) / \overline{B}^n(G)$, which is called the \textit{cohomology of $G$} (with trivial real coefficients). \\

Group cohomology is a classical object of study. Its computation is made easier by the following deep result \cite[Corollary 4.8]{Frigerio}:

\begin{theorem}
Let $X$ be a path-connected topological space admitting a universal cover $\tilde{X}$, and suppose that $H_i(\tilde{X}, \mathbb{R}) = 0$ for all $i \geq 1$. Then $H^n(X, \mathbb{R})$ is canonically isomorphic to $H^n(\pi_1(X))$ for all $n \geq 1$.
\end{theorem}

This allows us to calculate the cohomology of all free groups of finite rank in positive degree:

\begin{theorem}
Let $F$ be a free group of finite rank. Then $H^n(F) = 0$ for all $n \geq 2$.
\end{theorem}

\begin{proof}
Let $S$ be a basis of $F$. Then $F \cong \pi_1(X)$, where $X$ is a bouquet of $|S|$ circles. This is a one-dimensional cell complex so $H^n(X, \mathbb{R}) = 0$ for all $n \geq 2$. Moreover, $X$ admits a $2|S|$-valent tree $\tilde{X}$ as its universal cover. In particular, $\tilde{X}$ is contractible, so $H_i(\tilde{X},\mathbb{R}) = 0$ for all $i \geq 1$. We conclude by applying the previous theorem.
\end{proof}

Since we understand ordinary group cohomology much better than bounded cohomology, it is useful to have a way to compare the two. The key observation is that $(C^\bullet_b(G), \delta^\bullet)$ is a subcomplex of $(C^\bullet(G), \delta^\bullet)$. Therefore the inclusion $C^n_b(G)^G \to C^n(G)^G$ induces a map at the level of cohomology
$$c^n : H^n_b(G) \to H^n(G),$$
which we call the \textit{comparison map}. The kernel of this map is denoted by $EH^n_b(G)$, and called \textit{exact bounded cohomology of $G$} (with trivial real coefficients). From the previous theorem, we immediately deduce:

\begin{corollary}
\label{EHF}

Let $F$ be a free group of finite rank. Then $EH^n_b(F) = H^n_b(F)$ for all $n \geq 2$.
\end{corollary}

\subsubsection{The cup product}

We start by defining the cup product on the level of cochain complexes:
$$\smile : C^n_b(G) \times C^m_b(G) \to C^{n+m}_b(G) : (f, f') \mapsto f \smile f';$$
where
$$(f \smile f')(g_0, \ldots, g_{n+m}) = f(g_0, \ldots, g_n) f'(g_n, \ldots, g_{n+m}).$$
This is a bilinear map which is $G$-equivariant, that is, $g \cdot (f \smile f') = g \cdot f \smile g \cdot f'$. So it restricts to a map $\smile : C^n_b(G)^G \times C^m_b(G)^G \to C^{n+m}_b(G)^G.$
Moreover:

\begin{lemma}
\label{cup_formula}
The coboundary of a cup product satisfies
$$\delta^{n+m}(f \smile f') = \delta^n f \smile f' + (-1)^nf \smile \delta^m f'.$$
\end{lemma}

\begin{proof}
$$\delta^{n+m}(f \smile f')(g_0, \ldots, g_{n+m+1}) = \sum\limits_{i = 0}^{n+m+1} (-1)^i (f \smile f')(g_0, \ldots, \hat{g_i}, \ldots, g_{n+m+1}) = $$
$$ = \sum\limits_{i = 0}^n (-1)^i f(g_0, \ldots, \hat{g_i}, \ldots, g_{n+1})f'(g_{n+1}, \ldots, g_{n+m+1}) + $$
$$ + \sum\limits_{i = n+1}^{n+m+1} (-1)^i f(g_0, \ldots, g_n)f'(g_n, \ldots, \hat{g_i}, \ldots, g_{n+m+1}) = $$
$$ = \left[ \delta^n f(g_0, \ldots, g_{n+1}) - (-1)^{n+1} f(g_0, \ldots, g_n) \right] f'(g_{n+1}, \ldots, g_{n+m+1}) + $$
$$ + (-1)^n f(g_0, \ldots, g_n) \left[ \delta^m f'(g_n, \ldots, g_{n+m+1}) - f'(g_{n+1}, \ldots, g_{n+m+1}) \right] = $$
$$ = \left[ \delta^n f \smile f' + (-1)^n f \smile \delta^m f' \right](g_0, \ldots, g_{n+m+1}).$$
\end{proof}

\begin{corollary}
\label{cup_bc}

The cup product descends to a well-defined map in bounded cohomology
$$\smile : H^n_b(G) \times H^m_b(G) \to H^{n+m}_b(G) : ([f], [f']) \mapsto [f] \smile [f'] := [f \smile f'],$$
which is continuous with respect to the topologies induced by the Gromov seminorms.
\end{corollary}

\begin{proof}
By Lemma \ref{cup_formula}, the cup product of two cocycles is a cocycle, and the cup product of a coycle and a coboundary, or that of a coboundary and a cocycle, is a coboundary. Therefore for classes $[f] \in H^n_b(G), [f'] \in H^m_b(G)$, the class $[f] \smile [f'] \in H^{n+m}_b(G)$ is well-defined. Since the map $\smile : Z^n_b(G) \times Z^m_b(G) \to Z^{n+m}_b(G)$ satisfies $||f \smile f'||_\infty \leq ||f||_\infty ||f'||_\infty$, it is continuous with respect to the norm $|| \cdot ||_\infty$, so the induced map on the quotient is also continuous.
\end{proof}

As usual, we can look at the same thing in the bar resolution.

\begin{lemma}
Under the canonical isometric isomorphism $B^\bullet : C^\bullet_b(G)^G \to \overline{C}^\bullet_b(G)$, the cup product becomes
$$f \smile f' (g_1, \ldots, g_{n+m}) := f(g_1, \ldots, g_n) f'(g_{n+1}, \ldots, g_{n+m}).$$
\end{lemma}

\begin{proof}
The cup product in $\overline{C}^\bullet_b(G)$ is defined by
$$\smile : \overline{C}^n_b(G) \to \overline{C}^m_b(G) \to \overline{C}^{n+m}_b(G) : (B^n f, B^m f') \mapsto B^n f \smile B^m f' := B^{n+m}(f \smile f').$$
Let $h := g_1 \cdots g_n$. Then, using the $G$-invariance of $f'$:
$$B^n f \smile B^m f' (g_1, \ldots, g_{n+m}) = B^{n+m}(f \smile f')(g_1, \ldots, g_{n+m}) = $$
$$ = (f \smile f')(1, g_1, \ldots, g_1 \cdots g_{n+m}) = f(1, g_1, \ldots, g_1 \cdots g_n) \, f'(h, hg_{n+1}, \ldots, hg_{n+1} \cdots g_{n+m}) = $$
$$ = f(1, g_1, \ldots, g_1 \cdots g_n) f' (1, g_{n+1}, \ldots, g_{n+1} \cdots g_{n+m}) = B^nf (g_1, \ldots, g_n) B^m f'(g_{n+1}, \ldots, g_{n+m}).$$
\end{proof}

Although it is not clear from the definition, the cup product is \textit{graded-commutative}, meaning that if $([f], [g]) \in H^n_b(G) \times H^m_b(G)$, then $[f] \smile [g] = (-1)^{nm} [g] \smile [f]$. This can be proven easily by working with the resolution of \textit{alternating cocycles}, as in \cite[Subsection 4.10]{Frigerio}. However it will not be needed for our results, and so we will not prove it.

\subsubsection{Duality}

We can change our point of view and see $\overline{C}^n_b(G)$ as a dual space. This approach is very useful because it allows to use classical duality results from functional analysis and consider the weak-$*$ topology on $\overline{C}^n_b(G)$. \\

Define $\overline{C}_n(G)$ to be the real free vector space with basis $G^n$. Define the boundary map $d_n : \overline{C}_n(G) \to \overline{C}_{n-1}(G)$ on the basis by
$$d_n(g_1, \ldots, g_n) = (g_2, \ldots, g_n) + \sum\limits_{i = 1}^{n-1} (-1)^i (g_1, \ldots, g_ig_{i+1}, \ldots, g_n) + (-1)^n (g_1, \ldots, g_{n-1}).$$
Then $(\overline{C}_\bullet(G), d_\bullet)$ is a chain complex of vector spaces. The algebraic dual of $\overline{C}_n(G)$ is $\overline{C}^n(G)$ (maps from a free vector space are defined on the basis). Moreover, we can equip $\overline{C}_n(G)$ with an $\ell^1$ norm, that is
$$\left| \left| \, \sum \alpha_{\overline{g}} \cdot \overline{g} \, \right| \right|_1 := \sum |\alpha_{\overline{g}}|.$$
With this norm, $\overline{C}_n(G)$ is a normed vector space whose topological dual is $\overline{C}^n_b(G)$, and $d_n$ is a bounded linear map. \\

By construction, $\overline{\delta}^n$ is the dual map of $d_{n-1}$. Thus $(\overline{C}^\bullet(G), \overline{\delta}^\bullet)$ is the algebraic dual, and $(\overline{C}^\bullet_b(G), \overline{\delta}^\bullet)$ is the topological dual, of the chain complex of real normed vector spaces $(\overline{C}_\bullet(G), d_\bullet)$. \\

\textbf{From now on, we will only work with the bar resolution, so since there is no room for confusion we will just denote $C^\bullet_b, \delta^\bullet, Z^\bullet_b, B^\bullet_b, C^\bullet, Z^\bullet, B^\bullet, C_\bullet$, without the bar on top.}

\subsection{Quasimorphisms}
\label{subsecqm}

Quasimorphisms are the main object of interest of this thesis, since they describe the whole second bounded cohomology of the free group.

\subsubsection{Definition and first properties}

Let us focus on exact bounded cohomology in degree two: $EH^2_b(G)$. To identify our objects more clearly, we will denote by $\ell^\infty(G)$ the space of bounded real-valued functions on $G$, and notice that $C^1_b(G) = \ell^\infty(G)$; and by $Hom(G) := Hom(G, \mathbb{R})$ the space of real-valued homomorphisms of $G$, and notice that $Z^1(G) = Hom(G)$ (see Example \ref{deg1}). \\

Let $f \in Z^2_b(G)$ be such that $[f] \in EH^2_b(G)$. This is equivalent to the fact that $f$ is bounded and $f \in B^2(G)$; that is, there exists an element $\varphi \in C^1(G)$ (which, crucially, is not necessarily bounded) such that $\delta^1 \varphi = f$. This motivates the following definition:

\begin{definition}
For a map $\varphi : G \to \mathbb{R}$, define its \textit{defect} to be
$$D(\varphi) := \sup\limits_{g, h \in G} |\varphi(g) + \varphi(h) - \varphi(gh)|.$$
We say that $\varphi$ is a (real-valued) \textit{quasimorphism} if $D(\varphi) < \infty$. We denote by $Q(G)$ the space of quasimorphisms of $G$.
\end{definition}

The simplest examples of quasimorphisms are homomorphisms and bounded functions. Recall from Example \ref{deg1} that $\ell^\infty(G) \cap Hom(G) = \{ 0 \}$.

\begin{definition}
The space of \textit{trivial quasimorphisms} is $\ell^\infty(G) \oplus Hom(G) \leq Q(G)$.
\end{definition}

By definition of the bar resolution, $\varphi(g) + \varphi(h) - \varphi(gh) = \delta^1 \varphi (g, h)$. Therefore any element in $EH^2_b(G)$ can be represented by $\delta^1 \varphi$, where $\varphi$ is a quasimorphism. Moreover, $[\delta \varphi] = 0 \in H^2_b(G)$ if and only if $\delta^1 \varphi = \delta^1 \psi$ for some $\psi \in C^1_b(G) = \ell^\infty(G)$, in which case we can write $\varphi = \psi + (\varphi - \psi)$, where $\varphi - \psi \in Z^1(G) = Hom(G)$. We thus proved:

\begin{proposition}
\label{EHiso}

The map $\delta^1$ induces an isomorphism of vector spaces
$$Q(G) / \ell^\infty(G) \oplus Hom(G) \to EH^2_b(G).$$
\end{proposition}

\begin{definition}
Two quasimorphisms are \textit{equivalent} if they differ by a trivial quasimorphism. By Proposition \ref{EHiso}, two quasimorphisms are equivalent if and only if their coboundaries represent the same bounded cohomology class.
\end{definition}

The definition of a quasimorphism gives us a finite bound for the distance between $\varphi(gh)$ and $\varphi(g) + \varphi(h)$. More generally:

\begin{lemma}
\label{qmsum}
Let $\varphi \in Q(G)$ and $g_1, \ldots, g_n \in G$. Then
$$\left| \sum\limits_{i = 1}^n \varphi(g_i) - \varphi(g_1 \cdots g_n) \right| \leq (n-1)D(\varphi).$$
\end{lemma}

\begin{proof}
The case $n = 2$ is by definition. Suppose that the statement holds up to $n$. Then
$$\left| \sum\limits_{i = 1}^{n+1} \varphi(g_i) - \varphi(g_1 \cdots g_{n+1}) \right| \leq \left|  \sum\limits_{i = 1}^n \varphi(g_i) - \varphi(g_1 \cdots g_n) \right| + |\varphi(g_1 \cdots g_n) + $$
$$ + \varphi(g_{n+1}) - \varphi(g_1 \cdots g_{n+1})| \leq (n-1)D(\varphi) + D(\varphi) = n D(\varphi).$$
\end{proof}

We can quickly reduce to a class of quasimorphisms that is better behaved:

\begin{definition}
Let $f : G \to \mathbb{R}$ be a map. We say that $f$ is \textit{alternating} if $f(g^{-1}) = -f(g)$ for any $g \in G$. In particular, $f(1) = 0$. We denote by $Q_{alt}(G)$ the space of alternating quasimorphisms, and by $\ell^\infty_{alt}(G)$ the space of alternating bounded functions.
\end{definition}

For a quasimorphism $\varphi \in Q(G)$, define
$$\varphi'(g) := \frac{\varphi(g) - \varphi(g^{-1})}{2}.$$

\begin{lemma}
\label{antisymm}

Let $\varphi \in Q(G)$. Then $\varphi' \in Q_{alt}(G)$. Moreover $||\varphi - \varphi'||_\infty \leq D(\varphi)$ (so $\varphi$ and $\varphi'$ are equivalent), and $D(\varphi') \leq D(\varphi)$.
\end{lemma}

\begin{proof}
We start by estimating the distance of $\varphi$ and $\varphi'$:
$$|\varphi(g) - \varphi'(g)| = \left| \frac{\varphi(g) + \varphi(g^{-1})}{2} \right| = \left| \frac{\delta^1 \varphi(g, g^{-1})}{2} + \frac{\varphi(1)}{2} \right| \leq D(\varphi),$$
where we used that $|\varphi(1)| = |\delta^1 \varphi(1, 1)| \leq D(\varphi)$. From this it follows that $\varphi'$ is a quasimorphism, and it is clear that it is alternating. We estimate its defect:
$$|\delta^1 \varphi'(g, h)| = \left| \frac{\varphi(g) + \varphi(h) - \varphi(gh)}{2} - \frac{\varphi(h^{-1}) + \varphi(g^{-1}) - \varphi(h^{-1}g^{-1})}{2} \right| \leq D(\varphi).$$
\end{proof}

\begin{corollary}
The map $\delta^1$ induces an isomorphism of vector spaces
$$Q_{alt}(G)/\ell^\infty_{alt}(G) \oplus Hom(G) \to EH^2_b(G).$$
\end{corollary}

\begin{proof}
By the previous lemma, every quasimorphism is at a bounded distance from an alternating one, so the map $Q_{alt}(G) \to EH^2_b(G)$ from Proposition \ref{EHiso} is surjective.
\end{proof}

We will shortly see that we can do even better, but for the moment it is useful to explicitely record this result since we will always work with alternating quasimorphisms. \\

Here is another useful fact about alternating quasimorphisms.

\begin{lemma}
Let $\varphi \in Q_{alt}(G)$. Then $\varphi([g, h]) \leq 3 D(\varphi)$.
\end{lemma}

\begin{proof}
Using Lemma \ref{qmsum} and the alternating property:
$$|\varphi([g, h])| = |\varphi([g, h]) - \varphi(g) - \varphi(h) - \varphi(g^{-1}) - \varphi(h^{-1})| \leq 3 D(\varphi).$$
\end{proof}

As a corollary, we obtain a similar estimate for all quasimorphisms:

\begin{corollary}
\label{dcomm}

Let $\varphi \in Q(G)$. Then $\varphi([g, h]) \leq 4 D(\varphi)$.
\end{corollary}

\begin{proof}
Applying Lemma \ref{antisymm}:
$$|\varphi([g, h])| \leq ||\varphi - \varphi'||_\infty + |\varphi'([g, h])| \leq D(\varphi) + 3D(\varphi') \leq 4D(\varphi).$$
\end{proof}

Using approximations by well-behaved classes of quasimorphisms (alternating or homogeneous, which will be the subject of the next paragraph) is a very useful technique. Here is another such example:

\begin{lemma}
\label{qm_gg-1}
Let $\varphi \in Q(G)$. Then $|\varphi(g) + \varphi(g^{-1})| \leq 2 D(\varphi)$ for all $g \in G$.
\end{lemma}

\begin{proof}
$$|\varphi(g) + \varphi(g^{-1})| = |(\varphi(g) - \varphi'(g)) + (\varphi(g^{-1}) - \varphi'(g^{-1}))| \leq 2 ||\varphi - \varphi'||_\infty \leq 2D(\varphi).$$
\end{proof}

\subsubsection{Homogeneous quasimorphisms}

We have seen that one can approximate any quasimorphism by an alternating one. This is useful because alternating quasimorphisms are those that appear more naturally, and the association $\varphi \mapsto \varphi'$ is simple and easy to compute. There is another special class of quasimorphisms that plays a more important role.

\begin{definition}
A quasimorphism $\varphi$ is called \textit{homogeneous} if $\varphi(g^n) = n \varphi(g)$ for all $g \in G$ and all $n \in \mathbb{Z}$. We denote by $Q_h(G)$ the space of homogeneous quasimorphisms.
\end{definition}

We will shortly see that one can approximate any quasimorphism \textit{uniquely} by a homogeneous one. This is much better than what happens with alternating quasimorphisms. The drawback is that the corresponding association $\varphi \mapsto \overline{\varphi}$ is much harder to compute. For the moment, let us start by showing just how well-behaved homogeneous quasimorphisms are:

\begin{lemma}
\label{hqm}
\begin{enumerate}
\item Homogeneous quasimorphisms restrict to homomorphisms on abelian subgroups.
\item Homogeneous quasimorphisms are conjugacy invariant.
\item If $\varphi \in Q_h(G)$ and $g, h \in G$, then $|\varphi([g, h])| \leq D(\varphi)$.
\item If $\varphi, \psi \in Q_h(G)$ satisfy $||\varphi - \psi||_\infty < \infty$, then $\varphi = \psi$.
\end{enumerate}
\end{lemma}

\begin{proof}
1. Let $A \leq G$ be an abelian subgroup, and let $a, b \in A$. Then
$$|\varphi(a) + \varphi(b) - \varphi(ab)| = \frac{1}{n} |\varphi(a^n) + \varphi(b^n) - \varphi(a^n b^n)| \leq \frac{1}{n} D(\varphi) \xrightarrow{n \to \infty} 0.$$

2. Let $g, h \in G$. Then, using Lemma \ref{qmsum} and the fact that homogeneous quasimorphisms are alternating:
$$|\varphi(ghg^{-1}) - \varphi(h)| = \frac{1}{n}|\varphi(gh^ng^{-1}) - \varphi(h^n)| = $$
$$ = \frac{1}{n}|\varphi(gh^ng^{-1}) - \varphi(g) - \varphi(h^n) - \varphi(g^{-1})| \leq \frac{1}{n} 2 D(\varphi) \xrightarrow{n \to \infty} 0.$$

3. Using the previous point, and the fact that homogeneous quasimorphisms are alternating, we obtain $\varphi(ghg^{-1}) = \varphi(h) = - \varphi(h^{-1})$. Thus:
$$|\varphi([g, h])| = |\varphi(ghg^{-1}) + \varphi(h^{-1}) - \varphi(ghg^{-1}h^{-1})| \leq D(\varphi).$$

4. Since $Q_h(G)$ is a vector space, it is enough to show that any bounded homogeneous quasimorphism is trivial. Let $\varphi \in Q_h(G)$ be bounded by $C$. Then for any $g \in G$:
$$|\varphi(g)| = \frac{1}{n}|\varphi(g^n)| \leq \frac{1}{n} C \xrightarrow{n \to \infty} 0.$$
\end{proof}

The last point of this lemma shows in particular that for any quasimorphism, there exists at most one homogeneous quasimorphism that is at bounded distance from it. In fact, such a homogeneous quasimorphism does exist:

\begin{definition}
Let $\varphi \in Q(G)$. For $g \in G$, define
$$\overline{\varphi}(g) := \lim\limits_{n \to \pm \infty} \frac{\varphi(g^n)}{n}.$$
We call the map $\overline{\varphi}$ the \textit{homogenization} of $G$.
\end{definition}

\begin{proposition}
\label{homogenization}
For any quasimorphism $\varphi$, the homogenization $\overline{\varphi}$ is a well-defined homogeneous quasimorphism. Moreover, $D(\overline{\varphi}) \leq 4 D(\varphi)$ and $||\varphi - \overline{\varphi}||_\infty \leq D(\varphi)$. Therefore $\overline{\varphi}$ is the unique homogeneous quasimorphism that is at a bounded distance from $\varphi$.
\end{proposition}

\begin{remark}
In fact, the stronger inequality $D(\overline{\varphi}) \leq 2 D(\varphi)$ holds: see for example \cite[Lemma 2.58]{Calegari}. However the proof is more involved, and for our purposes all that matters is that the bound is linear.
\end{remark}

\begin{proof}
Fix $g \in G$. By Lemma \ref{qmsum}, for any $n, m \geq 1$:
$$\left| \frac{\varphi(g^{nm})}{nm} - \frac{\varphi(g^m)}{m} \right| = \frac{1}{nm} |\varphi((g^m)^n) - n\varphi(g^m)| \leq \frac{1}{nm}(n-1) D(\varphi) \leq \frac{1}{m} D(\varphi).$$
Thus:
$$\left| \frac{\varphi(g^n)}{n} - \frac{\varphi(g^m)}{m} \right| \leq \left| \frac{\varphi(g^n)}{n} - \frac{\varphi(g^{nm})}{nm} \right| + \left| \frac{\varphi(g^{nm})}{nm} - \frac{\varphi(g^m)}{m} \right| \leq \left( \frac{1}{n} + \frac{1}{m} \right) D(\varphi).$$
It follows that the sequence $(\varphi(g^n)/n)_{n \geq 1}$ is Cauchy. Replacing $g$ by $g^{-1}$, the same is true for the sequence $(\varphi(g^n)/n)_{n \leq -1}$. Moreover $|\varphi(g^n) + \varphi(g^{-n})| \leq 2D(\varphi)$ by Lemma \ref{qm_gg-1}. We conclude that
$$\overline{\varphi}(g) = \lim\limits_{n \to \pm \infty} \frac{\varphi(g^n)}{n}$$
exists for every $g \in G$. By construction this map satisfies $\overline{\varphi}(g^n) = n \varphi(g)$ for every $n \in \mathbb{Z}$. Moreover, by Lemma \ref{qmsum}
$$\left| \varphi(g) - \frac{\varphi(g^n)}{n} \right| \leq \frac{1}{n} |n \varphi(g) - \varphi(g^n)| \leq \frac{1}{n}(n-1) D(\varphi) \leq D(\varphi);$$
so by taking the limit $||\varphi - \overline{\varphi}||_\infty \leq D(\varphi)$. It follows immediately that $\overline{\varphi}$ is a quasimorphism and that it satisfies $D(\overline{\varphi}) \leq 3 ||\varphi - \overline{\varphi}||_\infty + D(\varphi) \leq 4D(\varphi)$.
\end{proof}

We deduce the following important result:

\begin{theorem}
\label{EHiso2}
The map $\delta^1$ induces an isomorphism of vector spaces
$$Q_h(F) / Hom(G) \to EH^2_b(G).$$
\end{theorem}

\begin{proof}
Proposition \ref{homogenization} and the last point of Lemma \ref{hqm} imply that $Q(G) = Q_h(G) \oplus \ell^\infty(G)$. Now the result follows immediately from Proposition \ref{EHiso}.
\end{proof}

\subsubsection{The defect topology}

The map $D : Q_h(G) \to \mathbb{R}_{\geq 0}$ satisfies $D(\varphi + \psi) \leq D(\varphi) + D(\psi)$, for $\lambda \in \mathbb{R}$: $D(\lambda \varphi) = |\lambda|D(\varphi)$, and finally $D(\varphi) = 0$ if and only if $\varphi \in Hom(G)$. In other words: if we mod out homomorphisms, the defect is a norm.

\begin{definition}
The \textit{defect norm} on $EH^2_b(G)$ is the one induced by the norm $D$ on the space $Q_h(G)/Hom(G)$. The topology induced by this norm is called the \textit{defect topology}.
\end{definition}

Recall that $EH^2_b(G)$ also comes equipped with the restriction of the Gromov seminorm.

\begin{proposition}
\label{deftop}

The restriction of the Gromov seminorm to $EH^2_b(G)$ is a norm, and it is equivalent to the defect norm.
\end{proposition}

\begin{proof}
Let $c \in Z^2_b(G)$ be a cocycle whose class is in $EH^2_b(G)$. Let $\varphi \in Q(G)$ be such that $\delta^1 \varphi = c$, so that the defect norm of $[c]$ is $D(\overline{\varphi})$.
$$||[c]||_\infty = ||[\delta^1 \varphi] ||_\infty = \inf\limits_{b \in \ell^\infty(G)} ||\delta^1(\varphi + b)||_\infty = \inf\limits_{b \in \ell^\infty(G)} D(\varphi + b) \leq D(\overline{\varphi}).$$
Moreover, using linearity of the homogenization and Proposition \ref{homogenization}, for any $b \in \ell^\infty(G)$ we have
$$D(\overline{\varphi}) = D(\overline{\varphi + b})  \leq 4 D(\varphi + b).$$
We conclude that
$$||[c]||_\infty \leq D(\overline{\varphi}) \leq 4 ||[c]||_\infty.$$
Since $D$ is a norm, it follows that $|| \cdot ||_\infty$ is also a norm.
\end{proof}

\begin{remark}
The factor of 4 in the proof can be improved to a factor of 2, as we mentioned in Proposition \ref{homogenization}.
\end{remark}

In fact, by a result of Matsumoto and Morita, the Gromov seminorm on $H^2_b(G)$ is always a norm: for a proof see \cite[Corollary 6.7]{Frigerio}. Since we are only concerned with a free group $F$, for which $EH^2_b(F) = H^2_b(F)$, this result suffices.

\begin{corollary}
Let $F$ be a free group of finite rank. Then $H^2_b(F)$ equipped with the defect norm or the Gromov seminorm is a Banach space.
\end{corollary}

\begin{proof}
By Corollary \ref{EHF}, $EH^2_b(F) = H^2_b(F)$. By Proposition \ref{deftop}, the Gromov seminorm on $H^2_b(F)$ is a norm equivalent to the defect norm. Since it is a norm, the space $B^2_b(F)$ is closed in the Banach space $Z^2_b(F)$, so the quotient space $H^2_b(F) = Z^2_b(F)/B^2_b(F)$ is also Banach.
\end{proof}

By Corollary \ref{cup_bc}, the cup product is continuous with respect to the topologies induced by the Gromov seminorms. Translating this to the defect topology, we deduce the following:

\begin{corollary}
\label{cup_cont}

Let $\varphi, \psi \in Q(F)$. Let $(\varphi_n)_{n \geq 1}, (\psi_n)_{n \geq 1} \subset Q(F)$ be such that $D(\varphi_n - \varphi), D(\psi_n - \psi) \xrightarrow{n \to \infty} 0$ and $[\delta^1 \varphi_n] \smile [\delta^1 \psi_n] = 0 \in H^4_b(F)$ for all $n \geq 1$. Then $[\delta^1 \varphi] \smile [\delta^1 \psi]$ has vanishing Gromov seminorm.
\end{corollary}

\begin{proof}
By Proposition \ref{homogenization}:
$$D(\overline{\varphi_n} - \overline{\varphi}) = D(\overline{\varphi_n - \varphi}) \leq 4 D(\varphi_n - \varphi) \xrightarrow{n \to \infty} 0.$$
Therefore $[\delta^1 \varphi_n] \to [\delta^1 \varphi]$ in the defect topology, so also in the topology induced by the Gromov seminorm, by Proposition \ref{deftop}. The same goes for $\psi_n, \psi$. It follows from Corollary \ref{cup_bc} that $0 = [\delta^1 \varphi_n] \smile [\delta^1 \psi_n] \to [\delta^1 \varphi] \smile [\delta^1 \psi]$ with respect to the topology induced by the Gromov seminorm on $H^4_b(F)$. The result follows by noticing that the closure of $0 \in H^4_b(F)$ is precisely the space of elements with vanishing Gromov seminorm.
\end{proof}

If the Gromov seminorm on $H^4_b(F)$ were a norm, it would follow that the conclusion of the previous corollary is actual vanishing of the cup product. However, nothing is known about $H^4_b(F)$. For comparison, we have already seen that the Gromov seminorm on $H^2_b(F)$ is a norm, and it turns out that in $H^3_b(F)$ (as for any acylindrically hyperbolic group) the subspace of elements with vanishing Gromov seminorm is uncountably dimensional \cite{FPS}.

\pagebreak

\section{Quasimorphisms of the free group}
\label{secf}

We will begin by introducing the reduced defect, a variation on the defect for quasimorphisms of the free group that makes some calculations simpler. Next, we define Rolli quasimorphisms, which constitute our first example of non-trivial quasimorphisms of the free group, and suffice to prove that $H^2_b(F)$ is uncountably dimensional. This will serve as a shorter way to prove this important result, since we will study Brooks quasimorphisms much more in depth; and as a first example of some ideas that will come back throughout the rest of this thesis.

We will then review some combinatorial notions on the free group, in order to work with Brooks quasimorphisms and their infinite sums, culminating in Grigorchuk's density theorem. The space $Cal$ will finally be introduced along with some of its subspaces, and then compared to the whole of $H^2_b(F)$. \\

Throughout, $F$ is a non-abelian free group of finite rank with a fixed basis $S = \{ s_1, \ldots, s_n \}$.

\subsection{The reduced defect}

When checking whether a map $\varphi : F \to \mathbb{R}$ is a quasimorphism, most of the time it is easier to estimate the quantity $|\varphi(g) + \varphi(h) - \varphi(gh)|$ when $gh$ is a reduced product, i.e., one with no cancellation. This can be stated equivalently as $|gh| = |g| + |h|$. The following lemma shows that, for alternating maps, estimating on such pairs is enough.
 
\begin{definition}
For a map $\varphi : F \to \mathbb{R}$, define its \textit{reduced defect} to be
$$\tilde{D}(\varphi) := \sup\limits_{gh \text{ reduced }} |\varphi(g) + \varphi(h) - \varphi(gh)|.$$
\end{definition}

\begin{lemma}
\label{rd}

Let $\varphi$ be an alternating map. Then:
$$\tilde{D}(\varphi) \leq D(\varphi) \leq 3 \tilde{D}(\varphi).$$
In particular, $\varphi$ is a quasimorphism if and only if $\tilde{D}(\varphi) < \infty$, and a homomorphism if and only if $\tilde{D}(\varphi) = 0$.
\end{lemma}

\begin{proof}
The first inequality is clear. Let $g, h \in F$. Then we can write $g = t_1^{-1}c$, $h = c^{-1}t_2$ as reduced expressions so that $gh = t_1^{-1}t_2$ is reduced. In other words, $c$ is the largest subword of $g$ that gets cancelled out when multiplying with $h$  (see Figure \ref{rd_fig}). Then, using that $\varphi(c^{-1}) = - \varphi(c)$:
$$|\varphi(g) + \varphi(h) - \varphi(gh)| \leq |\varphi(t_1^{-1}) + \varphi(c) - \varphi(t_1^{-1}c)| +  |\varphi(c^{-1}) + \varphi(t_2) - \varphi(c^{-1}t_2)| + $$
$$ + |\varphi(t_1^{-1}) + \varphi(t_2) - \varphi(t_1^{-1}t_2)| \leq 3 \tilde{D}(\varphi).$$
\end{proof}

\begin{figure}
  \centering
  \includegraphics[width=6cm]{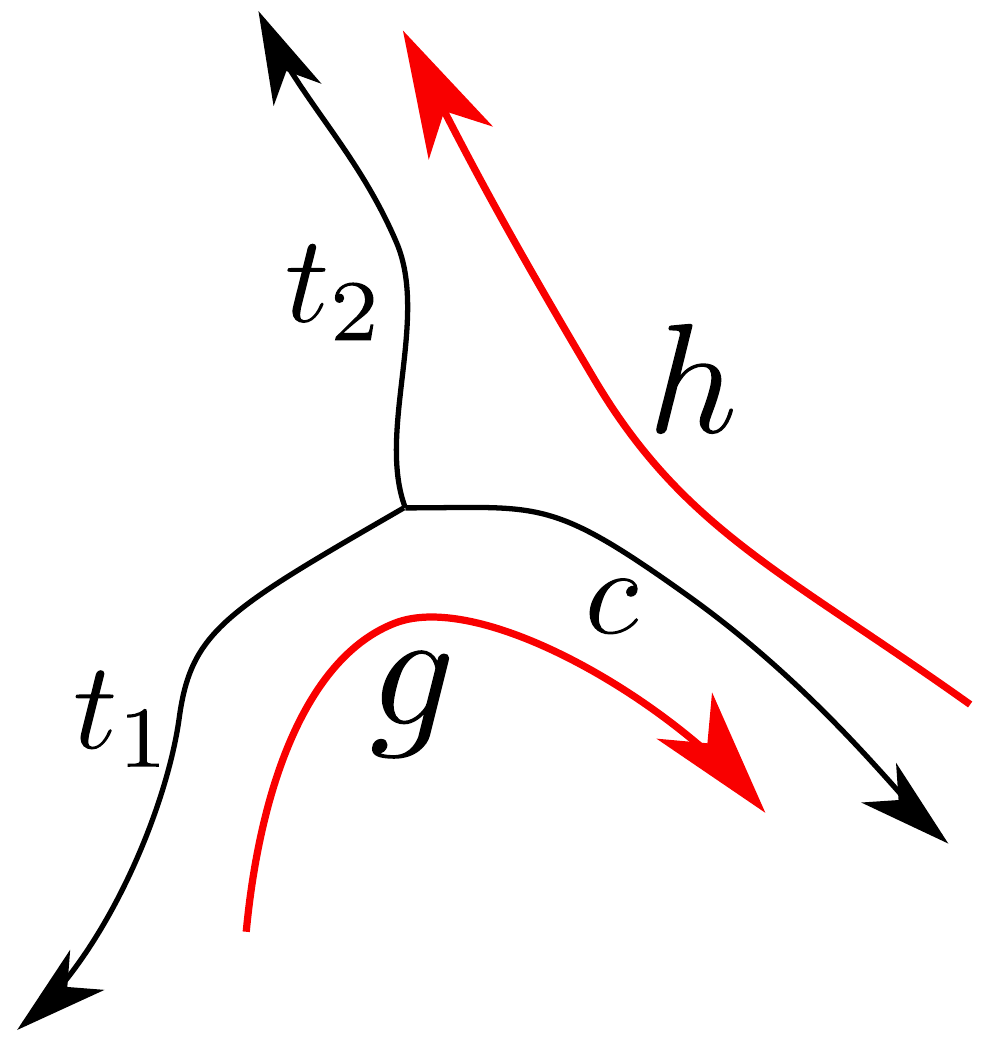}
  \caption{$c$ is the largest subword that gets simplified in the product $gh$.}
  \label{rd_fig}
\end{figure}

The requirement that $\varphi$ is alternating is necessary, as the next example shows.

\begin{example}
Define $f : F \to \mathbb{R} : g \mapsto |g|$. Then for any $g, h$ such that $gh$ is reduced, we have $|gh| = |g| + |h|$, so $\tilde{D}(f) = 0$. However, $f$ is not a quasimorphism. Indeed, $|f(g) + f(g^{-1}) - f(1)| = 2|g|$, which can be arbitrarily large.
\end{example}

\subsection{Rolli quasimorphisms}

Any element $1 \neq g \in F$ can be uniquely written as $g = \prod s_{i_j}^{m_j}$, where $s_{i_j} \neq s_{i_{j+1}}$ and $m_j \in \mathbb{Z} \, \backslash \, \{ 0 \}$. Now consider a sequence of bounded alternating maps $\lambda := (\lambda_1, \ldots, \lambda_n) \in \ell^\infty_{alt}(\mathbb{Z})^n$. Define
$$\phi_\lambda : F \to \mathbb{R} : g = \prod s_{i_j}^{m_j} \mapsto \sum \lambda_{i_j}(m_j).$$

We call maps of this form \textit{Rolli quasimorphisms}. The case in which $\lambda_i = \lambda_j$ for all $1 \leq i, j \leq n$ was introduced by Rolli in the paper \cite{Rolli}. Notice that since all $\lambda_i$ are alternating, the map $\phi_\lambda$ is alternating.

\begin{proposition}[Rolli]
With the notation above, $\phi_\lambda$ is an alternating quasimorphism.
\end{proposition}

\begin{proof}
We will show that $\tilde{D}(\phi_\lambda) < \infty$. Let $g, h \in F$ be such that $gh$ is reduced. We write out $g$ and $h$ as products of powers of generators: $g = s_{i_1}^{m_1} \cdots s_{i_k}^{m_k}$; $h = s_{j_1}^{p_1} \cdots s_{j_l}^{p_l}$. Since $gh$ is reduced, writing this element the same way we split into two cases: either $i_k \neq j_1$ and $gh = s_{i_1}^{m_1} \cdots s_{i_k}^{m_k} s_{j_1}^{p_1} \cdots s_{j_l}^{p_l}$; or $i_k = j_1 =: i$, which forces $m_k$ and $p_1$ to have the same sign, and $gh = s_{i_1}^{m_1} \cdots s_{i_{k-1}}^{m_{k-1}} s_i^{m_k + p_1} s_{j_2}^{p_2} \cdots s_{j_l}^{p_l}$. It is immediate to check that in the first case $\delta^1 \phi_\lambda(g, h) = 0$. In the second case:
$$|\delta^1 \phi_\lambda(g, h)| = |\lambda_i(m_k) + \lambda_i(p_1) - \lambda_i(m_k + p_1)| \leq 3||\lambda_i||_\infty.$$
So $\phi_\lambda$ is indeed a quasimorphism.
\end{proof}

These quasimorphisms alone are enough to prove that $H^2_b(F)$ is uncountaby dimensional:

\begin{proposition}[Rolli]
The map $\ell^\infty_{alt}(\mathbb{Z})^n \to H^2_b(F) : \lambda \mapsto [\delta^1 \phi_\lambda]$ is injective. Therefore $H^2_b(F)$ is uncountably dimensional.
\end{proposition}

\begin{proof}
Since $\lambda \mapsto \phi_\lambda$ is linear, as well as $\varphi \mapsto [\delta^1 \varphi]$, it is enough to show that if $\lambda \neq 0$, then $[\delta^1 \phi_\lambda] \neq 0$. So suppose that $[\delta^1 \phi_\lambda] = 0$. By Proposition \ref{EHiso}, $\phi_\lambda = b + h$, where $b \in \ell^\infty(F)$ and $h \in Hom(F)$. So for any $s_i \in S$ and any $k \geq 1$, we have $k \cdot h(s_i) = h(s_i^k) = \lambda_i(k) - b(k)$. Since the right-hand side is bounded, it follows that $h(s_i) = 0$ for all $1 \leq i \leq n$, and so $h = 0$ and $\phi_\lambda = b$ is bounded.

Now for $i \neq j$ and $a, b, c \geq 1$, it holds $\phi_\lambda((s_i^a s_j^b)^c) = c (\lambda_i(a) + \lambda_j(b))$. Since the left-hand side is bounded, so is the right-hand side, which forces $\lambda_i(a) + \lambda_j(b) = 0$ for all $a, b \geq 1$. This implies that $\lambda_i$ and $\lambda_j$ are constant and $\lambda_i = - \lambda_j$. Repeating the same argument with $b \leq -1$, using that $\lambda_j$ is alternating, we obtain $\lambda_i = \lambda_j$, and so $\lambda_i = \lambda_j = 0$. This being true for all $i \neq j$, we conclude that $\lambda = 0$.
\end{proof}

\subsection{Some combinatorics of words}
\label{s_cow}

Here we introduce some of the notions that are at the heart of the combinatorics of Brooks quasimorphisms and their infinite sums. The content of this section can be found in chapter one of \cite{cgt}, and chapters one and five of \cite{cow}. The second book deals with the combinatorics of words in the free monoid, so here the result are adapted to the setting of free groups.

\subsubsection{First definitions}


\begin{definition}
Let $w \in F$. We denote by $|w|$ the \textit{word length} of $w$ with respect to our fixed basis $S$. That is, if $w = s_1 \cdots s_n$, where $s_i \in S^{ \pm 1}$ and $s_{i+1} \neq s_i$, then $|w| = n$.

Let $w, w' \in F$. We say that the product $ww'$ is \textit{reduced} if there is no cancellation. That is, if $|ww'| = |w| + |w'|$.
\end{definition}

\begin{definition}
Let $w \in F$. Suppose that we can write $w = uw'v$ as a reduced product. Then we say that $w'$ is a \textit{subword} of $w$. If $u$ is empty, it is a \textit{prefix}. If $v$ is empty, it is a \textit{suffix}. If $w' \neq w$, it is a \textit{proper} subword, prefix or suffix. 
\end{definition}

\begin{remark}
In combinatorics of words, what we defined as a subword is usually called a \textit{factor}, while the term subword refers to another concept. However, for Brooks quasimorphisms this terminology is well-established enough that we decided to use this convention. The combinatorial notion of subword will never be dealt with in this thesis, so there is no risk of confusion.
\end{remark}

\begin{definition}
Let $w = s_1 \cdots s_n$ be a word. A \textit{cyclic permutation} of $w$ is a word of the form $s_{k+1} \cdots s_n s_1 \cdots s_k$.

A word is \textit{cyclically reduced} if it cannot be written as a reduced product of the form $xyx^{-1}$, where $x, y \neq 1$. Equivalently, $w$ is cyclically reduced if all of its cyclic permutations are reduced.

The \textit{length} of a conjugacy class in $F$ is the minimal length of one of its elements. Equivalently, it is the length of a cyclically reduced representative.
\end{definition}

Any conjugacy class has a cyclically reduced representative $w$, and the following lemma shows that the set of cyclically reduced representatives is precisely the set of cyclic permutations of $w$.

\begin{lemma}
\label{conj_cr}
Let $w, w'$ be cyclically reduced words in $F$. The following are equivalent:
\begin{enumerate}
\item $w$ and $w'$ are conjugate.
\item There exist words $u, v$ such that $w = uv$ and $w' = vu$ as reduced products.
\item $w'$ is a cyclic permutation of $w$.
\end{enumerate}
\end{lemma}

\begin{proof}
$1. \Rightarrow 2.$ Suppose that $w' = vwv^{-1}$ (where the product is not necessarily reduced). Since $w'$ is cyclically reduced, its first letter cannot be equal to the inverse of its last letter. This is not possible if only parts of the conjugating elements $v, v^{-1}$ cancel out. Therefore we may assume that all of $v^{-1}$ gets canceled out (the other case being analogous). This means that we can write $w = uv$ as a reduced product, and then $w' = v(uv)v^{-1} = vu$. This product is also reduced: otherwise the last letter of $v$ would cancel out the first letter of $u$, and $w$ would not be cyclically reduced.

$2. \Rightarrow 3.$ Let $w = s_1 \cdots s_n$, with $u = s_1 \cdots s_k$ and $v = s_{k+1} \cdots s_n$. Then $w' = vu = s_{k+1} \cdots s_n s_1 \cdots s_k$, which is a cyclic permutation of $w$.

$3. \Rightarrow 1.$ A cyclic permutation is a conjugate by a suffix.
\end{proof}

A special role is played by the conjugacy classes of length 1, which correspond to the conjugates of the elements of $S^{\pm 1}$.


\begin{definition}
A word is \textit{simple} if it is not a proper power. A conjugacy class is \textit{simple} if one of its elements is simple. It follows automatically that all of the elements of a simple conjugacy class are simple, since a conjugate of a $k$-th power is again a $k$-th power.
\end{definition}

\begin{remark}
We use this terminology following Grigorchuk's paper \cite{Grigor}. However, the standard term in combinatorics of words is \textit{irreducible}.
\end{remark}

\begin{lemma}
\label{simple}

Let $w$ be a cyclically reduced word. Then $w$ is simple if and only if it has $|w|$ cyclically reduced conjugates.
\end{lemma}

\begin{proof}
Suppose that $w$ is not simple, so we can write $w = u^k$ for some word $u$ and $k \geq 2$. Notice that $u$ is cyclically reduced: otherwise $u$ would start with $s$ and end with $s^{-1}$, and so $w$ would have the same property and not be cyclically reduced. Therefore $|w| = k|u|$. Then conjugating by $u$ gives back $w$, so there are at most $|u|$ cyclic permutations of $w$, and $|u| < k|u| = |w|$.

Suppose that $w$ has less than $|w|$ cyclically reduced conjugates. Then there exists such a conjugate $s_1 \cdots s_n$ of $w$ and some $1 \leq k < n$ such that $s_1 \cdots s_n = s_{k+1} \cdots s_n s_1 \cdots s_k$. Comparing letter by letter, it follows that $s_1 \cdots s_n$ is a proper power of $s_1 \cdots s_k$. Then $w$ is a proper power, being conjugate to a proper power.
\end{proof}

The notion of conjugacy for cyclically reduced words is purely combinatorial (i.e., all of the above can be adapted to the free monoid). Let us see how the group-theoretic notion of inversion behaves with respect to it:

\begin{lemma}
\label{w_conj_w-1}
A non-identity element of $F$ is not conjugate to its inverse.
\end{lemma}


\begin{proof}
Since any element is conjugate to a cyclically reduced word, and the inverse of a cyclically reduced word is cyclically reduced, it is enough to show the lemma for cyclically reduced words. So suppose by contradiction that $w$ and $w^{-1}$ are non-empty and conjugate cyclically reduced words. By Lemma \ref{conj_cr}, there exist words $u, v$ such that $w = uv$ and $w^{-1} = vu$ as reduced products. Thus $vu = v^{-1}u^{-1}$, so comparing lengths $v = v^{-1}$ and $u = u^{-1}$. But then $u$ and $v$ must be empty, a contradiction.
\end{proof}


The next definition is the key combinatorial concept behind many arguments in this thesis:

\begin{definition}
We say that two words $w_1, w_2$ \textit{overlap properly} if they can be written as reduced products of the form $w_i = xy$, $w_{3-i} = yz$ with $x, y, z$ non-empty. We say that they \textit{overlap} if either one is a proper subword of the other, or if they overlap properly.
\end{definition}

\begin{remark}
By our definition of overlap, the trivial subword $w$ of $w$ does not count as an overlap. The only way a word can overlap itself is properly.
\end{remark}

We will see that the less overlapping words, the easier the combinatorics. So it is useful to understand the two first possible overlaps: that of a word with its inverse, and with itself.

\begin{lemma}
\label{w_ol_w-1}

A non-empty word $w$ does not overlap its inverse.
\end{lemma}

\begin{proof}
Since $w$ and $w^{-1}$ are distinct and of the same length, one cannot be a subword of the other. Suppose that $w = xy$ and $w^{-1} = yz$ as in the definition. Thus $yz = y^{-1}x^{-1}$ and so $y = y^{-1}$, which is possible only if $y$ is empty.
\end{proof}

\begin{definition}
A word is \textit{self-overlapping} if it overlaps with itself.
\end{definition}

Following the definition, $w$ is self-overlapping if $w = xy = yz$ as a reduced product for some $x, y, z \neq 1$. But this is enough to deduce more \cite[Lemma 2]{Faziev}:


\begin{lemma}[Faiziev]
\label{xyx}

Let $w$ be self-overlapping. Then $w$ can be written as a reduced product $xyx$, where $x$ is non-empty and non-self-overlapping.
\end{lemma}

\begin{proof}
Let $P$ be the set of proper prefixes and $S$ the set of proper suffixes of $w$. By hypothesis $P \cap S \neq \emptyset$. Let $x$ be the shortest non-empty word of this intersection. We claim that $x$ is non-self-overlapping. Indeed, suppose that $x = ab = bc$. Since $x = bc \in P$, $b \in P$; and since $x = ab \in S$, $b \in S$. Therefore $b \in P \cap S$, and by minimality $b = x$. But this implies that $|x| \leq |w|/2$, else the prefix $x$ would overlap the suffix $x$.
\end{proof}


\subsubsection{Lyndon words and fundamental sets}

We have already mentioned that every element of $F$ has a cyclically reduced conjugate. The following is probably more surprising (\cite[Lemma 3]{Faziev}, \cite[Lemma 5.5]{Grigor}):


\begin{lemma}[Faiziev]
\label{conj_nso}

Any simple conjugacy class has a cyclically reduced, non-self-overlapping representative.
\end{lemma}

We will shortly prove this lemma, but first, let us state the consequence that is of greatest interest to us.

\begin{definition}
\label{def_fund}

A \textit{fundamental set} is a symmetric set of cyclically reduced non-self-overlapping representatives of simple conjugacy classes of length at least 2.
\end{definition}

The definition is a mouthful, but it will hopefully seem natural after seeing the statement of Grigorchuk's density Theorem \ref{gdt}. \\

From Lemmas \ref{w_conj_w-1} and \ref{conj_nso} we immediately deduce:

\begin{corollary}[Faiziev]
\label{fund_set}

A free group of finite rank admits a fundamental set.
\end{corollary}

In order to prove Lemma \ref{conj_nso}, we introduce the following special type of words. In what follows, we fix a total order $\preceq$ on $S^{\pm 1}$, and extend it to the lexicographic order on $F$.

We recall how this is defined: given two words $w, w'$, if $w$ is a prefix of $w'$, then $w \preceq w'$. If neither word is a prefix of the other, we can write $w = xay$ and $w' = xbz$ as reduced products, where $x, y, z$ are (possibly empty) reduced words, $a, b \in S^{\pm 1}$ and $a \neq b$. Then $w \preceq w'$ if and only if $a \preceq b$.

We will also use the symbol $\prec$ in the obvious way.

\begin{definition}
A cyclically reduced simple word is a \textit{Lyndon word} if it is the $\preceq$-minimal element among its cyclically reduced conjugates.
\end{definition}

Lyndon words are an object of great interest in combinatorics of words, especially because of their factorization properties. We refer the reader to \cite[Chapter 5]{cow} for more information. \\

By definition, every simple conjugacy class contains a Lyndon word. Therefore Lemma \ref{conj_nso} follows directly from \cite[Proposition 1.1]{lynd_nso}:

\begin{lemma}[Duval]
\label{lynd_nso}

Lyndon words are non-self-overlapping.
\end{lemma}

\begin{proof}
Let $w$ be a Lyndon word, and suppose by contradiction that it is self-overlapping. By Lemma \ref{xyx}, we can write $w = uvu$ as a reduced product, where $u$ is non-empty. Since $w$ is simple, $v$ is also non-empty, and $uuv, uvu$ and $vuu$ are all distinct by Lemma \ref{simple}. Since $w$ is Lyndon, $w = uvu \prec uuv, vuu$. First, $uvu \prec vuu$ implies that $uv \prec vu$. Second, $uvu \prec uuv$ implies $vu \prec uv$. This is a contradiction, and we conclude.
\end{proof}

This lemma actually allows to identify explicitly some fundamental sets:

\begin{corollary}
Given a choice of $F^+$ and of a total order on $S^{\pm 1}$, let $\mathcal{F}^+$ denote the set of Lyndon words in $F^+$. Let $\mathcal{F} := \mathcal{F}^+ \cup (\mathcal{F}^+)^{-1}$. Then $\mathcal{F}$ is a fundamental set.
\end{corollary}

\subsubsection{Special families of words}

We denote by $\mathcal{N}$ the set of non-self-overlapping words. Notice that $\mathcal{N} = \mathcal{N}^{-1}$. \\

When dealing with infinite sums of Brooks quasimorphisms, the combinatorics will be determined by how the corresponding words overlap with one another. We consider two extreme cases. The following definition follows the terminology in \cite{qout}:

\begin{definition}
A family $I \subset F$ is \textit{independent} if no two distinct words in $I$ overlap. In particular, all words in $I$ must be non-self-overlapping.
\end{definition}

\begin{example}
\label{ex_ind}

The families $I = \{ ab^nab^{-1} : n \geq 1 \}$ and $I' = \{ a^n(ab)^nb : n \geq 1 \}$ are independent. $I'$ is the family considered by Grigorchuk in \cite[Lemma 5.10]{Grigor}. Notice that also $I \cup I^{-1}$ and $I' \cup I'^{-1}$ are independent.
\end{example}

On the opposite side of the spectrum are families in which any two distinct words overlap. Among these, compatible families play an important role.

\begin{definition}
Let $u, v \in F$ be such that $uv$ is reduced. We call $u|v$ a \textit{reduced expression}. We say that $w$ \textit{overlaps the juncture of} $u|v$ if we can write $u = xu'$, $v = v'y$ and $w = u'v'$ as reduced products, with $u', v' \neq 1$. It follows that $|w| \geq 2$. We denote this by $w \in j(u|v)$.
\end{definition}

\begin{definition}
\label{comp_def}

Let $C$ be a finite collection of words. We say that $C$ is a \textit{compatible family} if there exists a reduced expression $u|v$ such that for all $w \in C$ we have $w \in j(u|v)$. For an infinite $C$, we give the same definition but allow $u$ to be left-infinite and $v$ to be right-infinite.

A compatible family described by the reduced expression $u|v$ is called \textit{full} if it contains all words $w$ such that $w \in j(u|v)$.
\end{definition}

The following should be clear.

\begin{lemma}
All words in a compatible family have length at least 2, and any two distinct words in the same compatible family overlap.

Every word of length at least two forms a compatible family of size 1. Two words with an overlap of size at least two form a compatible family of size 2.
\end{lemma}

\begin{example}
The families $I = \{a^nb^m : n, m \geq 1 \}$ and $I' = \{ a^{2n}b^{2n-1} \cdots a^2 b a b^2 \cdots a^{2n-1} b^{2n} : n \geq 1 \}$ are compatible.
\end{example}

Notice that by Lemma \ref{w_ol_w-1} if $w \in C$, then $w^{-1} \notin C$.

\subsection{Brooks quasimorphisms}

Brooks quasimorphisms were the first non-trivial quasimorphisms to be defined, in Brooks's paper \cite{Brooks}. There are some variations of this definition, and some inconsistencies in the definitions and the notations used in the literature. We use Calegari's notation from \cite{Calegari}.

\begin{definition}
\label{Brooks}

Let $w \in F$. Define the \textit{big counting function} by
$$C_w : F \to \mathbb{N} : g \mapsto \#\text{ occurrences of } w \text{ as a subword of } g,$$
and the \textit{big Brooks quasimorphism} by
$$H_w : F \to \mathbb{Z} : g \mapsto C_w(g) - C_{w^{-1}}(g).$$
Define the \textit{small counting function} by
$$c_w : F \to \mathbb{N} : g \mapsto \text{maximal \# of \textit{disjoint} occurrences of } w \text{ as a subword of } g,$$
and the \textit{small Brooks quasimorphism} by
$$h_w : F \to \mathbb{Z} : g \mapsto c_w(g) - c_{w^{-1}}(g).$$
\end{definition}

\begin{example}
Let $w = ab$, so $w^2 = abab$. We calculate some values of the Brooks quasimorphisms on $w$ and $w^2$, which give us a non-self-overlapping and a self-overlapping example. Evaluating at $g = ababab = w^3$:
$$H_w(g) = h_w(g) = 3;$$
$$H_{w^2}(g) = 2; \,\,\,\,\,\,\, h_{w^2}(g) = 1.$$
Evaluating at $g = a^2bab^2(a^{-1}b^{-1})^4 = aw^2ba^{-1}w^{-3}b^{-1}$:
$$H_w(g) = h_w(g) = 2 - 3 = -1;$$
$$H_{w^2}(g) = 1 - 2 = -1; \,\,\,\,\,\,\, h_{w^2}(g) = 1 - 1 = 0.$$
\end{example}

Here is a very simple but important fact, which the reader should have noticed in the previous example:

\begin{lemma}
\label{nso_disj}

Let $w$ be a non-self-overlapping word. Then all occurrences of $w$ and $w^{-1}$ in a reduced word are disjoint. In particular $C_w = c_w$ and $C_{w^{-1}} = c_{w^{-1}}$, so $H_w = h_w$.
\end{lemma}

In general big Brooks quasimorphisms have simpler combinatorics and are easier to compute, while small Brooks quasimorphisms behave better algebraically, in that they have a uniformly bounded defect. This is why non-self-overlapping words are so much easier to work with.

\begin{proposition}
\label{bigsmall}

Let $w \in F$ be any word. Then $\tilde{D}(H_w) \leq |w| - 1$ and $\tilde{D}(h_w) \leq 1$. Thus $D(H_w) \leq 3(|w| - 1)$ and $D(h_w) \leq 3$.

Moreover, the value of $\delta^1 H_w(g, h)$ only depends on the $|w|$-neighbourhood of the center of the tripod with endpoints $1, g, gh$.
\end{proposition}

\begin{proof}
Let $g|h$ be a reduced expression. Consider an occurrence of $w^{\pm 1}$ in $g$ or in $h$. Then this occurrence is counted once in $H_w(g) + H_w(h)$ and once in $H_w(gh)$, so it is not counted in $\delta^1 H_w(g, h)$. The only occurrences that do not cancel out are those that overlap the juncture of $g|h$, since they are counted in $H_w(gh)$ but not in $H_w(g) + H_w(h)$. But there are at most $|w| - 1$ copies of $w^{\pm 1}$ that can overlap the juncture of $g|h$. Moreover there are either only copies of $w$ or only copies of $w^{-1}$, since $w$ and $w^{-1}$ do not overlap by Lemma \ref{w_ol_w-1}. This shows both that $|\delta^1 H_w (g, h)| \leq |w| - 1$ and that $\delta^1 H_w (g, h)$ only depends on the $|w|$-neighbourhood of the center of $g|h$. Moving to general pairs, write $g = t_1^{-1}c$, $h = c^{-1} t_2$ and $gh = t_1^{-1}t_2$ as reduced products (see Figure \ref{rd_fig}). Then, using that $H_w$ is alternating:
$$\delta^1 H_w(g, h) = H_w(g) + H_w(h) - H_w(gh) = -\delta^1 H_w(t_1^{-1}, c) - \delta^1 H_w(c^{-1}, t_2) + \delta^1 H_w(t_1^{-1}, t_2).$$
By the reduced case, each of the last three terms only depends on the $|w|$-neighbourhood of the center of the tripod, and so $\delta^1 H_w(g, h)$ does too. \\

Moving to $h_w$, let again $g|h$ be a reduced expression. Let $U_g$ be a maximal collection of disjoint occurrences of $w$ as a subword of $g$, where by maximal we mean that $c_w(g) = |U_g|$. Let $U_h$ be the analogue for $h$. Then $U_g \cup U_h$ is a collection of disjoint occurrences of $w$ as a subword of $gh$, which shows that $c_w(gh) \geq c_w(g) + c_w(h)$. Now let $U_{gh}$ be a maximal collection of disjoint occurrences of $w$ as a subword of $gh$. If $U_{gh}$ does not contain a copy of $w$ that intersects the juncture of $g|h$, then it decomposes as $V_g \cup V_h$, where all occurrences in $V_g$ are contained in $g$ and analogously for $V_h$ and $h$; so $c_w(gh) \leq c_w(g) + c_w(h)$. Else $U_{gh}$ contains a copy of $w$ intersecting the juncture, in which case there is precisely one, since all occurrences of $w$ in $U_{gh}$ are disjoint, and so similarly $c_w(gh) \leq c_w(g) + c_w(h) + 1$. Putting everything together, $c_w(gh) - c_w(g) - c_w(h) \in \{0, 1 \}$, and 1 can be the value only if $w \in j(g|h)$. Similarly $c_{w^{-1}}(gh) - c_{w^{-1}}(g) - c_{w^{-1}}(h) \in \{ 0, 1 \}$, and 1 can be the value only if $w^{-1} \in j(g|h)$. But $w$ and $w^{-1}$ do not overlap by Lemma \ref{w_ol_w-1}, so only one of these two values can be non-zero. We conclude that $\delta^1 h_w(g, h) \in \{-1, 0, 1 \}$.
\end{proof}

\begin{remark}
In general we cannot do better than $D(H_w) = \mathcal{O}(|w|)$: for instance $\delta^1 H_{a^n} (a^n, a^n) = 1-n$ for any $n \geq 1$. It turns out that $D(h_w) \leq 2$, but the proof is more involved and we will not need this sharper bound; for a proof see \cite[Proposition 2.30]{Calegari}.
\end{remark}

The fact that $\delta^1 H_w(g, h)$ only depends on a small neighbourhood of the center of the tripod with endpoints $1, g, gh$ makes the calculations much easier, since there are large amounts of information that can be safely ignored. This is not the case for $h_w$.

\begin{example}
\label{small_Brooks}
Let $w$ be a cyclically reduced word. Then $\delta^1 h_{w^2}(w^{2n}, w^{2n}) = n + n - 2n = 0$, while $\delta^1 h_{w^2}(w^{2n+1}, w^{2n+1}) = n + n - (2n + 1) = -1$. So adding relatively small words at the extremities can change the value of the coboundary. 
\end{example}

The argument in the proof for $H_w$ can be turned into a useful formula for non-self-overlapping words:

\begin{lemma}
\label{formula_hw}

Let $w$ be a non-self-overlapping word and $g|h$ a reduced expression. Then:
$$\delta^1 h_w(g, h) = \delta^1 H_w(g, h) = -\mathbbm{1} \{w \in j(g|h) \} + \mathbbm{1} \{w^{-1} \in j(g|h) \}.$$
\end{lemma}

Apart from the statement of Theorem \ref{gdt3} and the discussion at the end of Subsection \ref{subsec_so}, we will only be concerned with big Brooks quasimorphisms from now on. Thus, whenever we write $h_w$, it is to stress the fact that $w$ is a non-self-overlapping word, and so all occurrences of $w$ in a reduced word are disjoint by Lemma \ref{nso_disj}.

\subsubsection{Homogenization}

The homogenization of a Brooks quasimorphism on a non-self-overlapping word also has a combinatorial definition \cite[Lemma 5.2]{Grigor}.

\begin{definition}
Let $g \in F$. We denote by $\overline{g}$ the cyclic word on $g$.
\end{definition}

One should visualize $\overline{g}$ as a circular word. The crucial point to notice is that in $\overline{g}$ we lose track of where $g$ begins and where it ends. This is expressed more formally in the following lemma:

\begin{lemma}
The map $g \mapsto \overline{g}$ is conjugacy invariant. We have $|g| \geq |\overline{g}|$, with equality if and only if $g$ is cyclically reduced.
\end{lemma}

\begin{lemma}
Let $w$ be a non-self-overlapping word. Then
$$\overline{h_w} = \# \text{ occurrences of } w \text{ in } \overline{g} - \# \text{ occurrences of } w^{-1} \text{ in } \overline{g}.$$
\end{lemma}

\begin{proof}
Let $\tilde{h}_w$ denote the map defined in the statement. Since $g \mapsto \overline{g}$ is conjugacy invariant, so is $\tilde{h}_w$. We need to show that $\tilde{h}_w = \overline{h_w}$ is the homogenization of $h_w$. We will show this by estimating the distance between $\tilde{h}_w$ and $h_w$.

Let $g$ be cyclically reduced, so that $g|g$ is a reduced expression. Then if $|w| \leq |g|$, counting occurrences of $w$ in $\overline{g}$ is the same as counting them in $g$, except for those which overlap the juncture of $g|g$. If $|w| > |g|$, then $\tilde{h}_w(g) = h_w(g) = 0$. So by Lemma \ref{formula_hw} we have
$$\tilde{h}_w(g) - h_w(g) \in \{ -\delta h_w(g, g), 0 \}.$$
In particular $||\tilde{h}_w - h_w||_\infty \leq 1$ on cyclically reduced words. Now let $g$ be any word, and write $g = xhx^{-1}$ as a reduced product, where $h$ is cyclically reduced. Then, using that $\tilde{h}_w$ is conjugacy invariant and $h_w$ is alternating:
$$|\tilde{h}_w(g) - h_w(g)| = |\tilde{h}_w(h) - h_w(xhx^{-1})| \leq |\tilde{h}_w(h) - h_w(h)| + |h_w(h) - h_w(xhx^{-1})| \leq $$
$$ \leq 1 + |h_w(xhx^{-1}) - h_w(x) - h_w(h) - h_w(x^{-1})| \leq 1 + 2D(h_w).$$
It follows that $\tilde{h}_w$ is at a bounded distance from $h_w$, in particular it is a quasimorphism. Moreover, the reader can check that $\tilde{h}_w(g^n) = n\tilde{h}_w(g)$, so $\tilde{h}_w$ is a homogeneous quasimorphism. Since $\overline{h_w}$ is the unique homogeneous quasimorphism that is at a bounded distance from $h_w$, we conclude.
\end{proof}

\subsection{The pointwise topology}

Recall that we have an isomorphism $Q(F) / Hom(F) \oplus \ell^\infty(F) \cong H^2_b(F)$ (Proposition \ref{EHiso}) and that since $Q(F) = Q_h(F) \oplus \ell^\infty(F)$, we even have $Q_h(F) / Hom(F) \cong H^2_b(F)$ (Proposition \ref{EHiso2}). This is true in a very general context, but for the free group we can go one step further.

\begin{definition}
Define $Q^0_h(F) := \{ \varphi \in Q_h(F) : \varphi|_S \equiv 0\}$.
\end{definition}

\begin{lemma}[Grigorchuk]

We have $Q_h(F) = Q^0_h(F) \oplus Hom(F)$. Therefore the map $\delta^1$ induces an isomorphism of vector spaces:
$$Q^0_h(F) \to H^2_b(F).$$
\end{lemma}

\begin{proof}
We have a natural projection $\pi : Q_h(F) \to Hom(F)$, by mapping $\varphi$ to the homomorphism that coincides with $\varphi$ on the basis. The kernel of $\pi$ is precisely $Q^0_h(F)$.
\end{proof}

This allows to identify $H^2_b(F)$ with a space of real-valued functions on $F$. This space comes naturally equipped with the subspace topology of the product topology of $\{ f : F \to \mathbb{R} \} \cong \mathbb{R}^F$. Since $F$ is countable, this space is first-countable, and so convergent sequences determine the topology. By definition of the product topology, $f_n \to f$ in this topology if and only if $f_n$ converges pointwise to $f$, which is why this topology on $Q^0_h(F)$ is called the \textit{topology of pointwise convergence}. The same goes for the spaces $Q(F)$ and $Q_{alt}(F)$.

\begin{definition}
The \textit{pointwise topology} on $H^2_b(F)$ is the one corresponding to the topology of pointwise convergence on $Q^0_h(F)$ under the isomorphism induced by $\delta^1$.
\end{definition}

\subsubsection{Grigorchuk's density theorem}

Any infinite sum of Brooks quasimorphisms is a well-defined map. Indeed, given $g \in F$, $h_w(g), \overline{h_w}(g) \neq 0$ only if $|w| \leq |g|$, and there are only finitely many such words. The relation with the previous definition is the following:
$$\sum_{\substack{w \in \mathcal{F}^+ \\ |w| \leq n}} \alpha_w \overline{h_w} \xrightarrow{n \to \infty} \sum\limits_{w \in \mathcal{F}^+} \alpha_w \overline{h_w}$$
in the pointwise topology on $Q^0_h(F)$, if the right-hand side is indeed an element of $Q^0_h(F)$ (we will shortly go back to this question). \\

Now let $\mathcal{F}$ be a fundamental set (Definition \ref{def_fund}). Notice that if $w \in \mathcal{F}$, then by definition $|w| \geq 2$, so $\overline{h_w} \in Q^0_h(F)$. In fact, much more is true \cite[Theorem 5.7]{Grigor}.

\begin{definition}
Let $V$ be a real topological vector space. We say that the set $\{ x_i \}_{i \geq 1}$ is a \textit{basis} of $V$ if for any $x \in V$ there exist unique coefficients $a_i \in \mathbb{R}$ such that
$$\sum\limits_{i \leq N} a_i x_i \xrightarrow{N \to \infty} x.$$
\end{definition}

\begin{theorem}[Grigorchuk]
\label{gdt}

The set $\{ \overline{h_w} \}_{w \in \mathcal{F}^+}$ is a basis of $Q^0_h(F)$ with the pointwise topology. In other words, any $\varphi \in Q^0_h(F)$ may be uniquely written as
$$\sum\limits_{w \in \mathcal{F}^+} \alpha_w \overline{h_w},$$
for some coefficient map $\alpha : \mathcal{F}^+ \to \mathbb{R}$.
\end{theorem}

\begin{remark}
Since $\overline{h_w} = - \overline{h_{w^{-1}}}$, after choosing $\mathcal{F}$, this basis depends on the choice of $\mathcal{F}^+$ only up to sign.
\end{remark}

This theorem should make the definition of a fundamental set look more natural. The fact that elements of $\mathcal{F}$ are representatives of conjugacy classes reflects the fact that homogeneous quasimorphisms are conjugacy invariant. The restriction to simple conjugacy classes reflects the fact that homogeneous quasimorphisms satisfy $\varphi(g^n) = n \varphi(g)$ for all $g \in F$ and all $n \in \mathbb{Z}$. Finally, excluding the conjugacy classes of length 1 reflects the fact that elements of $Q^0_h(F)$ vanish on the basis. \\

We will not prove Theorem \ref{gdt}, but Grigorchuk's proof can be easily adapted to show a similar result, which is also relevant in our setting.

\begin{theorem}
\label{gdt2}

The set $\{H_w\}_{w \in F^+}$ is a basis of $Q_{alt}(F)$ with the topology of pointwise convergence. In other words, any $\varphi \in Q_{alt}(F)$ may be uniquely written as
$$\sum\limits_{w \in F^+} \alpha_w H_w,$$
for some coefficient map $\alpha : F^+ \to \mathbb{R}$.
\end{theorem}

\begin{remark}
Since $H_w = - H_{w^{-1}}$, this basis depends on the choice of $F^+$ only up to sign.
\end{remark}

\begin{proof}
We start by showing that the $H_w$ are linearly independent. Consider the (not necessarily finite) sum $f = \sum \alpha_w H_w$, where not all $\alpha_w$ are 0. Let $w_0$ be a word of minimal length such that $\alpha_{w_0} \neq 0$. Now if $|w| \geq |w_0|$ and $w \neq w_0^{\pm 1}$, then $H_w(w_0) = 0$; and if $|w| < |w_0|$, then $\alpha_w = 0$ by the choice of $w_0$. Thus $f(w_0) = \sum\limits_{|w| \geq |w_0|} \alpha_w H_w(w_0) =  \alpha_{w_0} \neq 0$. \\

Now we show that $\{H_w\}_{w \in F^+}$ spans $Q_{alt}(F)$. As a matter of fact, it spans all alternating functions $F \to \mathbb{R}$. Indeed, let $f : F \to \mathbb{R}$ be an alternating function. We construct the sequence $(\alpha_w)_{w \in F^+}$ by induction on the length of words. Define $F_N$ to be the subset of $F^+$ of elements of length at most $N$. Define $f_1$ to be the homomorphism that coincides with $f$ on the basis. Then $f_1 \in \langle H_w : w \in F_1 \rangle$ and $f - f_1$ vanishes on $F_1$. Now suppose we have defined $f_N \in \langle H_w : w \in F_N \rangle$ such that $f - f_N$ vanishes on $F_N$. Given $w \in F^+$ of length $N + 1$, define $\alpha_w := f(w) - f_N(w)$. Define $g_{N + 1} := \sum_{|w| = N + 1} \alpha_w H_w$, and $f_{N + 1} = f_N + g_{N + 1}$. Then a similar argument as the one for linear independence shows that $f - f_{N + 1}$ vanishes on $F_{N + 1}$. Also, $f_{N + 1}$ is a finite sum with coefficients supported on $F_{N + 1}$ and it has the same coefficients as $f_N$ on $F_N$. Therefore $f_N \to f$ pointwise and $f = \sum \alpha_w H_w$. 
\end{proof}

\begin{remark}
There is a mistake in the proof of linear independence in Grigorchuk's paper, and moreover only finite sums are considered. This can be corrected by choosing a word of minimal length instead of a word of maximal length, just as we did in the proof of Theorem \ref{gdt2}.
\end{remark}

It emerges from the proof that we can do the exact same thing with small Brooks quasimorphisms, although the coefficients will probably be different. Actually we can generalize this even more, since that the only thing we needed to get both linear independence and generation is that if $|w| \geq |w_0|$, then $H_w(w_0) = \pm \mathbbm{1} \{ w = w_0^{\pm 1} \}$. So we get:

\begin{theorem}
\label{gdt3}

For all $w \in F^+$, choose $\varphi_w \in Q_{alt}(F)$ such that if $|w| \geq |w_0|$, then $\varphi_w(w_0) = \pm \mathbbm{1} \{ w = w_0^{\pm 1} \}$. Then $\{ \varphi_w \}_{w \in F^+}$ is a basis of $Q_{alt}(F)$. In particular, $\{ h_w \}_{w \in F^+}$ is a basis of $Q_{alt}(F)$.
\end{theorem}

Crucially, not every sum will give rise to a quasimorphism. Our proof of Theorem \ref{gdt2} shows that $\{ H_w \}_{w \in F^+}$ generates the space of all alternating functions $f : F \to \mathbb{R}$. Similarly, Grigorchuk's proof of Theorem \ref{gdt} shows that $\{ \overline{h_w} \}_{w \in \mathcal{F}^+}$ generates the space of all real-valued functions on $F$ that are conjugacy invariant, vanish on the basis, and satisfy $f(g^n) = n f(g)$ for all $g \in F$ and $n \in \mathbb{Z}$. These are not all quasimorphisms.

\begin{example}
Let $\mathcal{F}$ be a fundamental set, and suppose that it contains $ab$. Define $f : \mathcal{F} \to \mathbb{R}$ sending $ab$ to 1, $b^{-1}a^{-1}$ to $-1$, and everything else to 0. Extend $f$ to a map on all of $F$ that is conjugacy invariant, vanishes on the basis and such that $f(g^n) = nf(g)$ for all $g \in F$ and all $n \in \mathbb{Z}$. The resulting $f$ is, however, not a quasimorphism. For instance, $f((ab)^n) + f(a) - f((ab)^na) = f((ab)^n) = n$, which can be arbitrarily large.
\end{example}

The next two subsections will be devoted to trying to understand when a given coefficient map $\alpha$ gives rise to a quasimorphism.

\subsubsection{Continuity properties of the cup product}

We have already seen that the cup product is continuous with respect to the defect topology. However, a priori it is not continuous with respect to the pointwise topology. We cannot prove that it is definitely not continuous: in fact, such a proof would imply that $H^4_b(F) \neq 0$, which is a major open question.

But in this thesis we are merely interested in trivial cup products, and in that case we can say something.

\begin{lemma}
Let $G$ be a countable group, $f_n \in C^k_b(G)$. Suppose that $||f_n||_\infty \leq K$ for all $n \geq 1$. Then $(f_n)_{n \geq 1}$ admits a subsequence that converges pointwise to some $f \in C^k_b(G)$ with $||f||_\infty \leq K$.
\end{lemma}

\begin{proof}
Recall that $C^k_b(G)$ with its supremum norm $|| \cdot ||_\infty$ can be seen as the dual of $C_k(G)$ with its $\ell^1$ norm. Now $C_k(G)$ embeds densely into $\ell^1(G^k)$, which is a separable Banach space since $G$ is countable. Therefore $C^k_b(G)$ may also be seen as the dual of $\ell^1(G^k)$. By the Banach-Alaoglu theorem, closed balls in $C^k_b(G)$ are compact in the weak-$*$ topology. Since $C^k_b(G)$ is the dual of a separable space, so it is itself separable, compactness implies sequential compactness. Therefore, up to subsequence, $f_n$ converges in the weak-$*$ topology to some $f$ with $||f||_\infty \leq K$. Finally, weak-$*$ convergence for $C^k_b(G) = C_k(G)^*$ means pointwise convergence for $C^k_b(G) = \{ f : G^k \to \mathbb{R} : ||f||_\infty < \infty \}$.
\end{proof}

\begin{corollary}
Let $f_n \in B^k_b(G)$, and suppose that $f_n \to f \in C^k_b(G)$ pointwise. Suppose that there exists a $K > 0$ and $\pi_n \in C^{k-1}_b(G)$ such that $\delta^{k-1} \pi_n = f_n$ and $||\pi_n||_\infty \leq K$. Then $f \in B^k_b(G)$.
\end{corollary}

\begin{proof}
By the previous lemma, up to subsequence $\pi_n \to \pi$ pointwise, where $||\pi||_\infty \leq K$. Then $\delta^{k-1} \pi = f$ by pointwise convergence, and so $f$ has a bounded primitive.
\end{proof}

We record this corollary explicitly in the special case we are interested in:

\begin{corollary}
Let $\varphi, \psi \in Q(G)$. Let $(\varphi_n)_{n \geq 1}, (\psi_n)_{n \geq 1} \subset Q(G)$ be such that $\varphi_n \to \varphi, \psi_n \to \psi$ pointwise and $[\delta^1 \varphi_n] \smile [\delta^1 \psi_n] = 0 \in H^4_b(G)$ for all $n \geq 1$. Assume moreover that these cup products admit a sequence of uniformly bounded primitives. Then $[\delta^1 \varphi] \smile [\delta^1 \psi] = 0$.
\end{corollary}

\subsection{Calegari quasimorphisms}
\label{subsec_cal}

From now on, $\alpha$ will always denote a map $F \to \mathbb{R}$ that is alternating and supported on $\mathcal{N}$. We denote
$$\varphi_\alpha := \sum\limits_{w \in \mathcal{N}^+} \alpha_w h_w,$$
which is a well-defined alternating function from $F$ to $\mathbb{R}$. Notice that $\alpha_{w^{-1}} h_{w^{-1}} = \alpha_w h_w$, so this definition is independent of the choice of $\mathcal{N}^+$. \\

We will focus our attention on a family of infinite sums that give quasimorphisms:

\begin{definition}
Define
$$\kappa_\alpha(1) := \sup\limits \left( \sum\limits_{w \in C} |\alpha_w| \right),$$
where the supremum runs over all compatible families $C$. We say that $\varphi_\alpha$ is a \textit{Calegari quasimorphism} if $\kappa_\alpha(1) < \infty$. We denote by $Cal$ the space of Calegari quasimorphisms.
\end{definition}

Notice that since an element of length at least two forms a compatible family of size 1, $\alpha$ must be bounded in order for $\varphi_\alpha$ to be in $Cal$. \\

The notation $\kappa_\alpha(1)$ will be justified shortly. Recall that if $w \in C$, then $w^{-1} \notin C$. Also notice that the sum over an infinite compatible family is the supremum of the sum over its finite compatible subfamilies, so one can let the supremum run only over all \textit{finite} compatible families to compute $\kappa_\alpha(1)$. \\

As the name suggests, Calegari quasimorphisms are in fact quasimorphisms \cite[Proposition 2.34]{Calegari}:

\begin{proposition}[Calegari]
\label{cal}

Calegari quasimorphisms are quasimorphisms. More precisely, $\tilde{D}(\varphi_\alpha) \leq \kappa_\alpha(1)$, and so $D(\varphi_\alpha) \leq 3 \kappa_\alpha(1)$.
\end{proposition}

\begin{proof}
Let $u|v$ be a reduced expression. Using the formula for $\delta^1 h_w$ from Lemma \ref{formula_hw}, and the fact that $\alpha$ is alternating:
$$\delta^1 \varphi_\alpha(u, v) = \sum\limits_{w \in \mathcal{N}^+} \alpha_w \delta^1 h_w (u, v) = -\sum\limits_{w \in j(u|v)} \alpha_w.$$
The termwise application of the map $\delta^1$ is justified by the fact that for any $g \in F$, only finitely terms of the sum satisfy $h_w(g) \neq 0$.

Notice that $\{ w : w \in j(u|v) \}$ is a finite full compatible family, and all finite full compatible families are of this form. The result follows by taking the supremum over all reduced expressions.

The last inequality follows from Lemma \ref{rd}.
\end{proof}

This class of quasimorphisms does not a priori describe all quasimorphisms of the form $\varphi_\alpha$, where $\alpha$ is supported on $\mathcal{N}$. We will see in Subsection \ref{big_cal} how these two classes differ.

The rest of this subsection will be devoted to highlighting some subspaces of $Cal$, and showing how they sit inside each other. This is summarized in the following diagram, where arrows indicate inclusion:

\begin{center}
	\begin{tikzcd}
		& & Cal \\
		& \ell^1_{Ind} \arrow[ru] & & \kappa(c_0) \arrow[lu] \\
		\Sigma_{Ind} \arrow[ru] & & \ell^1_{Br} \arrow[lu] \arrow[ru] & & \kappa(\ell^1) \arrow[lu] \\
		& & & w\ell^1_{Br} \arrow[lu] \arrow[ru] \\
		& & \Sigma_{Br} \arrow[ru] \arrow[lluu]
	\end{tikzcd}
\end{center}

\subsubsection{$\Sigma_{Br}$, $w\ell^1_{Br}$ and $\ell^1_{Br}$}

These are the simplest examples of Calegari quasimorphisms:

\begin{definition}
Define $\Sigma_{Br}$ to be the space of finite sums of Brooks quasimorphisms on non-self-overlapping words, i.e., of those $\varphi_\alpha$ such that $\alpha$ has finite support. Define $\ell^1_{Br}$ to be the space of $\ell^1$ sums of Brooks quasimorphisms on non-self-overlapping words, i.e., of those $\varphi_\alpha$ such that $\alpha$ is an $\ell^1$ sequence. Finally, define $w\ell^1_{Br}$ to be the space of those $\varphi_\alpha$ such that
$$\sum\limits_{w \in \mathcal{N}^+} |w||\alpha_w| < \infty.$$
\end{definition}

\begin{lemma}
$\Sigma_{Br} \subset w \ell^1_{Br} \subset \ell^1_{Br} \subset Cal$.
\end{lemma}

\begin{proof}
The first two inclusions are clear. For the last one, given any compatible family $C$:
$$\sum\limits_{w \in C} |\alpha_w| \leq \sum\limits_{w \in \mathcal{N}} |\alpha_w|,$$
and the latter is bounded independently of $C$ if $\varphi_\alpha \in \ell^1_{Br}$.
\end{proof}

Grigorchuk had already noticed that elements of $\ell^1_{Br}$ are quasimorphisms \cite[Proposition 5.9]{Grigor}. However, these are far from being the only ones, as we will now see.

\subsubsection{$\Sigma_{Ind}$ and $\ell^1_{Ind}$}

In \cite[Proposition 5.10]{Grigor}, Grigorchuk proves that any bounded coefficient map $\alpha$ supported on $I := \{a^n(ab)^nb : n \geq 1 \}$ yields a quasimorphism. Recall from Example \ref{ex_ind} that this is an independent family, and that moreover $I^{\pm 1}$ is also independent. This can be generalized:

\begin{definition}
Define $\Sigma_{Ind}$ to be the space of those $\varphi_\alpha$ such that $\alpha = \alpha_1 + \cdots + \alpha_n$ and $\alpha_i$ is bounded, alternating, and supported on a symmetric independent family $I_i$. Define $\ell^1_{Ind}$ to be the space of those $\varphi_\alpha$ such that
$$\alpha = \sum\limits_{i \geq 1} \alpha_i; \,\,\,\,\,\,\,\,\,\, \sum\limits_{i \geq 1} ||\alpha_i||_\infty < \infty,$$
where $\alpha_i$ is as before. For $\varphi_\alpha \in \ell^1_{Ind}$, the infimum of the value $\sum ||\alpha_i||_\infty$ as in the definition is denoted by $SInd_{\alpha}$.
\end{definition}

\begin{lemma}
$\Sigma_{Ind} \subset \ell^1_{Ind} \subset Cal$. Moreover $\Sigma_{Br} \subset \Sigma_{Ind}$ and $\ell^1_{Br} \subset \ell^1_{Ind}$.
\end{lemma}

\begin{proof}
The only inclusion that is not clear is $\ell^1_{Ind} \subset Cal$. Let $\varphi_\alpha \in \ell^1_{Ind}$ where $\alpha = \sum \alpha_i$, and $\alpha_i$ is supported on the symmetric independent family $I_i$, as in the definition. Any two distinct elements of $I_i$ cannot overlap, so they cannot be in the same compatible family. Therefore any compatible family contains at most one element of each $I_i$. Thus, given any compatible family $C$:
$$\sum\limits_{w \in C} |\alpha_w| \leq \sum\limits_{i \geq 1} ||\alpha_i||_\infty,$$
and the latter is bounded independently of $C$ if $\varphi_\alpha \in \ell^1_{Ind}$.
\end{proof}

In Appendix \ref{og}, we will study $\Sigma_{Ind}$ with a graph-theoretic approach. In particular, we will provide the following recipe for constructing larger and larger elements of $\Sigma_{Ind}$ (Example \ref{ex_SInd}):

\begin{example}
Let $V$ be a finite disjoint union of symmetric independent families, and let $N \geq 1$. Let $W$ be the set of non-self-overlapping words that can be written as reduced products $xwy$, where $|x|, |y| \leq N$ and $w \in V$. Then for any bounded symmetric coefficient map $\alpha$ supported on $W$, $\varphi_\alpha \in \Sigma_{Ind}$.
\end{example}

We will also show that $\ell^1_{Ind}$, which at first sight is the largest family we have defined, is not the whole of $Cal$ (Example \ref{ex_l1Ind}):

\begin{example}
Let $F = F_3 = \langle a, b, c \rangle$, let $V = \{ (ab^nc)(ab^mc) : n > m \}$. Then $V \subset \mathcal{N}$; in fact, any word in $V$ is Lyndon with respect to any total ordering such that $a \prec b \prec c$. Let $\alpha$ be any bounded symmetric coefficient map such that $supp(\alpha) = V^{\pm 1}$. Then $\varphi_\alpha \in Cal \, \backslash \, \Sigma_{Ind}$. Suppose moreover that $\inf_{w \in supp(\alpha)} |\alpha_w| > 0$. Then $\varphi_\alpha \in Cal \, \backslash \, \ell^1_{Ind}$.
\end{example}

\subsubsection{The map $\kappa_{\alpha}$}

Given $\alpha$ as before, we can define more generally a map $\kappa_\alpha$ that takes into account the length of the words in a compatible family as well:
\begin{definition}
For $n \geq 0$, define
$$\kappa_\alpha(n) := \sup \left( \sum_{\substack{w \in C \\ |w| > n}} |\alpha_w| \right).$$
\end{definition}

This agrees with our definition of $\kappa_\alpha(1)$, since any word in a compatible family has length at least 2. Also notice that $\kappa_\alpha(0) = \kappa_\alpha(1)$.

Clearly $(\kappa_\alpha(n))_{n \geq 0}$ is a non-increasing sequence. We may formally allow this to take infinite values, but it would not tell us anything interesting:

\begin{lemma}
\label{kappa_fin}

Suppose that there exists some $N$ such that $\kappa_\alpha(N) < \infty$. Then $\kappa_\alpha(1) < \infty$, so $\kappa_\alpha(n)< \infty$ for all $n \geq 0$.
\end{lemma}

\begin{proof}
Let $C$ be a compatible family.
$$\sum\limits_{w \in C} |\alpha_w| \leq \sum\limits_{|w| \leq N} |\alpha_w| + \sum_{\substack{w \in C \\ |w| > N}} |\alpha_w| \leq \sum\limits_{|w| \leq N} |\alpha_w| + \kappa_\alpha(N).$$
Since the ball of radius $N$ is finite, this quantity is finite and bounded independently of $C$. The last claim follows from the fact that $\kappa_\alpha(n)$ is non-increasing, and $\kappa_\alpha(0) = \kappa_\alpha(1)$.
\end{proof}

\begin{definition}
Define $\kappa(c_0)$ to be the space of those $\varphi_\alpha$ such that $\kappa_\alpha(n) \xrightarrow{n \to \infty} 0$. Define $\kappa(\ell^1)$ to be the space of those $\varphi_\alpha$ such that $(\kappa_\alpha(n))_{n \geq 0}$ is an $\ell^1$ sequence, with $\sum_{n \geq 0} \kappa_\alpha(n) =: S \kappa_\alpha$.
\end{definition}

\begin{lemma}
$\kappa(\ell^1) \subset \kappa(c_0) \subset Cal$. Moreover $\ell^1_{Br} \subset \kappa(c_0)$ and $w \ell^1_{Br} \subset \kappa(\ell^1)$.
\end{lemma}

\begin{proof}
The first inclusion is clear, and the second one follows from Lemma \ref{kappa_fin}. For $\ell^1_{Br} \subset \kappa(c_0)$:
$$\sum_{\substack{w \in C \\ |w| > n}} |\alpha_w| \leq \sum\limits_{|w| > n} |\alpha_w|,$$
and the latter is independent of $C$ and goes to 0 as $n \to \infty$ if $\varphi_\alpha \in \ell^1_{Br}$.

For $w \ell^1_{Br} \subset \kappa(\ell^1)$, using also the estimate above:
$$\sum\limits_{n \geq 0} \kappa_\alpha(n) \leq \sum\limits_{n \geq 0} \sum_{\substack{w \in \mathcal{N}^+ \\ |w| > n}} |\alpha_w| = \sum\limits_{w \in \mathcal{N}^+} \sum\limits_{n < |w|} |\alpha_w| = \sum\limits_{w \in \mathcal{N}^+} |w||\alpha_w|,$$
and the latter is bounded independently of $C$ if $\varphi_\alpha \in w \ell^1_{Br}$.
\end{proof}

The reason why we introduced the space $w\ell^1_{Br}$ is because it is an easier to check condition than $\kappa(\ell^1)$ (which will actually appear in the statement of our main theorem): indeed, for $w \ell^1_{Br}$ we only need to estimate one sum, while for $\kappa(\ell^1)$ we need, at least in principle, to estimate infinitely many. Moreover, when generalizing part of our main theorem to big Brooks quasimorphisms on self-overlapping words in Subsection \ref{subsec_so}, it is the definition of space $w \ell^1_{Br}$ which is the easiest to adapt. \\

The importance of $\kappa(c_0)$ lies in the following result:

\begin{lemma}
\label{kc0}

$\Sigma_{Br}$ is defect-dense in $\kappa(c_0)$.
\end{lemma}

\begin{proof}
Let $\varphi_\alpha \in \kappa(c_0)$. We will show that $\sum_{|w| \leq N} \alpha_w h_w \xrightarrow{N \to \infty} \varphi_\alpha$ in the defect topology. Indeed, let $\alpha^N = \alpha \cdot \mathbbm{1} \{ |w| > N \}$, which is still bounded, alternating and supported on $\mathcal{N}$. Then $ \left( \varphi_\alpha - \sum_{|w| \leq N} \alpha_w h_w \right) = \varphi_{\alpha^N} \in Cal$. Moreover, $\kappa_{\alpha^N}(1) = \kappa_\alpha(N)$. By Proposition \ref{cal}:
$$D \left( \varphi_\alpha - \sum_{|w| \leq N} \alpha_w h_w \right) = D(\varphi_{\alpha^N}) \leq 3 \kappa_{\alpha^N}(1) = 3 \kappa_\alpha(N) \xrightarrow{N \to \infty} 0.$$
\end{proof}

\subsection{How big is the space $Cal$?}
\label{big_cal}

Our main theorem will concern the space $Cal$, and the subspaces that were defined in the previous subsection. In this subsection we try to understand how big this space is with respect to $H^2_b(F)$. We start by comparing $Cal$ to the space of quasimorphisms of the free group that can be written as a sum of Brooks quasimorphisms supported on non-self-overlapping words. Then, we show a result that suggests that looking at these quasimorphisms should be enough to describe the whole of $H^2_b(F)$.

\subsubsection{The space $\mathcal{N}_{Br}$}

We define $\mathcal{N}_{Br} := \{ \varphi_\alpha \in Q(F) : supp(\alpha) \subset \mathcal{N} \}$ with the same notation as in the previous subsection. By definition $Cal \subset \mathcal{N}_{Br}$, and here we want to see the difference between these two spaces. We start by proving an equivalent characterization of $Cal$ that allows to get rid of the absolute value on the coefficients:

\begin{lemma}
\label{cal_noabs}

$\varphi_\alpha \in Cal$ if and only if
$$\sup \left| \sum\limits_{w \in C} \alpha_w \right| < \infty;$$
where the supremum runs over all compatible families in $\mathcal{N}$.
\end{lemma}

\begin{proof}
The "only if" is clear. For the "if", fix $\varphi_\alpha$ and let $T$ be the finite quantity from the statement. Let $C$ be a compatible family. Write $C = C_+ \sqcup C_-$, where $C_+ := \{ w \in C : \alpha_w \geq 0 \}, C_- := \{ w \in C : \alpha_w < 0 \}$. Then $C_+$ and $C_-$ are also compatible families. So
$$\sum\limits_{w \in C} |\alpha_w| = \sum\limits_{w \in C_+} \alpha_w - \sum\limits_{w \in C_-} \alpha_w \leq 2T.$$
Taking the supremum over all compatible families, we obtain $\kappa_\alpha(1) \leq 2T < \infty$, so $\varphi_\alpha \in Cal$.
\end{proof}

For $\mathcal{N}_{Br}$ we have a similar characterization, but we only need to check full compatible families (Definition \ref{comp_def}).

\begin{lemma}
\label{N_Br}

$\varphi_\alpha \in \mathcal{N}_{Br}$ if and only if
$$\sup \left| \sum\limits_{w \in C} \alpha_w \right| < \infty;$$
where the supremum runs over all \underline{full} compatible families in $\mathcal{N}$.
\end{lemma}

\begin{proof}
It follows directly from the proof of Proposition \ref{cal} that the quantity in the statement equals $\tilde{D}(\varphi_\alpha)$.
\end{proof}

The only difference is that in the quantity that needs to be bounded for $\mathcal{N}_{Br}$ there can be large cancellations. The proof of Lemma \ref{cal_noabs} amounts to showing that in $Cal$, since the quantity needs to be bounded for all subfamilies of full compatible families, the cancellations cannot be arbitrarily large. \\

It should be possible to construct an example of a quasimorphism $\varphi_\alpha \in \mathcal{N}_{Br} \, \backslash \, Cal$. However, the more interesting question is the following:

\begin{question}
Is every $\varphi_\alpha \in \mathcal{N}_{Br}$ equivalent to some $\varphi_\beta \in Cal$?
\end{question}

The characterization of $\mathcal{N}_{Br}$ from Lemma \ref{N_Br} allows to prove one desirable property:

\begin{corollary}
\label{a_bdd}

Let $\varphi_\alpha \in \mathcal{N}_{Br}$. Then $\alpha$ is bounded.
\end{corollary}

\begin{proof}
Let $T \geq 0$ be the finite quantity in Lemma \ref{N_Br}. Let $w \in F$. We want to give a uniform bound for $|\alpha_w|$, so we can assume that $|w| \geq 3$. Then we can write $w = xvy$ as a reduced expression, where $x, y$ are in the basis. Let $C$ be the full compatible family given by the reduced expression $x|vy$, and $C'$ the one given by $x|v$. Then $C = C' \cup \{ w \}$. Thus:
$$|\alpha_w| = \left| \sum\limits_{u \in C} \alpha_u - \sum\limits_{u \in C'} \alpha_u \right| \leq 2T.$$
\end{proof}

\subsubsection{Termwise homogenization}

Grigorchuk's density theorem shows that $H^2_b(F)$ is described by infinite sums of \textit{homogenizations} of Brooks quasimorphisms supported on non-self-overlapping words: in fact, on a fundamental set $\mathcal{F}$. However, the combinatorics of Brooks quasimorphisms are simpler than those of their homogenizations, in particular in the framework of decompositions that will be dealt with in the next section. Therefore one could ask if the infinite span of $\{ h_w \}_{w \in \mathcal{F}^+}$ also represents the whole of $H^2_b(F)$; this would imply that $\mathcal{N}_{Br}$ represents the whole of $H^2_b(F)$. This would be true if any quasimorphism of the form $\sum_{w \in \mathcal{F}^+} \alpha_w \overline{h_w}$ were equivalent to $\sum_{w \in \mathcal{F}^+} \alpha_w h_w$. We are not able to prove this, but the following proposition is a partial result in that direction that seems to be of independent interest.

\begin{proposition}
\label{twhom}

Let
$$\varphi_\alpha := \sum\limits_{w \in \mathcal{N}^+} \alpha_w h_w; \,\,\,\,\,\, \tilde{\varphi}_\alpha := \sum\limits_{w \in \mathcal{N}^+} \alpha_w \overline{h_w}.$$
Suppose further that $\varphi_\alpha$ is a quasimorphism. Then $\tilde{\varphi}_\alpha$ is a quasimorphism and $\overline{\varphi_\alpha} = \tilde{\varphi}_\alpha$.
\end{proposition}

\begin{proof}
We need to show that $|| \varphi_\alpha - \tilde{\varphi}_\alpha ||_\infty < \infty$. We will bound this quantity in terms of $D(\varphi_\alpha)$, which is finite by hypothesis. Notice that $\tilde{\varphi}_\alpha$ is a sum of homogeneous quasimorphisms, so it is conjugacy-invariant, and similarly $\varphi_\alpha$ is alternating. Therefore, for all $a, g \in F$, we have
$$|\varphi_\alpha(aga^{-1}) - \tilde{\varphi}_\alpha(aga^{-1})| \leq |\varphi_\alpha(aga^{-1}) - \varphi_\alpha(g)| + |\varphi_\alpha(g) - \tilde{\varphi}_\alpha(g)| \leq 2 D(\varphi_\alpha) + |\varphi_\alpha(g) - \tilde{\varphi}_\alpha(g)|.$$
So we only need to bound $|\varphi_\alpha(g) - \tilde{\varphi}_\alpha(g)|$ when $g$ is cyclically reduced, in which case $g|g$ is a reduced expression and $|\overline{g}| = |g|$.

Now let $g$ be cyclically reduced. We calculate:
$$\varphi_\alpha(g) - \tilde{\varphi}_\alpha(g) = \sum\limits_{|w| \leq |g|} \alpha_w (h_w(g) - \overline{h_w}(g)).$$
Since $g$ is cyclically reduced, every occurrence of $w^{\pm 1}$ in $g$ is also an occurrence of $w^{\pm 1}$ in $\overline{g}$. So, for $|w| \leq |g|$, $h_w(g) - \overline{h_w}(g) \neq 0$ only if $w$ intersects the juncture of $g|g$, that is, only if $\delta^1 h_w(g, g) \neq 0$. Moreover, since we are only considering $|w| \leq |g|$, if $w$ intersects the juncture of $g|g$, then this copy of $w$ is also contained in $\overline{g}$. It follows from Lemma \ref{formula_hw} that $h_w(g) - \overline{h_w}(g) = \delta^1 h_w(g, g)$. Thus:
$$\varphi_\alpha(g) - \tilde{\varphi}_\alpha(g) = \sum\limits_{|w| \leq |g|} \alpha_w \delta^1 h_w(g, g) = \delta^1 \varphi (g, g) - \sum\limits_{|g| < |w| \leq 2 |g|} \alpha_w h_w(g^2).$$

We are left to show that the last term is finite. In fact, we will show that for any $|g| < |w| \leq 2|g|$, we have $h_w(g^2) = 0$. Since $|g| < |w| \leq 2|g|$ any two occurrences of $w^{\pm 1}$ in $g^2$ will overlap. Since $w$ does not overlap $w$ or $w^{-1}$, we have at most one occurrence of $w$ or one occurrence of $w^{-1}$. Suppose that there is indeed one occurrence, and without loss of generality assume that $w$ occurs in $g^2$. Let $g = s_1 \cdots s_k$ be the canonical reduced expression. The fact that $w$ occurs in $g^2$ means that $w = s_i \cdots s_k s_1 \cdots s_j$ for some $1 \leq i, j \leq k$. Since $|w| > |g|$, $i < j$. But then $w = (s_i \cdots s_j) \cdot (s_{j + 1} \cdots s_{i - 1}) \cdot (s_i \cdots s_j)$. We are assuming that $w$ is non-self-overlapping, so this is not possible. Therefore $h_w(g^2) = 0$, which concludes the proof.
\end{proof}

\begin{question}
Let $\varphi_\alpha, \tilde{\varphi}_\alpha$ be as above. Suppose that $\tilde{\varphi}_\alpha$ is a quasimorphism. Is then $\varphi_\alpha$ a quasimorphism? What if we restrict $supp(\alpha)$ to a (well-chosen) fundamental set?
\end{question}

Notice that if the answer to the above is affirmative, it follows from Proposition \ref{twhom} that $\overline{\varphi}_\alpha = \tilde{\varphi}_\alpha$.

\pagebreak

\section{Decompositions of the free group}
\label{secdec}

In the paper \cite{Heuer}, Heuer proves that the cup product $[\delta \varphi] \smile [\delta \psi]$ vanishes if $\varphi, \psi$ are either Brooks quasimorphisms on non-self-overlapping words or Rolli quasimorphisms. These two kinds of quasimorphisms are quite different at a first glance, but there is a way to deal with them together, which is to introduce the notion of \textit{decomposition} of a free group. In the first two subsections, we will go through Heuer's definitions and theorem, also giving an explicit bound on the primitive of the cup product that is not present in the original paper. In the third subsection, we will see how some of the subspaces of $Cal$ that we introduced previously can also be treated in this framework. With these tools, we are finally able to prove the main theorem. The last section uses the same techniques to prove some weaker vanishing results for cup products with big Brooks quasimorphisms on general (i.e., not necessarily non-self-overlapping) words. \\

Throughout, $F$ is a non-abelian free group of finite rank with a fixed basis $S = \{s_1, \ldots, s_n\}$. We refer the reader to Subsection \ref{not} for the notation regarding finite sequences.

\subsection{Definitions and first examples}

\subsubsection{Decompositions}

Intuitively, a decomposition is a way of writing out elements of the free group as reduced words in a non-standard alphabet, in such a way that the concatenation of the decompositions of $g$ and $h$ does not differ too much from the decomposition of $gh$ around the center of the tripod with endpoints $1, g, gh$.

\begin{definition}
\label{Dec}

Let $\mathcal{P} \subset F$ be a (possibly infinite) symmetric set of elements of $F$ called \textit{pieces}, such that $1 \notin \mathcal{P}$. A \textit{decomposition} of $F$ with pieces $\mathcal{P}$ is a map $\Delta: F \to \mathcal{P}^*$ assigning to $1 \in F$ the empty sequence, and to every element $1 \neq g \in F$ a finite sequence $\Delta(g) = (g_1, \ldots, g_k)$ with $g_j \in \mathcal{P}$ such that:

\begin{enumerate}

\item For every $1 \neq g \in F$ with $\Delta(g) = (g_1, \ldots, g_k)$ we have $g = g_1 \cdots g_k$ as a reduced product (no cancellation). This implies in particular that $S^{\pm 1} \subseteq \mathcal{P}$.

\item $\Delta(g^{-1}) = \Delta(g)^{-1}$ for all $g \in F$.

\item For every $1 \neq g \in F$ with $\Delta(g) = (g_1, \ldots, g_k)$ we have $\Delta(g_i \cdots g_j) = (g_i, \ldots, g_j)$ for $1 \leq i \leq j \leq k$. We refer to this property as $\Delta$ being \textit{infix closed}. This implies in particular that if $g \in \mathcal{P}$, then either $g$ never appears in the decomposition of an element, or $\Delta(g) = (g)$.

\item Draw out the \textit{$\Delta$-triangle} for $g, h$, i.e., the one given by $\Delta(g) \cup \Delta(h) \cup \Delta(gh)$ like in Figure \ref{Dec_fig}. Highlight the subsequences $c_1^{-1}$, which is the largest one both at the beginning of $\Delta(g)$ and of $\Delta(gh)$; $c_2$, which is the largest one both at the end of $\Delta(g)$ and of $\Delta(h^{-1})$; and $c_3$, which is the largest one both at the end of $\Delta(h)$ and of $\Delta(gh)$. What remains is three strings $r_i$ forming a relation around the center: $\underline{r_1} \underline{r_2} \underline{r_3} = 1$. We call these the \textit{$c$-part}, respectively the \textit{$r$-part}, of this $\Delta$-triangle. For $\Delta$ to be a decomposition, we ask that $|r_i| \leq D$ (where $|r_i|$ indicates the length of the sequence, and not the word length of $\underline{r_i}$) for some constant $D \geq 0$.

\end{enumerate}

The minimal $D$ from the fourth condition is called the \textit{defect} of the decomposition, and is denoted by $D(\Delta)$.

\end{definition}

\begin{figure}
  \centering
  \includegraphics[width=7cm]{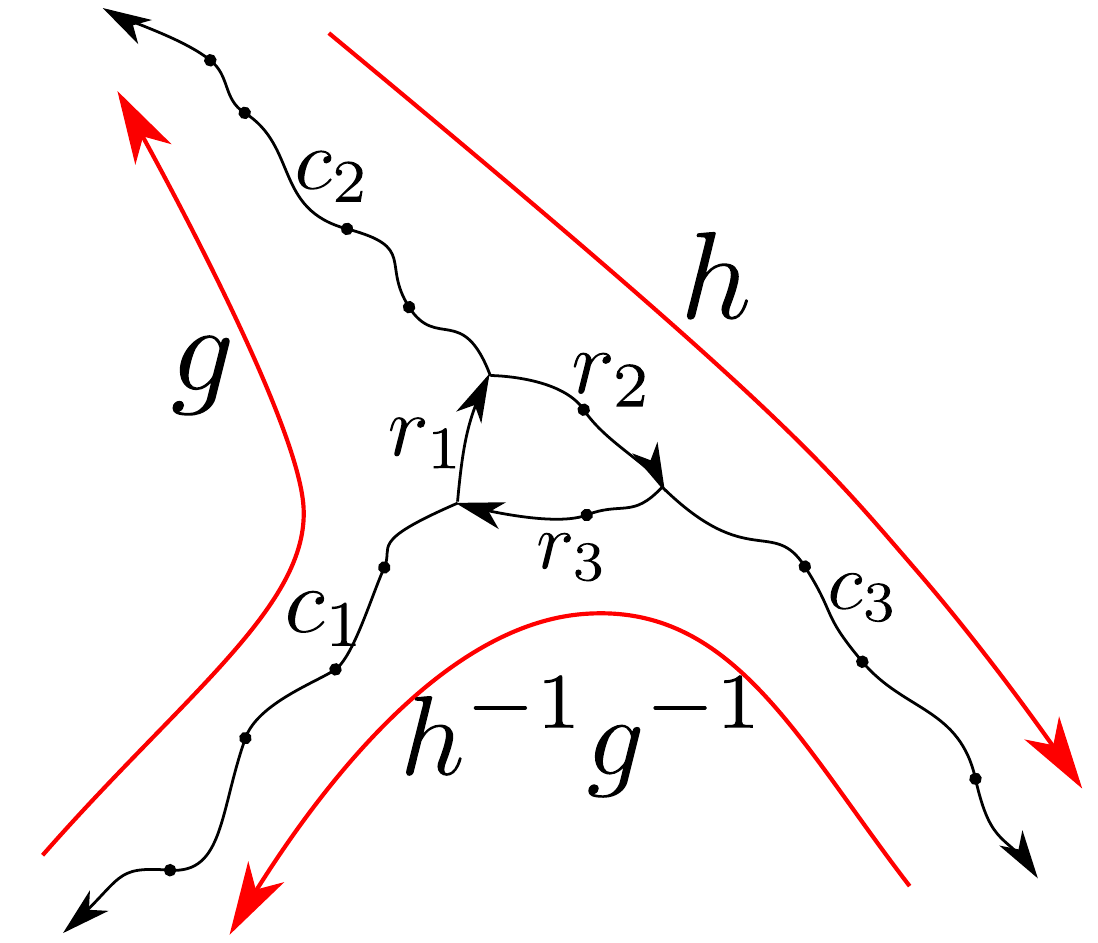}
  \caption{The $r$-part has to be bounded.}
  \label{Dec_fig}
\end{figure}

Here are some examples of decompositions:

\begin{example}
\label{Dec_ex}

1. We can write out the elements of $F$ as usual, with pieces $\mathcal{P}_{triv} = S^{\pm 1}$, giving rise to a trivial decomposition $\Delta_{triv}$. The defect is 0. \\

2. Another trivial example is by taking $\mathcal{P}_{b} = F \, \backslash \, \{ 1 \}$ and $\Delta(g) = (g)$ for all $1 \neq g \in F$. Let us denote this decomposition by $\Delta_b$. The defect is 1. \\

3. Given a non-self-overlapping word $w \in F$, we can write any element of $F$ by isolating all the (disjoint) occurrences of $w$ and $w^{-1}$. To do this we use the pieces $\mathcal{P}_w = \{w^{\pm 1} \} \cup \{ \text{words not containing } w^{\pm 1} \text{ as a subword} \}$. This yields the decomposition $\Delta_w$. The defect is 3: the worst-case scenario is of the form $r_1 = (xw_1)$, $r_2 = (w_2y)$, where $w_1w_2 = w$, and $r_3^{-1} = (x, w, y)$. \\

4. We can write out the elements of $F$ as products of powers of generators. For this we need the pieces $\mathcal{P}_{Rolli} = \{ s^m : s \in S, m \in \mathbb{Z} \, \backslash \, \{ 0 \} \}$, giving the decomposition $\Delta_{Rolli}$. The defect is 1: the worst-case scenario is of the form $r_1 = (s^n)$, $r_2 = (s^m)$, where $m \neq -n$, and $r_3^{-1} = (s^{n + m})$.

\end{example}

Examples 3. and 4. are the ones that allow to prove the main theorem in Heuer's paper. We remark that this concept was already hinted at in earlier work: the decomposition $\Delta_w$ was called \textit{w-decomposition} in Faiziev's paper \cite{Faziev}, and the decomposition $\Delta_{Rolli}$ was called \textit{factorization} in Rolli's paper \cite{Rolli}.

\begin{lemma}
\label{D_triv}

Let $\Delta$ be a decomposition with pieces $\mathcal{P}$, and suppose without loss of generality that each $p \in \mathcal{P}$ appears in the decomposition of some element. Then the following are equivalent:
\begin{enumerate}
\item $\mathcal{P} = S^{\pm 1}$.
\item $D(\Delta) = 0$.
\item $\Delta = \Delta_{triv}$.
\end{enumerate}
\end{lemma}

\begin{proof}
Let $\Delta$ be such a decomposition. Then $\Delta(p) = (p)$ for every $p \in \mathcal{P}$ by point 3. of Definition \ref{Dec}. Moreover $S^{\pm 1} \subseteq \mathcal{P}$ by point 2. of Definition \ref{Dec}. If $\mathcal{P} = S^{\pm 1}$, then $\Delta = \Delta_{triv}$, since there is a unique way to write elements of the free group as reduced products in the basis $S$ (this is the definition of a basis). This gives $1. \Rightarrow 3$. Else there exists $p \in \mathcal{P}$ with $p = s_1 \cdots s_n$ and $n > 1$. Then the $\Delta$-triangle for $s_1, s_2 \cdots s_n$ has $r_1 = (s_1)$, $r_2 = \Delta(s_2 \cdots s_n)$ and $r_3^{-1} = (p)$. So $D(\Delta) \geq 1$. This gives $2. \Rightarrow 1$. Finally $3. \Rightarrow 2.$ follows from Example \ref{Dec_ex}.
\end{proof}

\subsubsection{$\Delta$-Decomposable quasimorphisms}

Now we use decompositions to construct quasimorphisms:

\begin{definition}
\label{D-dec}

Let $\Delta$ be a decomposition with pieces $\mathcal{P}$ and $\lambda \in \ell^\infty_{alt}(\mathcal{P})$, i.e., $\lambda : \mathcal{P} \to \mathbb{R}$ is bounded and $\lambda(p^{-1}) = - \lambda(p)$ for any $p \in \mathcal{P}$. Then the map
$$\phi_{\lambda, \Delta} : F \to \mathbb{R} : g \mapsto \sum\limits_{j = 1}^k \lambda(g_j),$$
where $\Delta(g) = (g_1, \ldots, g_k)$, is called a \textit{$\Delta$-decomposable} quasimorphism.

\end{definition}

Notice that since $\lambda$ is alternating, $\phi_{\lambda, \Delta}$ is alternating by condition 2. of Definition \ref{Dec}.

\begin{proposition}
\label{D-dec-prop}

$\Delta$-decomposable quasimorphisms are alternating quasimorphisms. The value $\delta^1 \phi_{\lambda, \Delta} (g, h)$ only depends on the $r$-part of the corresponding $\Delta$-triangle, and $D(\phi_{\lambda, \Delta}) \leq 3 ||\lambda||_\infty D(\Delta).$
\end{proposition}

\begin{proof}
Let $c_i, r_i$ denote the $c$-part and the $r$-part of the $\Delta$-triangle for $g, h$, as in Figure \ref{Dec_fig}. We thus have the following reduced products:
$$g = \underline{c_1}^{-1} \underline{r_1} \underline{c_2}; \,\,\, h = \underline{c_2}^{-1} \underline{r_2} \underline{c_3}; \,\,\, gh = \underline{c_1}^{-1} \underline{r_3}^{-1} \underline{c_3}.$$
By conditions 2. and 3. of Definition \ref{Dec}, these products correspond to concatenations of the corresponding decompositions:
$\Delta(g) = \Delta(\underline{c_1}^{-1}) \cdot \Delta(\underline{r_1}) \cdot \Delta(\underline{c_2})$ and so on. Thus, using the definition of $\phi_{\lambda, \Delta}$ and the fact that $\phi_{\lambda, \Delta}$ is alternating,
$$\delta^1 \phi_{\lambda, \Delta}(g, h) = \phi_{\lambda, \Delta}(g) + \phi_{\lambda, \Delta}(h) - \phi_{\lambda, \Delta}(gh) = (-\phi_{\lambda, \Delta}(\underline{c_1}) + \phi_{\lambda, \Delta}(\underline{r_1}) + \phi_{\lambda, \Delta}(\underline{c_2})) + $$
$$+ (-\phi_{\lambda, \Delta}(\underline{c_2}) + \phi_{\lambda, \Delta}(\underline{r_2}) + \phi_{\lambda, \Delta}(\underline{c_3}))- (-\phi_{\lambda, \Delta}(\underline{c_1}) - \phi_{\lambda, \Delta}(\underline{r_3}) + \phi_{\lambda, \Delta}(\underline{c_3})) + = $$
$$ = \phi_{\lambda, \Delta}(\underline{r_1}) + \phi_{\lambda, \Delta}(\underline{r_2}) + \phi_{\lambda, \Delta}(\underline{r_3}).$$
This shows that the value only depends on the $r$-part. Moreover:
$$|\phi_{\lambda, \Delta}(\underline{r_i})| \leq ||\lambda||_\infty |r_i| \leq ||\lambda||_\infty D(\Delta),$$
and the bound on defect follows.
\end{proof}

\begin{example}
\label{D-dec_ex}

Going back to the decompositions from Example \ref{Dec_ex}, we can find familiar examples of quasimorphisms of the free group. \\

1. A quasimorphism is $\Delta_{triv}$-decomposable if and only if it is a homomorphism. \\

2. A quasimorphism is $\Delta_b$-decomposable if and only if it is in $\ell^\infty_{alt}(F)$. \\

3. Define $\lambda : \mathcal{P}_w \to \mathbb{R} : w^{\pm 1} \mapsto \pm 1; \, w^{\pm 1} \neq x \mapsto 0$. Then the $\Delta_w$-decomposable quasimorphism $\phi_{\lambda, \Delta_w}$ is precisely the Brooks quasimorphism $h_w$. \\

4. Let $S = \{s_1, \ldots, s_n\}$ and pick $\lambda_1, \ldots, \lambda_n \in \ell^\infty_{alt}(\mathbb{Z})$. Define $\lambda : \mathcal{P}_{Rolli} \to \mathbb{R} : s_i^m \mapsto \lambda_i(m)$. The $\Delta_{Rolli}$-decomposable quasimorphisms arising this way are precisely the Rolli quasimorphisms.

\end{example}

\begin{remark}

Given a decomposition $\Delta$, the set of $\Delta$-decomposable quasimorphisms is a subspace of the space of alternating quasimorphisms.

\end{remark}

\subsubsection{$\Delta$-continuous quasimorphisms}

The last concept we need to introduce is that of $\Delta$-continuous quasimorphisms. Recall that the coboundary of a $\Delta$-decomposable quasimorphism only depends on the $r$-part of the corresponding $\Delta$-triangle. The notion of $\Delta$-continuous quasimorphisms generalizes this: they are quasimorphisms that depend mostly on a small neighbourhood of the $r$-part. The first step is to formalize what we mean by "small":

\begin{definition}
\label{N-D-cont}

Let $\Delta$ be a decomposition. We define the function $N_\Delta : F^2 \times F^2 \to \mathbb{Z}_{\geq 0} \cup \{ \infty \}$ as follows. Let $(g, h), (g', h') \in F^2$, and draw the corresponding $\Delta$-triangles. If they do not have the same $r$-part, then $N_{\Delta}((g, h), (g', h')) := 0$. Otherwise we can draw the two corresponding $\Delta$-triangles on top of each other like in Figure \ref{D-cont_fig}. Then for each $i = 1, 2, 3$, the sequences $c_i, c_i'$ travel together for $N_i$ steps, and then split (the steps are seen as elements of the sequence, not as generators of $F$). If $c_i = c_i'$, we let $N_i := \infty$ (even though the sequences $c_i = c_i'$ are finite). We then define $N_{\Delta}((g, h), (g', h')) := \min(N_1, N_2, N_3)$.

\end{definition}

\begin{remark}
It follows from the definition that $N_\Delta((g, h), (g', h')) = \infty$ if and only if $(g, h) = (g', h')$.
\end{remark}

\begin{figure}
  \centering
  \includegraphics[width=7cm]{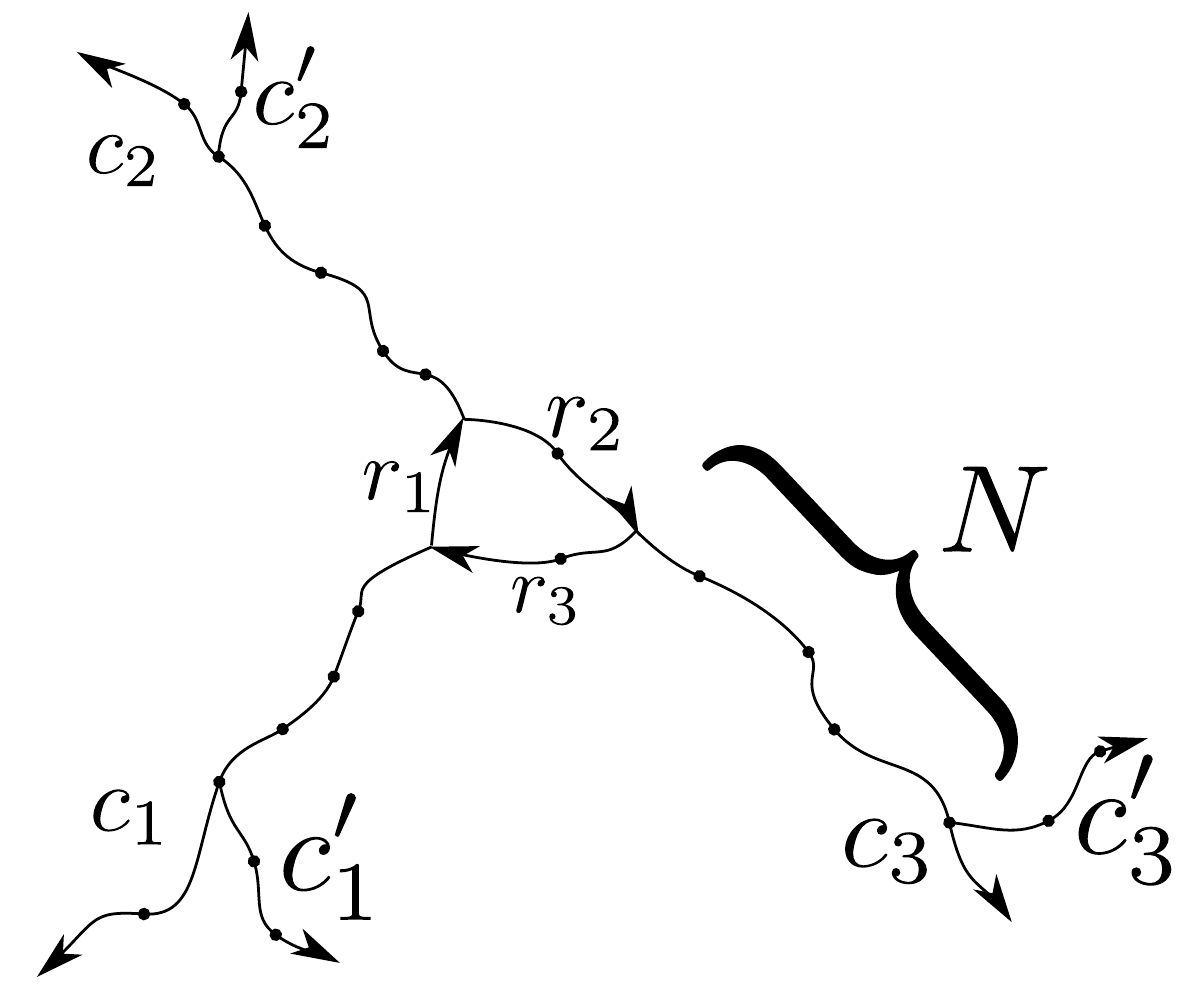}
  \caption{$N_\Delta$ would be the minimal $N$ that fits in all three branches.}
  \label{D-cont_fig}
\end{figure}

Notice that this function takes into account the "absolute time" during which the $c_i$ travel together, but \textit{not} the "relative time". That is, in Figure \ref{D-cont_fig}, the two separate ends of $c_i, c_i'$ could be very large compared to $N$, but $N_\Delta$ will not record that information.

\begin{definition}
\label{D-cont}

Let $\varphi \in Q_{alt}(F)$. We say that $\varphi$ is \textit{$\Delta$-continuous} if there exists a non-negative $\ell^1$ sequence $(x_j)_{j \geq 0}$ such that for all $(g, h) \neq (g', h') \in F^2$, we have:
$$|\delta^1 \varphi(g, h) - \delta^1 \varphi(g', h')| \leq x_N; \hspace{1cm} N = N_\Delta((g, h), (g', h')).$$
The infimum of the value $\sum_{j \geq 0} x_j$ over all sequences satisfying the above condition is denoted by $S_{\varphi, \Delta}$.
\end{definition}

One should interpret the condition of $(x_j)_{j \geq 0}$ being $\ell^1$ as very fast convergence to 0, so the values of $\delta^1 \varphi(g, h)$ and $\delta^1 \varphi(g', h')$ get very close when the corresponding $\Delta$-triangles have the same $r$-part and a large overlap around it. Of course the precise condition is used explicitly in the proof of Heuer's Theorem.

\begin{example}
\label{Db_cont}
Any alternating quasimorphism is $\Delta_b$-continuous. Indeed, any $\Delta_b$-triangle consists of only the $r$-part, so an alternating quasimorphism $\varphi$ is $\Delta_b$-continuous using the sequence $(x_j)_{j \geq 0}$, where $x_0 = 2 D(\varphi)$ and $x_j = 0$ for all $j \geq 1$.
\end{example}

We have two more main examples, one of which we have already mentioned:

\begin{proposition}
\label{D-Dec-D-cont}

If $\varphi$ is a $\Delta$-decomposable quasimorphisms, then it is $\Delta$-continuous, and $S_{\varphi, \Delta} \leq 2 D(\varphi)$. A Brooks quasimorphism on a non-self-overlapping word is $\Delta$-continuous for \underline{any} decomposition $\Delta$, and $S_{h_w, \Delta} \leq 6 |w|$
\end{proposition}

\begin{remark}
The second statement will be proven in greater generality in the third subsection.
\end{remark}

\begin{proof}
The first statement follows from the fact that if $\varphi$ is $\Delta$-decomposable, then $\delta^1 \varphi(g, h)$ only depends on the $r$-part of the corresponding $\Delta$-triangle. So one can choose $x_0 = 2 D(\varphi)$ and $x_j = 0$ for all $j \geq 1$. So $S_{\varphi, \Delta} \leq 2D(\varphi)$.

For the second statement, given a non-self-overlapping word $w$ and the corresponding Brooks quasimorphism $h_w$, the value of $\delta^1 h_w (g, h)$ only depends on the $|w|$-neighbourhood of the center of the tripod with endpoints $1, g, gh$, by Proposition \ref{bigsmall}. In particular, if $N_\Delta((g, h), (g', h')) \geq |w|$, then $\delta^1 h_w(g, h) = \delta^1 h_w (g', h')$. So we can define $x_j = 2 D(h_w) \leq 6$ for $0 \leq j < |w|$ and $x_j = 0$ for $j \geq |w|$. So $S_{h_w, \Delta} \leq 6 |w|$.
\end{proof}

\begin{remark}

Given a decomposition $\Delta$, the set of $\Delta$-continuous quasimorphisms is a subspace of the space of alternating quasimorphisms.

\end{remark}

\subsection{Heuer's Theorem}

We can finally state Heuer's Theorem:

\begin{theorem}[Heuer]
\label{Heuer_B}

Let $\varphi, \psi : F \to \mathbb{R}$ be quasimorphisms. Let $\Delta$ be a decomposition such that $\varphi$ is $\Delta$-decomposable and $\psi$ is $\Delta$-continuous. Then $[\delta \varphi] \smile [\delta \psi] = 0 \in H^4_b(F)$.
\end{theorem}

We will not prove the theorem here: the reader is referred to the paper \cite{Heuer} for the proof. Combining this with Proposition \ref{D-Dec-D-cont}, we get:

\begin{corollary}[Heuer]
\label{Heuer_A}

Let $\varphi, \psi : F \to \mathbb{R}$ be quasimorphisms which are either Brooks quasimorphisms on non-self-overlapping words or Rolli quasimorphisms. Then $[\delta \varphi] \smile [\delta \psi] = 0$.
\end{corollary}

\begin{proof}
Such Brooks and Rolli quasimorphisms are decomposable for $\Delta_{Rolli}$ or $\Delta_w$ for some $w \in \mathcal{N}$. Moreover Rolli quasimorphisms are $\Delta_{Rolli}$-continuous, and Brooks quasimorphisms are $\Delta$-continuous for any $\Delta$.
\end{proof}

\subsubsection{Explicit bound on the primitive}
\label{bound}

Heuer does not give an explicit bound for a primitive of this cup product in his paper. That is, for a map $\pi \in C^3_b(F)$ such that $\delta^3 \pi = \delta^1 \varphi \smile \delta^1 \psi$. We give one here, by making reference to the proof of Theorem C in \cite{Heuer}. It is not really important to give a precise bound, but a rough one will be useful when considering $\ell^1$ sums of quasimorphisms. We will use the same notation as in the paper (which differs slightly from ours), so that the reader can follow the proof more easily. \\

Let $\Delta \neq \Delta_{triv}$ be a decomposition with pieces $\mathcal{P}$, $\phi = \phi_{\lambda, \Delta}$ a $\Delta$-decomposable quasimorphism, $\omega$ a $\Delta$-continuous 2-cocycle (i.e., the coboundary of a $\Delta$-continuous quasimorphism). Let $R$ be the defect of $\Delta$; we are assuming $\Delta \neq \Delta_{triv}$ and so $R \neq 0$ by Lemma \ref{D_triv}. Let $S_{\omega, \Delta}$ the sum of the $\ell^1$ sequence in the definition of $\Delta$-continuity for $\omega$. Let us further define:
$$C_{\lambda, \omega, \Delta} := ||\lambda||_\infty(R ||\omega||_\infty + S_{\omega, \Delta}).$$

The key point is Proposition 4.5, since all estimates start from there. Looking at all three items of the proof, we see that the "uniformly bounded" translates to $"\leq 3 ||\lambda||_\infty R ||\omega||_\infty"$ for Item 1, and to $"\leq 2 ||\lambda||_\infty R ||\omega||_\infty + ||\lambda||_\infty S_{\omega, \Delta}"$ for Items 2 and 3. In all cases, the expressions are bounded by $3C_{\lambda, \omega, \Delta}$. \\

In Proposition 4.8, 4.5 is used three times; in Proposition 4.9, 4.5 is used once in the first two cases, and five times in the last case. In Proposition 4.10, the bound on $D(\theta)$ is given by applying Claim 4.11 five times, and Claim 4.11 is proven by applying 4.5 six times. Finally, the uniform bound of Lemma 4.13 is linear in terms of the defect of the quasimorphism $\rho$. When applying it in the end of the proof, we choose $\rho = \theta$ so once again we get a linear bound in terms of $C_{\lambda, \omega, \Delta}$.

We conclude that there exists a universal constant $N$ such that $||\pi||_\infty \leq N C_{\lambda, \omega, \Delta}.$ \\

Putting this back into our notation, we have proven that:

\begin{lemma}
\label{lbound}

Let $\Delta \neq \Delta_{triv}$ be a decomposition, $\phi_{\lambda, \Delta}$ a $\Delta$-decomposable quasimorphism and $\psi$ a $\Delta$-continuous one. Define
$$C_{\lambda, \psi, \Delta} := ||\lambda||_\infty(D(\Delta) D(\psi) + S_{\psi, \Delta}).$$
Then there exists a map $\pi \in C^3_b(F)$ such that $\delta^3 \pi = \delta^1 \phi_{\lambda, \Delta} \smile \delta^1 \psi$ and $$||\pi||_\infty = \mathcal{O}( C_{\lambda, \psi, \Delta}),$$
where the implied constant is universal
\end{lemma}

\begin{example}
Let $w, v \in \mathcal{N}$, and let $\phi_\lambda, \phi_\mu$ be the Rolli quasimorphisms defined by $\lambda, \mu \in \ell^\infty_{alt}(\mathbb{Z})^n$. Then we get the following bounds for the primitives of the cup products from Corollary \ref{Heuer_A}:
\begin{enumerate}
\item For $\delta^3 \pi = \delta^1 h_w \smile \delta^1 h_w$: $||\pi||_\infty << 1 \cdot (D(\Delta_w) \cdot D(h_w) + S_{h_w, \Delta_w}) \leq (3 \cdot 3 + 6) << 1$.

\item For $\delta^3 \pi = \delta^1 h_w \smile \delta^1 h_v$: $||\pi||_\infty << 1 \cdot (D(\Delta_w) \cdot D(h_v) + S_{h_v, \Delta_w}) \leq (3 \cdot 3 + 6|v|) << |v|$.

\item For $\delta^3 \pi = \delta^1 \phi_\lambda \smile \delta^1 h_v$: $||\pi||_\infty << ||\lambda||_\infty \cdot (D(\Delta_{Rolli}) \cdot D(h_v) + S_{h_v, \Delta_{Rolli}}) \leq ||\lambda||_\infty \cdot (1 \cdot 3 + 6|v|) << ||\lambda||_\infty \cdot |v|$.

\item For $\delta^3 \pi = \delta^1 \phi_\lambda \smile \delta^1 \phi_\mu$: $||\pi||_\infty << ||\lambda||_\infty \cdot (D(\Delta_{Rolli}) \cdot D(\phi_\mu) + S_{\phi_\mu, \Delta_{Rolli}}) \leq ||\lambda||_\infty \cdot (1 \cdot 3 ||\mu||_\infty + 2 \cdot 3 ||\mu||_\infty) << ||\lambda||_\infty \cdot ||\mu||_\infty.$
\end{enumerate}
\end{example}

\subsection{More examples}

The core part of the prof of our main theorem is to provide more examples of decompositions and decomposable and continuous quasimorphisms. By applying Heuer's Theorem \ref{Heuer_B}, we will immediately obtain triviality of the corresponding cup products.

\subsubsection{Decompositions $\Delta_I$}

We would like to define a decomposition that admits sums of Brooks quasimorphisms as decomposable quasimorphisms. The decomposition $\Delta_w$ works by isolating all occurrences of $w$, and gives rise to the decomposable quasimorphism $h_w$. Now if we have two non-self-overlapping words $v, w$ and we want to similarly define a decomposition giving rise to $h_v + h_w$, we run into trouble if $v$ and $w$ overlap, since then it is not clear how one can isolate occurrences of both $v$ and $w$. This motivates us to look at independent families.

\begin{definition}
Let $I$ be an independent symmetric family of words. Define $\Delta_I$ by having as pieces $\mathcal{P}_I := I \cup \{ \text{words not containing words in } I \text{ as subwords} \}$; and by writing out words isolating all occurrences of elements of $I$.
\end{definition}

Since $I$ is symmetric and independent, $\Delta_I$ is well-defined. Moreover, it is a decomposition:

\begin{lemma}
\label{D_I}

$\Delta_I$ is a decomposition with defect at most 5.

\end{lemma}

\begin{proof}
Conditions 1., 2. and 3. are clear. Let us focus on condition 4. Let $g, h \in F$ and consider the corresponding $\Delta_I$-triangle. Since we want to bound the $r$-part, we may assume that $c_1 = c_3 = \emptyset$: we can do this because $\Delta_I$ is infix-closed. Therefore we are in the case $\Delta(g) = (g_1, \ldots, g_k)$, $\Delta(h) = (h_1, \ldots, h_l)$ and $\Delta(gh) = (k_1, \ldots, k_m)$, with $g_1 \neq k_1$, $h_l \neq k_m$ and $r_3^{-1} = \Delta(gh)$. This means that when writing $gh = g_1 \cdots g_k h_1 \cdots h_l$, after simplifying and grouping the elements of $I$, we obtain $k_1 \cdots k_m$. In particular the first and last term must change. But for $g_1$ to be different than $k_1$, there are only three possible cases:

\begin{enumerate}

\item All of $g$ gets simplified. In this case $g = \underline{c_2}$ and so $r_1 = \emptyset$.

\item All terms from $g_2$ to $g_k$ get simplified, as well as part of $g_1$, so that $k_1$ is a proper subword of $g_1$, and $r_1 = (g_1)$.

\item All terms from $g_3$ to $g_k$ get simplified, as well as part of $g_2$, but $g_1 \notin I$, so $g_2 \in I$ (since otherwise the product $g_1 g_2$ would be a single term in the sequence) and $g_1$ is a proper subword of $k_1$ (the latter can be in $I$ or not). So $r_1 = (g_1, g_2)$.

\end{enumerate}

In all cases, $|r_1| \leq 2$. The situation with $h$ is analogous, so $|r_2| \leq 2$. Up to removing $c_2$, again because $\Delta_I$ is infix-closed, this allows us to reduce to the case $k, l \leq 2$. Therefore we only need to bound $m$. \\

Let us write $g = t_1 d$, $h = d^{-1} t_2$ and $gh = t_1 t_2$ as reduced products. We rewrite the sequence for $gh$ as $\Delta(gh) = (k_1, \ldots, k_m, k, k_{m+1}, \ldots, k_{m + n})$, where $k_1 \cdots k_m$ is a subword of $t_1$, $k$ overlaps the juncture of $t_1 | t_2$, and $k_{m + 1} \cdots k_{m + n}$ is a subword of $t_2$ (some of these are allowed be empty). We will show that $m \leq 2$, and similarly $n \leq 2$, which allows us to conclude that $|r_3| = |\Delta(gh)| \leq 5$. Once again we consider the three cases from before. Since $c_2$ has been removed, each case simplifies further.

\begin{enumerate}
\item In this case $g = \emptyset$, so $m = 0$.

\item In this case $g = g_1$. Therefore, for $i \leq m$, $k_i$ is a subword of $t_1$, which is a subword of $g = g_1$. But if $m \geq 2$, then both $k_1$ and $k_2$ are proper subwords of $g_1$, and either $k_1$ or $k_2$ must belong to $I$ (otherwise their product would be a single term in the sequence). If $g_1 \notin I$, then this contradicts the fact that $g_1$ contains no element of $I$ as a subword. If $g_1 \in I$, then this contradicts the fact that different words of $I$ do not overlap. Therefore in this case $m \leq 1$.

\item In this case have $g = g_1 g_2$, with $g_1 \notin I$, $g_2 \in I$, and $g_1$ is a proper subword of $k_1$. Since only part of $g_2$ gets simplified when writing $gh$, and $d$ is the largest suffix of $g$ that gets simplified, we can write $t_1 = g_1 t_1'$, where $t_1'$ is a subword of $g_2$. Therefore $k_2, \ldots, k_m$ are all subwords of $t_1'$, and hence of $g_2$. Now if $m \geq 3$, then both $k_2$ and $k_3$ are proper subwords of $g_2$, and either $k_2$ or $k_3$ must belong to $I$, by the same argument as in the previous case. Since they are proper subwords of $g_2 \in I$, we get a contradiction as before. Therefore in this case $m \leq 2$.
\end{enumerate}

Analogously, $n \leq 2$. Finally, we have shown that $|r_1|, |r_2| \leq 2$, $|r_3| \leq 5$, so $D(\Delta_I) \leq 5$.
\end{proof}

\begin{remark}

We have defined this decomposition on independent families in order to have a well-defined way of writing down elements. But looking at the proof of the lemma, we see that the independence property is also used extensively to get a bound on the defect.

\end{remark}

\begin{corollary}
\label{pnoqm}

Let $I$ be a symmetric independent family of words, and let $\alpha$ be a bounded coefficient map supported on $I$. Then $\varphi_\alpha$ is a $\Delta_I$-decomposable quasimorphism. It follows that an element of $\Sigma_{Ind}$ is a finite sum of decomposable quasimorphisms.

\end{corollary}

We end by noting that this decomposition had already been defined for finite independent families, by Hartnick and Schweitzer in their paper \cite{qout}.

\subsubsection{Continuity of $\kappa(\ell^1)$}

The proof that Brooks quasimorphisms are $\Delta$-continuous for any $\Delta$ can be extended to the whole space $\kappa(\ell^1)$. Recall that for $\varphi_\alpha \in \kappa(\ell^1)$, we denote $S \kappa_\alpha = \sum_{n \geq 0} \kappa_\alpha(n)$.

\begin{proposition}
\label{Ccont}

Let $\varphi_\alpha \in\kappa(\ell^1)$. Then $\varphi_\alpha$ is $\Delta$-continuous for any decomposition $\Delta$. Moreover, $S_{\varphi_\alpha, \Delta} \leq 36 S \kappa_\alpha$.
\end{proposition}

\begin{proof}
Let $\Delta$ be a decomposition. Let $N \geq 1$, and let $g, h, g', h'$ be such that $N_{\Delta}((g, h), (g', h')) = N$. We estimate:
$$\delta^1 \varphi_\alpha (g, h) - \delta^1 \varphi_\alpha (g', h') = \sum\limits_{|w| > N} \alpha_w (\delta^1 h_w(g, h) - \delta^1 h_w(g', h')).$$
The words of length not larger than $N$ get simplified because $N_\Delta ((g, h), (g', h')) = N$, just as in the proof of continuity of Brooks quasimorphisms (Proposition \ref{D-Dec-D-cont}).
Let $g = t_1d, h = d^{-1}t_2$ and $gh = t_1t_2$ as reduced expressions. Then $\delta^1 h_w(g, h) \neq 0$ only if $w^{\pm 1}$ is in one of the three compatible families given by $t_1|d, d^{-1}|t_2$ or $t_1|t_2$. Similarly for $\delta^1 h_w(g', h')$, with three other compatible families. Denote by $C_1, \ldots, C_6$ the compatible families mentioned. Then the quantity above can be estimated by
$$\sum\limits_{|w| > N} |\alpha_w| |\delta^1 h_w(g, h) - \delta^1 h_w(g', h')| \leq 6 \sum\limits_{i = 1}^6 \sum_{\substack{|w| > N \\ w^{\pm 1} \in C_i}} |\alpha_w| \leq 36 \kappa_\alpha(N).$$
It follows that we can choose $x_N = 36 \kappa_\alpha(N)$. Then the $\ell^1$ convergence of $\kappa_\alpha(N)$ implies that of $x_N$, which gives $\Delta$-continuity. The value of the corresponding $\ell^1$ sum can be estimated accordingly.
\end{proof}

\subsection{Proof of the main theorem}
\label{proof_main}

We finally have all elements needed for the proof of our main theorem. Recall the definitions of the spaces $\ell^1_{Ind}$, $\kappa(\ell^1)$, $\kappa(c_0)$ from Subsection \ref{subsec_cal}. Moreover, recall that to $\varphi_\alpha \in \ell^1_{Ind}$ we associate the value $SInd_\alpha$, and to $\varphi_\beta \in \kappa(\ell^1)$ we associate the value $S \kappa_\beta$.

\begin{theorem}[Main]
\label{th_main}

Let $\varphi_\alpha, \varphi_\beta \in Cal$. Let $\phi_\lambda$ be a Rolli quasimorphism.

\begin{enumerate}
\item If $\alpha$ and $\beta$ are supported on the same symmetric independent family $I$, then $[\delta^1 \varphi_\alpha] \smile [\delta^1 \varphi_\beta] = 0 \in H^4_b(F)$. The norm of the corresponding primitive is $\mathcal{O}(||\alpha||_\infty \cdot ||\beta||_\infty)$, where the implied constant is universal.

\item If $\varphi_\alpha \in \ell^1_{Ind}$ and $\varphi_\beta \in \kappa(\ell^1)$, then $[\delta^1 \varphi_\alpha] \smile [\delta^1 \varphi_\beta] = 0 \in H^4_b(F)$. The norm of the corresponding primitive is $\mathcal{O}(SInd_\alpha \cdot S \kappa_\beta)$, where the implied constant is universal.

\item If $\varphi_\beta \in \kappa(\ell^1)$, then $[\delta^1 \phi_\lambda] \smile [\delta^1 \varphi_\beta] = 0 \in H^4_b(F)$. The norm of the corresponding primitive is $\mathcal{O}(||\lambda||_\infty \cdot S \kappa_\beta)$, where the implied constant is universal.

\item If $\varphi_\alpha \in \ell^1_{Ind}$ and $\varphi_\beta \in \kappa(c_0)$, then $[\delta^1 \varphi_\alpha] \smile [\delta^1 \varphi_\beta]$ has vanishing Gromov seminorm.

\item If $\varphi_\beta \in \kappa(c_0)$, then $[\delta^1 \phi_\lambda] \smile [\delta^1 \varphi_\beta]$ has vanishing Gromov seminorm.
\end{enumerate}
\end{theorem}

\begin{proof}
1. $\varphi_\alpha$ and $\varphi_\beta$ are both $\Delta_I$-decomposable, and so in particular $\varphi_\beta$ is $\Delta_I$-continuous by Proposition \ref{D-Dec-D-cont}. By Theorem \ref{Heuer_B}, we obtain $[\delta^1 \varphi_\alpha] \smile [\delta^1 \varphi_\beta] = 0$.

We estimate the norm of the corresponding primitive. $D(\varphi_\beta) \leq 3 ||\beta||_\infty D(\Delta_I) \leq 15 ||\beta||_\infty$ by Proposition \ref{D-dec-prop} and Lemma \ref{D_I}; and $S_{\varphi_\beta, \Delta_I} \leq 2 D(\varphi_\beta) \leq 30 ||\beta||_\infty$ by Proposition \ref{D-Dec-D-cont}. By Lemma \ref{lbound}, the norm of the primitive is bounded by a universal constant times:
$$C_{\alpha, \varphi_\beta, \Delta_I} = ||\alpha||_\infty (D(\Delta_I)D(\varphi_\beta) + S_{\varphi_\beta, \Delta_I}) \leq ||\alpha||_\infty(5 \cdot 15 ||\beta||_\infty + 30 ||\beta||_\infty) << ||\alpha||_\infty \cdot ||\beta||_\infty.$$

2. By Proposition \ref{cal}, $D(\varphi_\beta) \leq 3 \kappa_\beta(1)$. Let $\alpha = \sum \alpha_i$, where $\alpha_i$ is supported on the symmetric independent family $I_i$ and $SInd_{\alpha} = \sum ||\alpha_i||_\infty < \infty$ (the value $SInd_\alpha$ is an infimum which a priori is not attained, but one can take the infimum over all possible sequences $(\alpha_i)_{i \geq 1}$ at the end of the calculation). For each $i \geq 1$, consider the quasimorphism $\varphi_{\alpha_i}$ and the decomposition $\Delta_{I_i}$. Then $\varphi_{\alpha_i}$ is $\Delta_{I_i}$-decomposable, and $\varphi_\beta$ is $\Delta_{I_i}$-continuous with $S_{\varphi_\beta, \Delta_{I_i}} \leq 36 S \kappa_\beta$ by Proposition \ref{Ccont}. Therefore $[\delta^1 \varphi_{\alpha_i}] \smile [\delta^1 \varphi_\beta] = 0$ by Theorem \ref{Heuer_B}. By Lemma \ref{lbound}, the primitive $\pi_i$ has norm bounded by a universal constant times:
$$C_{\alpha_i, \varphi_\beta, \Delta_{I_i}} = ||\alpha_i||_\infty (D(\Delta_{I_i}) D(\varphi_\beta) + S_{\varphi_\beta, \Delta_{I_i}}) << ||\alpha_i||_\infty S \kappa_\beta.$$

Now $\varphi_\alpha = \sum \varphi_{\alpha_i}$. This equality makes sense because when evaluated at any $g \in F$ the right-hand side is an $\ell^1$ sum:
$$\sum\limits_{i \geq 1} |\varphi_{\alpha_i}(g)| \leq \sum\limits_{i \geq 1} \sum\limits_{|w| \leq |g|} |(\alpha_i)_w \, h_w(g)| \leq \sum\limits_{i \geq 1} (2|S|)^{|g|} ||\alpha_i||_\infty |g|  \leq (2|S|)^{|g|} |g| SInd_\alpha.$$
Therefore if we define $\pi := \sum \pi_i$, then $\delta^1 \varphi_\alpha \smile \delta^1 \varphi_\beta = \delta^3 \pi$, and this equality makes sense by the same argument. Now
$$||\pi||_\infty \leq \sum\limits_{i \geq 1} ||\pi_i||_\infty << \sum\limits_{i \geq 1} ||\alpha_i||_\infty S \kappa_\beta = S Ind_\alpha \cdot S \kappa_\beta.$$
In particular $\pi$ is bounded, so $[\delta^1 \varphi_\alpha] \smile [\delta^1 \varphi_\beta] = 0$. \\

3. $\phi_\lambda$ is $\Delta_{Rolli}$-decomposable, $\Delta_{Rolli} = 1$ by Example \ref{Dec_ex}, so $D(\phi_\lambda) \leq 3||\lambda||_\infty$  by Proposition \ref{D-dec-prop}. Moreover $\varphi_\beta$ is $\Delta_{Rolli}$-continuous with $S_{\varphi_\beta, \Delta_{Rolli}} \leq 36 S \kappa_\beta$ by Proposition \ref{Ccont}. So $[\delta^1 \phi_\lambda] \smile [\delta^1 \varphi_\beta] = 0$ by Theorem \ref{Heuer_B}. As in the previous item $D(\varphi_\beta) \leq 3 \kappa_\beta(1)$. So the primitive is bounded by a universal constant times
$C_{\lambda, \varphi_\beta, \Delta_{Rolli}} << ||\lambda||_\infty \cdot S \kappa_\beta$. \\

Item 4. follows from item 2. and item 5. follows from item 3., after a direct application of Corollary \ref{cup_cont} (continuity of the cup product) and Lemma \ref{kc0} (defect-density of $\Sigma_{Br}$ in $\kappa(c_0)$).
\end{proof}

\subsection{Self-overlapping words}
\label{subsec_so}

We have focused on non-self-overlapping words, because they are easier to deal with and we were able to prove stronger results. However, something can be said also for self-overlapping words. In all of this subsection, words will not be assumed to be non-self-overlapping.

\begin{lemma}
A big Brooks quasimorphism is $\Delta$-continuous for \underline{any} decomposition $\Delta$, and $S_{H_w, \Delta} \leq 2 |w| D(H_w)$.
\end{lemma}

\begin{proof}
The proof is the same as that of Proposition \ref{D-Dec-D-cont}. By Proposition \ref{bigsmall}, the value of $\delta^1 H_w(g, h)$ only depends on the $|w|$-neighbourhood of the center of the tripod with endpoints $1, g, gh$. Thus if $N_\Delta((g, h), (g', h')) \geq |w|$, then $\delta^1 H_w(g, h) = \delta^1 H_w(g', h')$. So we can define $x_j = 2 D(H_w)$ for $0 \leq j < |w|$, and $x_j = 0$ for $j \geq |w|$. So $S_{H_w, \Delta} \leq 2 |w| D(H_w)$.
\end{proof}

We can even deal with some infinite sums that are the analogue of the space $w \ell^1_{Br}$:

\begin{proposition}
\label{Ccont2}
Let $\varphi := \sum \alpha_w H_w$ be such that $\sum |w| |\alpha_w| D(H_w) < \infty$. Then $\varphi$ is a $\Delta$-continuous quasimorphism for \underline{any} decomposition $\Delta$, with $D(\varphi) \leq \sum |\alpha_w| D(H_w)$ and $S_{\varphi, \Delta} \leq 2 \sum |w| |\alpha_w| D(H_w)$.
\end{proposition}

\begin{proof}
The estimate for the defect follows from the sublinearity of $D$.

Let $\Delta$ be a decomposition and $(x^w_j)_{j \geq 0}$ be the non-negative $\ell^1$ sequence giving $\Delta$-continuity of $H_w$. Let $(g, h) \neq (g', h') \in F^2$. If $N_\Delta((g, h), (g', h')) = N$:
$$|\delta^1 \varphi(g, h) - \delta^1 \varphi(g', h')| \leq \sum\limits_w |\alpha_w| |\delta^1 H_w(g, h) - \delta^1 H_w(g', h')| \leq \sum\limits_w |\alpha_w| |x^w_N| =: x_N.$$
We need to show that $(x_j)_{j \geq 0}$ is $\ell^1$ and calculate the series.
$$\sum\limits_{j \geq 0} x_j = \sum\limits_{j \geq 0} \sum \limits_w |\alpha_w| |x_j^w| = \sum\limits_{j \geq 0} \sum\limits_{|w| > j} |\alpha_w| 2 D(H_w) = $$
$$ = \sum\limits_w \sum\limits_{j < |w|} |\alpha_w| 2 D(H_w) = 2 \sum\limits_w |w| |\alpha_w| D(H_w) < \infty.$$
\end{proof}

\begin{corollary}
Let $\varphi := \sum \alpha_w H_w$ be such that $\sum |w|(|w| - 1) |\alpha_w| < \infty$. Then $\varphi$ is $\Delta$-continuous for any $\Delta$ and $S_{\varphi, \Delta} \leq 6 \sum |w|(|w| - 1) |\alpha_w|$.
\end{corollary}

\begin{proof}
$D(H_w) \leq 3(|w|-1)$ by Proposition \ref{bigsmall}.
\end{proof}

\begin{remark}
For estimates it may be easier to use: $\frac{1}{2} \sum |w|^2 |\alpha_w| \leq \sum |w|(|w| - 1) |\alpha_w| \leq \sum |w|^2 |\alpha_w|$ (if no $w$ appearing in the sum is a generator).
\end{remark}

This gives more vanishing cup products:

\begin{theorem}
\label{cupbig}
Let $\varphi_\alpha \in \ell^1_{Ind}$, $\phi_\lambda$ a Rolli quasimorphism and $\psi := \sum \beta_w H_w$. Let $S := \sum |w| |\beta_w| D(H_w)$ and suppose that $S < \infty$.
\begin{enumerate}
\item $[\delta^1 \varphi_\alpha] \smile [\delta^1 \psi] = 0 \in H^4_b(F)$. The norm of the corresponding primitive is $\mathcal{O}(SInd_\alpha \cdot S)$, where the implied constant is universal.
\item $[\delta^1 \phi_\lambda] \smile [\delta^1 \psi] = 0 \in H^4_b(F)$. The norm of the corresponding primitive is $\mathcal{O}(||\lambda||_\infty \cdot S)$, where the implied constant is universal.
\end{enumerate}
\end{theorem}

\begin{proof}
We have $D(\psi) \leq \sum |\beta_w| D(H_w) < S$ and $S_{\psi, \Delta} \leq 2 S$ by Proposition \ref{Ccont2}. Then the proof is the same as that of items 2. and 3. of Theorem \ref{th_main}, where $\varphi_\beta \in \kappa(\ell^1)$. The only difference is that $S_{\varphi_\beta, \Delta} << S \kappa_\beta$ is replaced by $S_{\psi, \Delta} << S$, and $D(\varphi_\beta) << S \kappa_\beta$ is replaced by $D(\psi) << S$.
\end{proof}

\begin{corollary}
\label{cupbig2}
Let $\varphi_\alpha, \phi_\lambda$ be as in Theorem \ref{cupbig}, and let $w$ be any word. Then:
\begin{enumerate}
\item $[\delta^1 \varphi_\alpha] \smile [\delta^1 H_w] = 0 \in H^4_b(F)$. The norm of the corresponding primitive is $\mathcal{O}(SInd_\alpha \cdot D(H_w) |w|))$, where the implied constant is universal.
\item $[\delta^1 \phi_\lambda] \smile [\delta^1 H_w] = 0 \in H^4_b(F)$. The norm of the corresponding primitive is $\mathcal{O}(||\lambda||_\infty \cdot D(H_w) |w|))$, where the implied constant is universal.
\end{enumerate}
\end{corollary}

This last corollary is a consequence of Theorem \ref{th_main} if $w$ is non-self-overlapping, but it tells us something new if $w$ is self-overlapping. In the case in which $\varphi_\alpha \in \Sigma_{Br}$, it follows from the vanishing theorem of Bucher and Monod (Theorem \ref{intro_BM}), but for a general $\varphi_\alpha \in \ell^1_{Ind}$, as well as for Rolli quasimorphisms, it is also new. Nonetheless, the vanishing of the cup product of two general big Brooks quasimorphisms does not follow from any of our results, and is still exclusive to Theorem \ref{intro_BM}. \\

We end by pointing out that these results do not carry over to general small Brooks quasimorphisms. Indeed, we have seen in Example \ref{small_Brooks} that given a cyclically reduced word $w$, even though the tripod with endpoints $1, w^{2n}, w^{4n}$ and the one with endpoints $1, w^{2n+1}, w^{4n+2}$ have overlaps of size $2n|w|$ around both branches (the middle one being empty since the product is reduced), the values of $\delta^1 h_{w^2}(w^{2n}, w^{2n})$ and $\delta^1 h_{w^2}(w^{2n+1}, w^{2n+1})$ differ by 1. Therefore $h_{w^2}$ cannot be $\Delta$-continuous for all $\Delta$. For instance, it is not $\Delta_{triv}$-continuous, it is not $\Delta_{Rolli}$-continuous unless $w$ is a power of a generator, and it is not $\Delta_v$-continuous if $v$ is a non-self-overlapping word that is a subword of some power of $w$.

\pagebreak

\section{Free products of decompositions}
\label{secprod}

We will start by introducing a decomposition on free products of free groups, and use it to define the notion of free products of quasimorphisms, which was introduced by Rolli in \cite{Rolli} as a generalization of Rolli quasimorphisms. We then introduce a notion of free product of decompositions, and describe the corresponding decomposable and continuous quasimorphisms in terms of free products of quasimorphisms. Finally, we study the map $\iota : F \to F*F$, state an open question, and prove a theorem that depends on it. \\

Throughout, $F_1, \ldots, F_n$ are non-abelian free groups of finite rank with fixed bases $S_1, \ldots, S_n$, and $*F := F_1 * \cdots * F_n$ is their free product, which is a free group with basis $S_1 \cup \cdots \cup S_n$. We denote by $\iota_i : F_i \to *F$ the canonical inclusion. If the index is clear from the context, we will occasionally identify an element of $F_i$ with its image in $*F$.

\subsection{A decomposition on free products}

Recall that the decomposition $\Delta_{Rolli}$ works by grouping together all consecutive occurrences of the same generator in a reduced word. We can generalize this to free products of free groups. Indeed, given an element $1 \neq g \in *F$, we can write it uniquely as $\iota_{i_1}(g_1) \iota_{i_2}(g_2) \cdots \iota_{i_k}(g_k)$, where $i_j \neq i_{j+1}$ and $g_j \neq 1$. We define $\Delta_*(g) := (\iota_{i_1}(g_1), \ldots, \iota_{i_k}(g_k))$.

\begin{lemma}
\label{D_*}

$\Delta_*$ is a decomposition with pieces $\iota_1(F_1 \, \backslash \, \{ 1 \}) \cup \cdots \cup \iota_n(F_n \, \backslash \, \{ 1 \})$ and defect 1.
\end{lemma}

\begin{proof}
Conditions 1., 2. and 3. are clear. For condition 4., write $\Delta_*(g) = c_1^{-1} \cdot r_1 \cdot c_2$, $\Delta_*(h) = c_2^{-1} \cdot r_2 \cdot c_3$ and $\Delta_*(gh) = c_1^{-1} \cdot r_3^{-1} \cdot c_3$, just as in Definition \ref{Dec}. If either one of $r_1$ or $r_2$ is empty, then $r_1 = r_2 = r_3 = \emptyset$, and we are done. Otherwise, the last element $x$ of $r_1$ and the first element $y$ of $r_2$ belong to the same $F_i$, and they cannot cancel out (since otherwise they would be in $c_2$), i.e., $x \neq y^{-1}$. It follows that $r_1 = (x)$, $r_2 = (y)$ and $r_3^{-1} = (xy)$ all have length 1.
\end{proof}

\begin{definition}
Given $g, h \in *F$, if the $r$-part of the $\Delta_*$-triangle is empty we write $r_*(g, h) = (1, 1)$; else $r_*(g, h) = (x, y)$, where $x$ and $y$ are as in the proof. We also denote $i_*(g, h) \in \{ 1, \ldots, n \}$ for the unique index such that $x, y \in F_{i_*(g, h)}$, and give it an arbitrary value if $x = y = 1$ (it will make no difference in all instances in which this notation is used).
\end{definition}

\begin{example}
If $F_i =  \mathbb{Z}$ for every $i$, then $\Delta_* = \Delta_{Rolli}$ on $*F$.
\end{example}

This already yields quasimorphisms of free products that generalize Rolli quasimorphisms: the $\Delta_*$-decomposable quasimorphisms. However, we can generalize this construction further, which is what we will do next.

\subsection{Free products of quasimorphisms}

\begin{definition}
For all $1 \leq i \leq n$, let $\varphi_i \in Q_{alt}(F_i)$. Then we define $\varphi_1 * \cdots * \varphi_n : *F \to \mathbb{R}$ as follows: if $\Delta_*(g) = (\iota_{i_1}(g_1), \ldots, \iota_{i_k}(g_k))$, then $\varphi_1 * \cdots * \varphi_n (g) = \sum \varphi_{i_j}(g_j)$.
\end{definition}

This definition resembles that of a $\Delta_*$-decomposable quasimorphism, but here we are allowing quasimorphisms on pieces, and not just bounded functions. Luckily this behaves well:

\begin{proposition}[Rolli]
\label{prod_qm}

Let $\varphi := \varphi_1 * \cdots * \varphi_n$ be as above. Then $\varphi$ is an alternating quasimorphism, and $D(\varphi) \leq \max (D(\varphi_1), \ldots, D(\varphi_n))$. Moreover, the following formula holds:
$$\delta^1 \varphi (g, h) = \delta^1 \varphi_{i_*(g, h)}(r_*(g, h)).$$
\end{proposition}

\begin{proof}

As in the proof of Lemma \ref{D_*}, write $\Delta_*(g) = c_1^{-1} \cdot (x) \cdot c_2$, $\Delta_*(h) = c_2^{-1} \cdot (y) \cdot c_3$, $\Delta_*(gh) = c_1 \cdot (xy) \cdot c_3$, with the convention that $(1)$ is the empty sequence. Then:
$$\delta^1 \varphi(g, h) = (-\varphi(\underline{c_1}) + \varphi_{i_*(g, h)}(x) + \varphi(\underline{c_2})) + (-\varphi(\underline{c_2}) + \varphi_{i_*(g, h)}(y) + \varphi(\underline{c_3})) + $$
$$- (-\varphi(\underline{c_1}) - \varphi_{i_*(g, h)}(xy) + \varphi(\underline{c_3})) = \delta^1 \varphi_{i_*(g, h)}(x, y).$$
\end{proof}

This proof (which purposely uses a different language than Rolli) resembles very much that of Proposition \ref{D-dec-prop}. This hints to the fact that in the next subsection we will define a decomposition for which the quasimorphism $\varphi$ is decomposable.

\subsection{Free products of decompositions}

Suppose now that for each $1 \leq i \leq n$ we have a decomposition $\Delta_i$ on $F_i$ with pieces $\mathcal{P}_i$. Given $g \in *F$ with $\Delta_*(g) = (\iota_{i_1}(g_1), \ldots, \iota_{i_k}(g_k))$, define $\Delta_1 * \cdots * \Delta_n$ by $ \Delta_1 * \cdots * \Delta_n (g) := \iota_{i_1}(\Delta_{i_1}(g_1)) \cdots \iota_{i_k}(\Delta_{i_k}(g_k))$. That is, we first apply $\Delta_*$, and then apply the suitable $\Delta_i$ to each element of the sequence. We call $\Delta_1 * \cdots * \Delta_n$ the \textit{free product} of the $\Delta_i$. For the rest of this subsection, we simply write $\Delta := \Delta_1 * \cdots * \Delta_n$.

\begin{lemma}
With the notation above, $\Delta$ is a decomposition on $*F$ with pieces $\iota_1(\mathcal{P}_1) \cup \cdots \cup \iota_n (\mathcal{P}_n)$ and defect $\max (D(\Delta_1), \ldots, D(\Delta_n))$. Moreover the $r$-part of the $\Delta$-triangle for $g, h$ is empty if $r_*(g, h) = (1, 1)$, and otherwise it corresponds to the $r$-part of the $\Delta_{i_*(g, h)}$-triangle for $r_*(g, h)$.
\end{lemma}

\begin{proof}
Conditions 1., 2. and 3. are clear. For condition 4., let $g, h \in F$. It follows by the definition that the $\Delta$-triangle for $g, h$ is a refinement of the $\Delta_*$-triangle for $g, h$. In particular the $c$-part of the $\Delta_*$-triangle will stay, up to refinement, in the $c$-part of the $\Delta$-triangle. So if the $r$-part of the $\Delta_*$-triangle is empty, i.e., if $r_*(g, h) = (1, 1)$, then so is the $r$-part of the $\Delta$-triangle and we are done. Else, let $r_*(g, h) = (x, y) \neq (1, 1)$. Then the $r$-part $(x) \cup (y) \cup (xy)$ gets refined into the $\Delta_{i_*(g, h)}$-triangle for $x, y$. Therefore the $r$-part of the $\Delta$-triangle for $g, h$ is the $r$-part of the $\Delta_{i_*(g, h)}$-triangle for $x, y$. The result follows.
\end{proof}

\begin{example}
For each $i$, let $\Delta_i := \Delta_b$. Then the free product of the $\Delta_i$ is $\Delta_*$.
\end{example}

Now that we have a new decomposition, we want to identify decomposable and continuous quasimorphisms. It follows pretty much directly from the definitions that:

\begin{lemma}
A quasimorphism $\varphi$ is $\Delta$-decomposable if and only if there exist $\Delta_i$-decomposable quasimorphisms $\varphi_i$ such that $\varphi = \varphi_1 * \cdots * \varphi_n$.
\end{lemma}

For continuous quasimorphisms, we are not able to prove an "if and only if" statement, but we will still be able to show the more interesting implication. The first step is to understand the function $N_\Delta$.

\begin{lemma}
Let $g, h, g', h' \in *F$. If $i_*(g, h) \neq i_*(g', h')$, then $N_\Delta((g, h), (g', h')) = 0$. Else, if moreover $r_*(g, h) \neq r_*(g', h')$, then:
$$N_\Delta((g, h), (g', h')) = N_{\Delta_{i_*(g, h)}}(r_*(g, h), r_*(g', h')).$$
\end{lemma}

\begin{proof}
This just follows from the fact that the $r$-part of the $\Delta$-triangle for $g, h$ is the $r$-part of the $\Delta_{i_*(g, h)}$-triangle for $r_*(g, h)$.
\end{proof}

Now suppose that $\varphi_i \in Q_{alt}(F_i)$ is $\Delta_i$-continuous, and let $\varphi := \varphi_1 * \cdots * \varphi_n$. For all $1 \leq i \leq n$, let $(x^i_j)_{j \geq 0}$ be the non-negative $\ell^1$ sequence from the definition of $\Delta_i$-continuity of $\varphi_i$. Define $x_0 = 2 D(\varphi)$ and $x_j = \max (x^1_j, \ldots, x^n_j)$ for $j \geq 1$, and notice that $(x_j)_{j \geq 0}$ is again a non-negative $\ell^1$ sequence.

\begin{lemma}
Using the sequence $(x_j)_{j \geq 0}$, $\varphi$ is $\Delta$-continuous.
\end{lemma}

\begin{proof}
Let $g, h, g', h' \in *F$ and $N := N_\Delta((g, h), (g', h'))$. Assume that $i_*(g, h) = i_*(g', h')$. If $r_*(g, h) = r_*(g', h')$, then $\delta^1 \varphi(g, h) = \delta^1 \varphi(g', h')$ by Proposition \ref{prod_qm}. Else, $N = N_{\Delta_{i_*(g, h)}}(r_*(g, h), r_*(g', h'))$, and so either $N = 0$ or:
$$|\delta^1 \varphi(g, h) - \delta^1 \varphi(g', h')| = |\delta^1 \varphi_{i_*(g, h)}(r_*(g, h)) - \delta^1 \varphi_{i_*(g', h')}(r_*(g', h'))| \leq x^{i_*(g, h)}_N \leq x_N.$$
If instead $i_*(g, h) \neq i_*(g', h')$, then $N = 0$ and the bound is trivial.
\end{proof}

Then a direct application of Heuer's Theorem (Theorem \ref{Heuer_B}) yields:

\begin{theorem}
\label{cup_free_prod}

Let $\varphi_i \in Q(F_i)$ be $\Delta_i$-decomposable, and $\psi_i \in Q(F_i)$ be $\Delta_i$-continuous. Let $\varphi := \varphi_1 * \cdots * \varphi_n$ and $\psi := \psi_1 * \cdots * \psi_n$. Then $[\delta^1 \varphi] \smile [\delta^1 \psi] = 0 \in H^4_b(*F)$.
\end{theorem}

\subsection{The map $\iota$}

We would now like to apply the last theorem to obtain more trivial cup products in the bounded cohomology of a free group $F$. The map $\iota$, that we are about to define, will allow us to make this connection. \\

Let $F$ be a non-abelian free group of finite rank with basis $S$, and consider the free product $F * F$, with basis $\iota_1(S) \cup \iota_2(S)$, for the canonical inclusions $\iota_1, \iota_2 : F \to F*F$.

\begin{definition}
Define $\iota : F \to F*F : g \mapsto \iota_1(g) \iota_2(g)$.
\end{definition}

This map allows us to make a relation between free products and sums of quasimorphisms:

\begin{lemma}
\label{iota_sum}

Let $\varphi_1, \varphi_2 \in Q_{alt}(F)$. Then:
$$(\varphi_1 * \varphi_2) \circ \iota = (\overline{\varphi_1 * \varphi_2}) \circ \iota = \varphi_1 + \varphi_2.$$

\end{lemma}

\begin{proof}

$$(\varphi_1 * \varphi_2)(\iota(g)) = (\varphi_1 * \varphi_2)(\iota_1(g)\iota_2(g)) = \varphi_1(g) + \varphi_2(g).$$

$$(\overline{\varphi_1 * \varphi_2}) (\iota(g)) = \lim\limits_{n \to \infty} \frac{1}{n} (\varphi_1 * \varphi_2) (\iota(g)^n) = \lim\limits_{n \to \infty} \frac{1}{n} (\varphi_1 * \varphi_2) (\iota_1(g) \iota_2(g) \cdots \iota_1(g) \iota_2(g)) = $$
$$ = \lim\limits_{n \to \infty} \frac{1}{n} \left( n\varphi_1(g) + n \varphi_2(g) \right) = \varphi_1(g) + \varphi_2(g).$$
\end{proof}

In particular, $(\varphi_1 * \varphi_2) \circ \iota$ is a quasimorphism. In fact, this is true in a much more general context: the map $\iota$ enjoys a special property which was defined and studied by Hartnick and Schweitzer in the paper \cite{qout}.

\begin{definition}
Let $f : G \to H$ be any map between groups. We say that $f$ is a \textit{quasihomomorphism} if $\varphi \circ f \in Q(G)$ for any $\varphi \in Q(H)$.
\end{definition}

\begin{lemma}
The map $\iota : F \to F*F$ is a quasihomomorphism.
\end{lemma}

\begin{proof}
Using the identity $xy = yx[x^{-1}, y^{-1}]$, for suitable commutators $c, c'$ we obtain:
$$\iota(gh) = \iota_1(g) \iota_1(h) \iota_2(g) \iota_2(h) = \iota_1(g) \iota_2(g) \iota_1(h) \cdot c \cdot \iota_2(h) = \iota(g) \iota(h) \cdot cc'.$$
Now let $\varphi \in Q(F*F)$. By Corollary \ref{dcomm}, we have $|\varphi(c)|, |\varphi(c')| \leq 4 D(\varphi)$. Therefore:
$$|\varphi(\iota(g)) + \varphi(\iota(h)) - \varphi(\iota(gh))| \leq |\varphi(\iota(g)) + \varphi(\iota(h)) - \varphi(\iota(g)\iota(h))| + $$
$$|\varphi(\iota(g) \iota(h)) + \varphi(cc') - \varphi(\iota(g) \iota(h) cc')| + |\varphi(cc')| \leq $$
$$ \leq 2D(\varphi) + |\varphi(cc') - \varphi(c) - \varphi(c')| + |\varphi(c)| + |\varphi(c')| \leq 11 D(\varphi).$$
\end{proof}

\begin{remark}
We have provided a direct proof, but this result is an immediate consequence of the additive structure of quasihomomorphisms. Indeed, by \cite[Corollary 3.9]{qout}, given quasihomomorphisms $f_1, f_2 : G \to H$, their product $f_1f_2 : G \to H : g \mapsto f_1(g) f_2(g)$ is also a quasihomomorphism. Since $\iota_1, \iota_2$ are homomorphisms, it follows that $\iota$ is a quasihomomorphism.
\end{remark}

\begin{example}
\label{pb_iota}

1. Let $\lambda, \mu \in \ell^\infty_{alt}(\mathbb{Z})^n$, so that $\lambda \times \mu \in \ell^\infty_{alt}(\mathbb{Z})^{2n}$. Consider the Rolli quasimorphism $\phi_{\lambda \times \mu} \in Q(F * F)$. Then $\phi_{\lambda \times \mu} \circ \iota = \phi_{\lambda + \mu} \in Q(F)$.

Therefore, $\iota$ pulls back Rolli quasimorphisms to Rolli quasimorphisms. \\

2. Let $w \in F$ be a non-self-overlapping word. Then $h_{\iota_1(w)} \circ \iota = h_{\iota_2(w)} \circ \iota = h_w \in Q(F)$. If instead $w \in F*F$ is a non-self-overlapping word that is not completely contained in one of the two copies of $F$, then $h_w \circ \iota$ is bounded. Indeed, if $w^{\pm 1}$ occurs as a subword of $\iota(g) = \iota_1(g) \iota_2(g)$, it must necessarily overlap the juncture of $\iota_1(g)|\iota_2(g)$, in order to include elements from both copies of $F$; so $|h_w \circ \iota| \leq 1$.

Therefore, $\iota$ pulls back Brooks quasimorphisms on a non-self-overlapping word to either Brooks quasimorphisms on a non-self-overlapping word or bounded maps. \\

3. Now let $w \in F*F$ be a self-overlapping word. The previous reasoning applies also to the case of $H_w$. Indeed, the only thing to notice is that any word in $F * F$ overlapping the juncture of $\iota_1(g) | \iota_2(g)$ is necessarily non-self-overlapping.

Therefore $\iota$ pulls back big Brooks quasimorphisms to either big Brooks quasimorphisms or bounded maps.
\end{example}

Moreover, this quasihomomorphism between free groups is not \textit{quasi-Ulam}, a condition satisfied by all quasihomomorphisms from free groups presented in \cite{qout}.

\begin{definition}
A map $f : F \to G$, where $G$ is any group, is called \textit{quasi-Ulam} if there exists a finite set $E \subset G$ such that:
\begin{enumerate}
\item If $g, h \in F$ are such that $gh$ is reduced, then $f(gh) \in E f(g) E f(h) E.$
\item $f(g^{-1}) \in Ef(g)^{-1}E$ for all $g \in F$.
\end{enumerate}
\end{definition}

By \cite[Proposition 5.3]{qout}, these are in fact quasihomomorphisms.

\begin{lemma}
$\iota$ is not quasi-Ulam. In fact, it does not satisfy either of the two conditions from the definition.
\end{lemma}

\begin{proof}
1. Let $E$ be a finite set. Pick a cyclically reduced word $g$ (so $gg$ is reduced) such that $|g| > \max_{e \in E} |e|$. We need to show that $\iota(g^2) \notin E \iota(g) E \iota(g) E$, or equivalently that
$$1 \notin \iota_1(g)^{-2} E \iota_1(g) \iota_2(g) E \iota_1(g) \iota_2(g) E \iota_2(g)^{-2}.$$
By the choice of $g$, no element of $E$ can cancel a whole $\iota_1(g)$ or $\iota_2(g)$, which means that for any $e \in E$, $\iota_1(g)^{-2} e \iota_1(g)$ does not simplify completely. In particular, the right-most element will still be in the first copy of $F$, and so it does not simplify with $\iota_2(g)$. Continuing this way, we see that an element on the right-hand-side set cannot be trivial. \\

2. We can show in a similar way that for $|g|$ large enough, $1 \notin \iota_1(g) E \iota_2(g)^{-1} \iota_1(g)^{-1} E \iota_2(g)$.
\end{proof}

At this point the natural question is: how does $\iota$ behave with respect to cup products?

\begin{question}
\label{qcupall}
Suppose that $\varphi, \psi \in Q(F*F)$ satisfy $[\delta^1 \varphi] \smile [\delta^1 \psi] = 0 \in H^4_b(F*F)$. Is it true that $[\delta^1(\varphi \circ \iota)] \smile [\delta^1 (\psi \circ \iota)] = 0 \in H^4_b(F)$?
\end{question}

\begin{example}
By Example \ref{pb_iota}, the question has a positive answer if $\varphi, \psi$ are quasimorphisms which are either a big Brooks quasimorphism or a Rolli quasimorphism. Indeed:

1. If $\varphi, \psi$ are Rolli quasimorphisms, this follows from Heuer's Theorem (Corollary \ref{Heuer_A}).

2. If $\varphi$ is a Rolli quasimorphism and $\psi$ is a big Brooks quasimorphism, this follows from Corollary \ref{cupbig2}.

3. If $\varphi, \psi$ are big Brooks quasimorphisms, this follows from Bucher and Monod's Theorem (Theorem \ref{intro_BM}).
\end{example}

This looks like a simple question, however we have not been able to prove it by a direct argument. The relevance of this question lies in the following result:

\begin{theorem}
\label{cupall}

Suppose that Question \ref{qcupall} is answered affirmatively, and let $\Delta$ be a decomposition. Then if $\varphi$ is $\Delta$-decomposable and $\psi \in Q(F)$ is \underline{any} other quasimorphism, we have $[\delta^1 \varphi] \smile [\delta^1 \psi] = 0 \in H^4_b(F)$.
\end{theorem}

\begin{proof}
Up to bounded distance, we may assume that $\psi$ is alternating. Recall from Example \ref{Db_cont} that any alternating quasimorphism is $\Delta_b$-continuous. Let $\underline{0} \in Q(F)$ be the zero map, which is a $\Delta$-continuous quasimorphism, as well as a $\Delta_b$-decomposable one. It follows that $\varphi * \underline{0}$ is $\Delta * \Delta_b$-decomposable and $\underline{0} * \psi$ is $\Delta * \Delta_b$-continuous. By Theorem \ref{cup_free_prod}, we have $[\delta^1 (\varphi * \underline{0})] \smile [\delta^1 (\underline{0} * \psi)] = 0$. By assumption, this implies that $[\delta^1 (\varphi * \underline{0}) \circ \iota] \smile [\delta^1 (\underline{0} * \psi) \circ \iota] = 0$. Finally, by Lemma \ref{iota_sum}, we conclude that $[\delta^1 \varphi] \smile [\delta^1 \psi] = 0$.
\end{proof}

\begin{corollary}
Suppose that Question \ref{qcupall} is answered affirmatively. Let $\varphi \in \Sigma_{Ind}$, $\phi$ a Rolli quasimorphism and $\psi$ any other quasimorphism. Then $[\delta^1 \varphi] \smile [\delta^1 \psi] = [\delta^1 \phi] \smile [\delta^1 \psi] = 0 \in H^4_b(F)$.
\end{corollary}

This result would be much stronger than everything we have managed to prove in this thesis. In particular it would be the first triviality result concerning the cup product of some non-trivial quasimorphism with the whole of $H^2_b(F)$.

\pagebreak

\section{Open questions}
\label{secop}

To start with the obvious, the general question of triviality of the cup product remains open:

\begin{question}
\label{q1}

Is the cup product $\smile : H^2_b(F) \times H^2_b(F) \to H^4_b(F)$ trivial?
\end{question}

A weaker question is the existence of a non-zero element in $H^2_b(F)$, whose cup product with everything else vanishes:

\begin{question}
Is there a non-trivial quasimorphism $\varphi \in Q(F)$ such that for any other quasimorphism $\psi \in Q(F)$ we have $[\delta^1 \varphi] \smile [\delta^1 \psi] = 0$?
\end{question}

By Theorem \ref{cupall}, this is implied by a result relating the map $\iota$ and the cup product:

\begin{question}
Suppose that $\varphi, \psi \in Q(F*F)$ satisfy $[\delta^1 \varphi] \smile [\delta^1 \psi] = 0 \in H^4_b(F*F)$. Is it true that $[\delta^1(\varphi \circ \iota)] \smile [\delta^1 (\psi \circ \iota)] = 0 \in H^4_b(F)$?
\end{question}

We have seen that $\iota$ is a quasihomomorphism that is not quasi-Ulam. Since the latter have already been studied, and many examples are available in the paper \cite{qout}, one could ask the same question for them:

\begin{question}
Let $f : F_1 \to F_2$ be a quasi-Ulam quasihomomorphism between free groups. Suppose that $\varphi, \psi \in Q(F_2)$ satisfy $[\delta^1 \varphi] \smile [\delta^1 \psi] = 0 \in H^4_b(F_2)$. Is it true that $[\delta^1(\varphi \circ f)] \smile [\delta^1 (\psi \circ f)] = 0 \in H^4_b(F_1)$?
\end{question}

At the end of Subsection \ref{subsec_so}, we mentioned that our results do not carry over to general small Brooks quasimorphisms. This begs the question:

\begin{question}
Is there any vanishing result for cup products involving general small Brooks quasimorphisms on self-overlapping words?
\end{question}

We have been able to prove triviality of many cup products, but for even more we could prove vanishing of the Gromov seminorm. This motivates the following weaker version of Question \ref{q1}:

\begin{question}
Does the image of the cup product $\smile : H^2_b(F) \times H^2_b(F) \to H^4_b(F)$ lie in the zero-norm subspace of $H^4_b(F)$?
\end{question}

Other open questions concern the largeness of the space $Cal$ with respect to the space $\mathcal{N}_{Br}$, and the largeness of the latter with respect to $H^2_b(F)$. More precisely:

\begin{question}
Is every $\varphi_\alpha \in \mathcal{N}_{Br}$ equivalent to some $\varphi_\beta \in Cal$?
\end{question}

\begin{question}
Does every class in $H^2_b(F)$ admit a representative in $\mathcal{N}_{Br}$?
\end{question}

The previous question would follow from the converse of Proposition \ref{twhom}:

\begin{question}
\label{q_fund}

Let $\alpha$ be an alternating coefficient map supported on $\mathcal{N}$, and let $\varphi_\alpha = \sum_{w \in \mathcal{N^+}} \alpha_w h_w$ and $\tilde{\varphi}_\alpha = \sum_{w \in \mathcal{N}^+} \alpha_w \overline{h_w}$. Suppose that $\tilde{\varphi}_\alpha$ is a quasimorphism. Is $\varphi_\alpha$ a quasimorphism? What if we restrict the support of $\alpha$ to a (well-chosen) fundamental set?
\end{question}

In the quest of getting from $Cal$ and its subspaces to the whole of $H^2_b(F)$, it may be of some help to understand the role of Rolli quasimorphisms: these are quasimorphisms which give trivial cup products and which we understand well, but whose definition is not in terms of infinite sums of Brooks quasimorphisms.

\begin{question}
Are all Rolli quasimorphisms equivalent to quasimorphisms in $Cal$ or $\mathcal{N}_{Br}$? If not, how large is the sum of $Cal$ or $\mathcal{N}_{Br}$ with the space of Rolli quasimorphisms with respect to $H^2_b(F)$?
\end{question}

In order to keep matters simple, in this thesis we have always fixed a basis $S$ of $F$, and given the definition of decomposition in terms of that basis (this is relevant for condition 1. in the definition). One can remove this convention and consider, instead of decompositions, pairs $(\Delta, S)$, where $S$ is a basis of $F$ and $\Delta$ a decomposition of $F$ with basis $S$. Then there is a natural action of $Aut(F)$ on the space of such pairs. Let $\Delta$ be a decomposition on $F$ with basis $S$ and pieces $\mathcal{P}$, and let $T \in Aut(F)$. Define $T(\Delta, S) = (T(\Delta), T(S))$ with pieces $T(\mathcal{P})$, where if $\Delta(g) = (g_1, \ldots, g_k)$, we define $T(\Delta)(T(g)) := (T(g_1), \ldots, T(g_k))$. This is compatible with the action of $Aut(F)$ on $Q(F)$ via $\varphi \mapsto \varphi \circ T^{-1}$, meaning that $T$ sends $\Delta$-decomposable/continuous quasimorphisms to $T(\Delta)$-decomposable/continuous quasimorphisms. The last open question we want to mention, although a little vague, is:

\begin{question}
What can one say about the dynamics $Aut(F)$ on the space of pairs $(\Delta, S)$? How does it relate to the dynamics of $Aut(F)$ on $H^2_b(F)$, as studied in \cite{out}?
\end{question}

\pagebreak

\appendix

\section{Overlap graphs}
\label{og}

In this section we introduce some graphs associated with subsets of a free group of finite rank $F$, which describe the way words overlap. Certain properties of these graphs will correspond to subspaces of $Cal$, for instance $\Sigma_{Ind}$. Thanks to some graph-theoretic arguments, we are able to give new examples and non-examples of elements of $\Sigma_{Ind}$. \\

We would like to thank Mathilde Bouvel and Matija Bucic for the helpful discussions.

\subsection{Graph-theoretic preliminaries}

Let $D$ be a directed graph. We denote by $U(D)$ the underlying undirected graph of $D$. The maximal length of a directed path is denoted by $lp(D) \in \mathbb{N} \cup \{ \infty \}$. \\

Let $G = (V(G), E(G))$ be an undirected graph. A \textit{clique} in $G$ is a complete induced subgraph. The maximal size of a clique is called the \textit{clique number} of $G$ and is denoted by $\omega(G) \in \mathbb{N} \cup \{ \infty \}$.

A \textit{$k$-coloring} of $G$ is a map $c : V(G) \to C$, where $C$ is a set of cardinality $k$ whose elements are called \textit{colors}. It is \textit{proper} if $c(v) \neq c(w)$ whenever $\{v, w\} \in E(G)$. The minimal $k$ such that $G$ is properly $k$-colorable is called the \textit{chromatic number} of $G$ and is denoted by $\chi(G) \in \mathbb{N} \cup \{ \infty \}$.

For a directed graph $D$, we denote $\omega(D) := \omega(U(D))$ and $\chi(D) := \chi(U(D))$. \\

The chromatic and clique number are related by the obvious inequality $\omega(G) \leq \chi(G)$. In general, no other relation holds: we will see at the end of this subsection that there exist graphs with finite clique number and infinite chromatic number. Classes of graphs for which some inequality in the other direction holds are called \textit{$\chi$-bounded}. More precisely, a class of finite graphs is called $\chi$-bounded if there exists a function $f : \mathbb{N} \to \mathbb{N}$ such that $\chi(G) \leq f(\omega(G))$ for any graph $G$ in the class. Many interesting classes of graphs are $\chi$-bounded: see \cite{chi} for a survey on this topic. \\

Since we will be working with infinite graphs, it is useful to reduce arguments about $\omega$ and $\chi$ to finitary ones. On the one hand, it is clear that
$$\omega(G) = \sup \{ \omega(H) : H \subset G \text{ finite} \}.$$
For the chromatic number, this is more subtle \cite{erdos}:

\begin{theorem}[De Brujin, Erdos]
Let $G$ be a graph, then
$$\chi(G) = \sup \{ \chi(H) : H \subset G \text{ finite} \}.$$
\end{theorem}

Concerning colorings of directed graphs, we can gather some information from the additional structure \cite[Theorem 8.4.1]{digraphs}:

\begin{theorem}[Gallai-Roy-Vitaver]
\label{GRV}

Let $D$ be a directed graph. Then $\chi(D) \leq lp(D) + 1$.
\end{theorem}

A directed graph $D = (V(D), E(D))$ is \textit{transitive} if $(u, w) \in E(D)$ whenever $(u, v), (v, w) \in E(D)$.

A \textit{tournament} is an orientation of a complete undirected graph. In other words, it is a directed graph such that for any pair of distinct vertices $v, w$, exactly one of $(v, w), (w, v)$ is an edge. \\

Let $D$ be a directed graph. An \textit{arc coloring} of $D$ is a map $c : E(D) \to C$. It is \textit{proper} if $c(u, v) \neq c(v, w)$ for any pair of edges $(u, v), (v, w)$. We define the \textit{arc chromatic} number of $D$ to be the minimal $k$ such that there exists a proper arc coloring of $D$ with $k$ colors.

The \textit{line graph} of $D$ is the directed graph $L(D)$ with vertex set $E(D)$ and an edge from $(u, v)$ to $(v, w)$ for any pair of edges $(u, v), (v, w)$ in $D$. Therefore a proper arc coloring of $D$ corresponds to a proper coloring of $L(D)$ and viceversa. Thus, the arc chromatic number of $D$ is equal to the chromatic number of $L(D)$. \\

The line graph construction allows us to give an example of a graph with finite clique number and infinite chromatic number. Let $T_n$ be the transitive tournament on $n$ vertices, that is, the graph with vertex set $\{1, \ldots, n\}$ and an edge from $i$ to $j$ if and only if $i < j$. Let $T_\infty$ be the infinite tournament, that is, the graph with vertex set $\mathbb{N}$ and adjacency defined the same way. First, it is easy to check that $\omega(L(T_n)) = 3$ for $n \geq 3$; so $\omega(L(T_\infty)) = 3$ as well. By \cite{arc_col}, the arc chromatic number of $T_n$ is asymptotically equivalent to $\log_2 n$, and therefore $\chi(L(T_n)) \to \infty$. It follows that $\chi(L(T_\infty)) = \infty$. We will use this example in the last subsection. \\

We end by recalling one of the most important theorems in graph theory, originally proven in 1930 \cite{ramsey}:

\begin{theorem}[Ramsey]
For each $n \geq 1$ there exists an integer $R(n)$ with the following property. Let $K$ be a complete graph with at least $R(n)$ vertices, and color edges of $K$ red or blue arbitrarily. Then $K$ contains a monochromatic clique of size at least $n$.
\end{theorem}

\subsection{First definitions}

In the following, $F$ will denote a free group of finite rank with a fixed basis $S$, and $V$ a symmetric subset of $F$ consisting of non-self-overlapping words. \\

Recall the definition of subwords and (proper) overlaps from subsection \ref{s_cow}. Two words can overlap properly in different ways, and in the sequel we will talk about overlaps of minimal length. Here is a characterization of those.

\begin{lemma}
Suppose that $w$ overlaps $w'$ properly. Then the overlap of minimal length is the unique non-self-overlapping one.
\end{lemma}

\begin{proof}
Suppose that $w$ overlaps $w'$ at $x$ and $y$, with $|x| > |y|$. Since $x$ and $y$ are suffixes of $w$, $y$ is a suffix of $x$. Since $x$ and $y$ are prefixes of $w'$, $y$ is a prefix of $x$. Therefore $x$ has a proper suffix that is also a proper prefix.

Conversely, if $w$ overlaps $w'$ at $x$, and $y$ is a proper suffix of $x$ that is also a proper prefix of $x$, then $y$ is a suffix of $w$ and a prefix of $w'$, so $w$ overlaps $w'$ at $y$ and $|y| < |x|$.
\end{proof}

\begin{definition}
Let $V$ be as above. We define $OG(V)$ to be the directed graph with vertex set $V$ and an edge from $w$ to $w'$ if $w$ overlaps $w'$ properly.

We define $SG(V)$ to be the directed graph with vertex set $V$ and an edge from $w$ to $w'$ if $w$ is a subword of $w'$.

Finally, we define $OSG(V)$ to be the directed graph with vertex set $V$ and edge set $E(OG(V)) \cup E(SG(V))$. In other words, it is the directed graph with vertex set $V$ and an edge from $w$ to $w'$ if and only if $w$ overlaps $w'$ (properly or not).
\end{definition}

Since all words in $V$ are assumed to be non-self-overlapping, these graphs have no loops. Notice that $SG(V)$ is a transitive directed graph. \\

The direction of the edges behaves nicely with respect to inversion. More precisely, $w$ overlaps properly $w'$ if and only if $(w')^{-1}$ overlaps properly $w^{-1}$. On the other hand, $w$ is a subword of $w'$ if and only if $w^{-1}$ is a subword of $(w')^{-1}$. Therefore the map $V \to V : w \mapsto w^{-1}$ induces an isomorphism of the underlying undirected graphs of $OG(V)$ and $SG(V)$, which in the first case flips the orientation and in the second case preserves it.

\begin{definition}
We define $\overline{OG}(V)$ to be the undirected graph obtained from $OG(V)$ by identifying $w$ and $w^{-1}$. We similarly define $\overline{SG}(V)$, which is still a directed graph, and $\overline{OSG}(V)$, which is undirected.
\end{definition}

Since $w$ does not overlap $w^{-1}$, these graphs still have no loops. \\

We will be interested in certain properties of these graphs, the clique and chromatic number, which are quantities that can be defined for any graph. We will moreover consider a third quantity that is specific to these graphs. Recall the definition of a compatible family from subsection \ref{s_cow}.

\begin{definition}
Let $V$ be as above. We define $\kappa(V) \in \mathbb{N} \cup \{\infty\}$ to be the maximal cardinality of a compatible family contained in $V$.
\end{definition}

Since all elements of a compatible family overlap (although not necessarily properly), we have $\kappa(V) \leq \omega(OSG(V))$.

\subsection{Subspaces of $Cal$}

We now discuss the relation with subspaces of $Cal$. Given $\varphi_\alpha \in Cal$, we can look at the overlap graphs defined above for $V = supp(\alpha)$. Recall that $\varphi_\alpha \in Cal$ means that
$$\sup\limits_C \left( \sum\limits_{w \in C} |\alpha_w| \right) < \infty;$$
where the supremum runs over all compatible families in $V$. Since $\alpha$ is necessarily bounded by Corollary \ref{a_bdd}, if compatible families in $V$ have bounded cardinality, then $\varphi_\alpha \in Cal$ automatically. Conversely, this is the most general condition on $V$ that can ensure that $\varphi_\alpha \in Cal$, without having to take into account the values of $\alpha$.

\begin{definition}
We define $CFin$ to be the space of $\varphi_\alpha \in Cal$ such that $\kappa(supp(\alpha)) < \infty$.
\end{definition}

Restricting to such quasimorphisms allows us to forget the values of $\alpha$ and only focus on $V$, allowing for combinatorial arguments.

\begin{lemma}
Let $\varphi_\alpha \in Cal$ and let $V := supp(\alpha)$. Then $\varphi_\alpha \in \Sigma_{Ind}$ if and only if $\chi(\overline{OSG}(V)) < \infty$.
\end{lemma}

\begin{proof}
A symmetric family $I$ is independent if and only if the graph $OSG(I)$ is independent (i.e., there are no edges), if and only if the graph $\overline{OSG}(I)$ is independent. A proper coloring of $\overline{OSG}(V)$ is a partition into independent sets, which therefore corresponds to a partition of $V$ into symmetric independent families.
\end{proof}

Since we have triviality results for cup products involving elements in $\Sigma_{Ind}$, this motivates us to find criteria that ensure that $\chi(\overline{OSG}(V)) < \infty$, under the minimal assumption that $\kappa(V) < \infty$. \\

We are not able to prove it, but we believe that the assumption $\kappa(V) < \infty$ is not very restrictive. More precisely, we believe that the following question has a positive answer:

\begin{question}
Is $CFin$ dense in $Cal$ with respect to the defect topology?
\end{question}

Here is some motivation as to why this should be true. Let $\varphi := \sum \alpha_w h_w \in Cal$. Define
$$\varphi_n := \sum\limits_{|\alpha_w| \geq \frac{1}{n}} \alpha_w h_w.$$
If $C$ is a compatible family contained in $supp(\alpha) \cap \{ w : |\alpha_w| \geq \frac{1}{n} \}$, then
$$\kappa_\alpha(1) \geq \sum\limits_{w \in C} |\alpha_w| \geq |C|/n;$$
which gives a bound on $|C|$ in terms of $n$. Therefore $\varphi_n \in CFin$. Looking at the difference $(\varphi - \varphi_n)$, we see that this equals $\varphi_{\alpha_n}$, where $\alpha_n(w) = \alpha_w \cdot \mathbbm{1} \{|\alpha_w| < \frac{1}{n}\}$. This gives:
$$D(\varphi - \varphi_n) \leq 3 \kappa_{\alpha_n}(1) = 3 \sup\limits_{C \subset V} \left( \sum_{\substack{w \in C \\ |\alpha_w| < 1/n}} |\alpha_w| \right).$$

\begin{question}
Does the quantity above converge to 0 as $n \to \infty$?
\end{question}

If this were true, $CFin$ would be defect-dense in $Cal$, so a triviality result for the cup product with elements of $CFin$ would yield cup products with all of $Cal$ with vanishing Gromov seminorm.

\subsection{Reduction to proper overlaps}

We recall the general question raised in the previous subsection: under what conditions does $\kappa(V) < \infty$ imply $\chi(\overline{OSG}(V)) < \infty$? Here we show that it is enough to look at proper overlaps:

\begin{proposition}
Suppose that $\kappa(V) < \infty$. Then $\chi(\overline{OSG}(V)) < \infty$ if and only if $\chi(\overline{OG}(V)) < \infty$.
\end{proposition}

\begin{proof}
By definition, $\overline{OSG}(V) = \overline{OG}(V) \cup \overline{SG}(V)$. A proper $k$-coloring of $\overline{OSG}(V)$ restricts to a proper $k$-coloring of $\overline{OG}(V)$, which proves that $\chi(\overline{OSG}(V)) < \infty \Rightarrow \chi(\overline{OG}(V)) < \infty$.

On the other hand, given a proper $k$-coloring $c : V / \pm \to C$ of $\overline{OG}(V)$ and a proper $l$-coloring $d : V / \pm \to D$ of $\overline{SG}(V)$, we can define a proper $kl$-coloring $c \times d : V / \pm \to C \times D : v^{\pm 1} \mapsto (c(v^{\pm 1}), d(v^{\pm 1}))$. Therefore in order to prove the other implication it suffices to show that $\kappa(V) < \infty \Rightarrow \chi(\overline{SG}(V)) < \infty$ \\

So suppose that $\kappa(V) < \infty$. Consider a sequence $w_1, w_2, \ldots, w_n$ of words in $V$, where $w_i$ is a proper subword of $w_{i+1}$. Up to removing the first two terms, we may assume that $|w_1| \geq 2$. Then it is easy to see that $\{ w_1, \ldots, w_n \}$ forms a compatible family. By hypothesis, such sequences are therefore of bounded length. But such sequences are precisely the directed paths in $SG(V)$, whence $lp(SG(V)) < \infty$.

Now consider a directed path $w_1^{\pm 1} \to \cdots \to w_n^{\pm 1}$ in $\overline{SG}(V)$. This means that $w_1$ is a subword of $w_2^{\varepsilon_2}$, which is a subword of $w_3^{\varepsilon_3}$, and so on, for some exponents $\varepsilon_i \in \{ \pm 1 \}$. Therefore this path lifts to a directed path in $SG(V)$. It follows that $lp(\overline{SG}(V)) < \infty$ as well. By the Gallai-Roy-Vitaver Theorem (Theorem \ref{GRV}), we conclude that $\chi(\overline{SG}(V)) < \infty$.
\end{proof}

The proof also shows:

\begin{corollary}
Suppose that $\chi(\overline{OG}(V)) < \infty$, and that there is a bound on sequences of nested subwords in $V$. Then $\chi(\overline{OSG}(V)) < \infty$.
\end{corollary}

Now that we only care about proper overlaps, we are able to give a simple criterion for $\chi(\overline{OSG}(V)) < \infty$:

\begin{lemma}
Suppose that $\kappa(V) < \infty$, or that there is a bound on the length of sequences of nested subwords in $V$. Suppose moreover there exists some $N \geq 1$ such that for any $w, w' \in V$ that overlap properly, the minimal overlap has size at most $N$. Then $\chi(\overline{OSG}(V)) < \infty$.
\end{lemma}

\begin{proof}
It suffices to show that $\chi(\overline{OG}(V)) < \infty$. Up to removing finitely many vertices, we may assume that all words in $V$ have length larger than $N$. Consider the coloring $c : V \to B_N^2 : w \mapsto (p_N(w), s_N(w))$, where $B_N$ is the ball of radius $N$ in $F$, $p_N(w)$ is the prefix, and $s_N(w)$ the suffix of length $N$ of $w$. Make this into a coloring $\overline{c}$ of $V/\pm$ by choosing a set $V^+$ and defining $\overline{c}(v^{\pm 1}) := c(v)$, for $v \in V^+$. Now suppose that $v^{\pm 1}$ and $w^{\pm 1}$ have the same color, where $v, w \in V^+$. This means that $v$ and $w$ have the same prefix and suffix of length $N$. Since minimal overlaps have size at most $N$, this implies that a proper overlap of $v$ with $w^{\pm 1}$ yields a proper overlap of $v$ with $v^{\pm 1}$. But $v$ is non-self-overlapping, and no word can overlap its inverse. Therefore $v^{\pm 1}$ and $w^{\pm 1}$ cannot be connected by an edge.
\end{proof}

\begin{example}

Let $I$ be a symmetric independent family, and $N \geq 1$. Define $V$ to be the set of non-self-overlapping words that can be written as $xwy$, where $|x|, |y| \leq N$ and $w \in I$. We will show that $\chi(\overline{OSG}(V)) < \infty$.

First, if $x_1w_1y_1, \ldots, x_nw_ny_n$ is a sequence of nested subwords, and $n$ is large enough in terms of $N$, then necessarily some $w_i$ must be a subword of some $w_j$, which is impossible since $I$ is independent. Secondly, if $w_1, w_2$ are long enough with respect to $N$, and $x_1w_1y_1$ and $x_2w_2y_2$ have an overlap of size larger than $N$, then $w_1$ and $w_2$ overlap, which is again impossible.

This example yields infinitely many families giving rise to elements of $\Sigma_{Ind}$, parametrized by an independent family $I$ and an integer $N$.
\end{example}

\begin{example}
\label{ex_SInd}

Let $V$ be such that $\chi(\overline{OSG}(V)) < \infty$, and $N \geq 1$. Define $W$ to be the set of non-self-overlapping words that can be written as $xwy$, where $|x|, |y| \leq N$ and $w \in V$. Then $\chi(\overline{OSG}(W)) < \infty$. This follows from the previous example by writing $V$ as a finite disjoint union of symmetric independent families.
\end{example}

\begin{example}
The condition is sufficient, but not necessary. Indeed, let $V = A \cup B$, where $A = \{ab^nc : n \geq 1\}^{\pm 1}$ and $B = \{b^ncd : n \geq 1\}^{\pm 1}$. Both $A$ and $B$ are symmetric and independent, so $\chi(\overline{OSG}(V)) = 2$. However, the only overlap between $ab^nc$ and $b^ncd$ is at $b^nc$, so minimal overlaps in $V$ have arbitrarily large size.
\end{example}

\subsection{Other (non-)implications}

We have already mentioned that $\chi(\overline{OSG}(V)) < \infty$ implies $\kappa(V) < \infty$. Let us be more precise:

\begin{lemma}
\label{impl}

Suppose that there is a bound on the length sequences of nested subwords in $V$. Then the following implications hold:

\begin{center}
	\begin{tikzcd}
		& & \kappa(V) < \infty \arrow[d, Leftarrow] \\
		& & \omega(OG(V)) < \infty \arrow[dl, Leftarrow] \arrow[dr, Leftarrow] \\
		& \chi(OG(V)) < \infty \arrow[dr, Leftarrow] & & \omega(\overline{OG}(V)) < \infty \arrow[ld, Leftarrow] \\
		& & \chi(\overline{OG}(V)) < \infty
	\end{tikzcd}
\end{center}

\end{lemma}

\begin{proof}
Since a word never overlaps its inverse, a proper $k$-coloring of $\overline{OG}(V)$ lifts to a proper $k$-coloring of $OG(V)$. Therefore $\chi(OG(V)) \leq \chi(\overline{OG}(V))$.

Two implications follow from the general fact that for any graph $G$, it holds $\omega(G) \leq \chi(G)$.

A clique in $OG(V)$ is mapped injectively to a clique in $\overline{OG}(V)$, since a word and its inverse cannot belong to the same clique. Therefore $\omega(OG(V)) \leq \omega(\overline{OG}(V))$. \\

We are left to show that $\omega(OG(V)) < \infty \Rightarrow \kappa(V) < \infty$. First, recall that $\kappa(V) \leq \omega(OSG(V))$. The hypothesis implies that $\omega(SG(V)) < \infty$. So we need to show that $\omega(SG(V)) < \infty, \omega(OG(V)) < \infty \Rightarrow \omega(OSG(V)) < \infty$.

To see this, consider a clique in $OSG(V)$, and color the edges red if they come from $OG(V)$, and blue if they come from $SG(V)$. If an edge comes from both graphs, choose arbitrarily one of the two colors. By Ramsey's Theorem, if the clique we started with is large enough, it contains an arbitrarily large monochromatic clique. Therefore if $\omega(OSG(V)) = \infty$, then at least one of $\omega(SG(V))$ and $\omega(OG(V))$ must be infinite.
\end{proof}

It is tempting to try and reverse these arrows. Here we give an example of $V \subset F_3$ such that $\omega(OG(V)) = \omega(\overline{OG}(V)) = 3$ but $\chi(OG(V)) = \chi(\overline{OG}(V)) = \infty$. Moreover, the words in $V$ will be Lyndon words, which shows that even restricting our attention to those $V$ which are contained in a fundamental set does not allow to reverse this arrow (see subsection \ref{s_cow} for the definitions of Lyndon words and fundamental sets).

\begin{example}
Consider $F = F_3 = \langle S \rangle$, where $S = \{a, b, c\}$, and order the generating set by $\{ a \prec b \prec c \}$. Extend this arbitrarily to a total order on $S^{\pm 1}$, and then to the lexicographic order on $F$. Define $V^+ := \{ (ab^nc)(ab^mc) : n > m \geq 1 \}$. Each word in $V^+$ is easily seen to be Lyndon, so in particular non-self-overlapping by Lemma \ref{lynd_nso}. Let $V := V^+ \cup (V^+)^{-1}$.

We start by considering the graph $OG(V^+)$ ($OG$ is defined in the obvious way for a non-symmetric set). We claim that this is isomorphic to the graph $L(T_\infty)$ from the first subsection. Indeed, an element of $V^+$ corresponds to a pair $(n, m)$ where $n > m \geq 1$, and there is an edge from $(n, m)$ to $(n', m')$ if and only if $m = n'$. It follows that $\omega(OG(V^+)) = 3$ and $\chi(OG(V^+)) = \infty$.

Finally, we notice that there is no overlap between elements of $V^+$ and elements of $(V^+)^{-1}$, so $OG(V)$ is a disjoint union of $OG(V^+)$, and a copy of $OG(V^+)$ with the opposite orientation. It follows that $\omega(OG(V)) = 3$ and $\chi(OG(V)) = \infty$, too. \\

This last observation allows us to deduce the same for the quotient graphs. Indeed, the inversion map on $OG(V)$ switches the two copies of $OG(V^+)$, changing only the orientation. Since this is not taken into account in the quotient, which is an undirected graph, we conclude that $\overline{OG}(V)$ is isomorphich to $U(OG(V^+))$. Therefore, once again, $\omega(\overline{OG}(V)) = 3$ and $\chi(\overline{OG}(V)) = \infty$.
\end{example}

\begin{example}
\label{ex_l1Ind}

Let $V$ be as above, and let $\alpha$ be any bounded symmetric coefficient map such that $V = supp(\alpha)$. Then the previous example shows that $\varphi_\alpha \in CFin \subset Cal$, but $\varphi_\alpha \notin \Sigma_{Ind}$. In fact, suppose further that that $\inf_{w \in supp(\alpha)} |\alpha_w| > 0$. Then $\varphi_\alpha \notin \ell^1_{Ind}$ either: if $\alpha$ were an $\ell^1$ sum of $\alpha_i$, each supported on a symmetric independent family, then that sum would have to be finite.
\end{example}

\subsection{Open question}

In this appendix we have proposed a graph-theoretic approach to study the space $CFin$. We have justified its utility by conjecturing the following:

\begin{question}
Is $CFin$ defect-dense in $Cal$?
\end{question}

The other questions relate to the reversibility of arrows from Lemma \ref{impl}. Namely:

\begin{question}
Suppose that $\kappa(V) < \infty$. Is it true that $\omega(OG(V)) < \infty$?
\end{question}

A Ramsey-theoretic approach seems promising. Indeed, a clique in $OG(V)$ is a tournament, and a classical application of Ramsey's Theorem implies that a large enough tournament contains an arbitrarily large transitive tournament. Is it possible to find a large compatible family inside a large transitive tournament in $OG(V)$?

\begin{question}
Suppose that $\omega(OG(V)) < \infty$. Is it true that $\omega(\overline{OG}(V)) < \infty$?
\end{question}

Once again, a Ramsey-theoretic approach seems promising. Indeed, a direct application of Ramsey's Theorem allows to reduce to showing that cliques of the form $\{w_1^{\pm 1}, \ldots, w_n^{\pm 1} \}$, where $w_i$ overlaps $w_j^{-1}$ for all $1 \leq i \neq j \leq n$, cannot exist.

\begin{question}
Suppose that $\chi(OG(V)) < \infty$. Is it true that $\chi(\overline{OG}(V)) < \infty$?
\end{question}

Consider a proper coloring $c : V \to C$ of $OG(V)$, and choose a set $V^+$. Define $\overline{c} : V/\pm \to C \times C : v^{\pm 1} \mapsto (c(v), c(v^{-1}))$, where $v \in V^+$. Then it is enough to show that any color class is colorable with finitely many colors. Since each color class corresponds to a symmetric subset of $V$ that is two-colorable, the hypothesis of the previous question can thus be strengthened to: $OG(V)$ is bipartite, and inversion flips the two parts.

\begin{question}
Are there some additional hypothesis on $V$ that can ensure $\chi$-boundedness of $OG(V)$?
\end{question}

We have seen that in general these overlap graphs are not $\chi$-bounded, but some sufficient condition would be desirable. \\

Finally, in light of Question \ref{q_fund}, answers to the questions above are of interest even when $V$ is restricted to a well-chosen fundamental set. We have seen that, for the last question, our construction of fundamental sets (using Lyndon words) does not help. It is however possible that some other construction of a fundamental set could.

\pagebreak

\addcontentsline{toc}{section}{References}

\bibliographystyle{alpha}
\bibliography{References}

\end{document}